\documentclass[11pt, usletter, reqno, twoside, makeidx]{amsart}
\usepackage[margin=2.8cm, marginpar=3.3cm]{geometry}

\usepackage{amssymb}
\usepackage{braket}
\usepackage{mathrsfs}
\usepackage{ifthen}
\usepackage[dvipdfmx]{graphicx}

\usepackage{tikz}
\usetikzlibrary{matrix,arrows,decorations.pathmorphing}
\usetikzlibrary{positioning}
\usepackage{comment} 
\usepackage{mleftright}
\usepackage{hyperref}
\usepackage{cleveref}
 \usepackage[all]{xy}
 \usepackage{amscd}
 \usepackage{amsrefs}
 \usepackage{color}
 \usepackage{enumitem}
\newlist{steps}{enumerate}{1}
\setlist[steps, 1]{label = Step \arabic*:}

\newcommand{\blue}[1]{\textcolor{blue}{#1}}

\theoremstyle{plain}
\newtheorem{thm}{Theorem}[section]
\crefname{thm}{Theorem}{Theorems}
\Crefname{thm}{Theorem}{Theorems}
\newtheorem{prop}[thm]{Proposition}
\crefname{prop}{Proposition}{Propositions}
\Crefname{prop}{Proposition}{Propositions}
\newtheorem{lem}[thm]{Lemma}
\crefname{lem}{Lemma}{Lemmas}
\Crefname{lem}{Lemma}{Lemmas}
\newtheorem{cor}[thm]{Corollary}
\crefname{cor}{Corollary}{Corollaries}
\Crefname{cor}{Corollary}{Corollaries}

\crefname{claim}{Claim}{Claims}
\Crefname{claim}{Claim}{Claims}

\crefname{property}{Property}{Properties}
\Crefname{property}{Property}{Properties}

\crefname{problem}{Problem}{Problems}
\Crefname{problem}{Problem}{Problems}
\newtheorem{ques}[thm]{Question}
\crefname{ques}{Question}{Questions}
\Crefname{ques}{Question}{Questions}

\theoremstyle{definition}
\newtheorem{defn}[thm]{Definition}
\crefname{defn}{Definition}{Definitions}
\Crefname{defn}{Definition}{Definitions}

\crefname{notation}{Notation}{Notations}
\Crefname{notation}{Notation}{Notations}

\crefname{convention}{Convention}{Conventions}
\Crefname{convention}{Convention}{Conventions}

\crefname{cond}{Condition}{Conditions}
\Crefname{cond}{Condition}{Conditions}

\crefname{assum}{Assumption}{Assumptions}
\Crefname{assum}{Assumption}{Assumptions}

\crefname{conj}{Conjecture}{Conjectures}
\Crefname{conj}{Conjecture}{Conjectures}

\newtheorem{rem}[thm]{Remark}
\crefname{rem}{Remark}{Remarks}
\Crefname{rem}{Remark}{Remarks}
\newtheorem{ex}[thm]{Example}
\crefname{ex}{Example}{Examples}
\Crefname{ex}{Example}{Examples}

\crefname{section}{Section}{Sections}
\Crefname{section}{Section}{Sections}
\crefname{subsection}{Subsection}{Subsections}
\Crefname{subsection}{Subsection}{Subsections}
\crefname{figure}{Figure}{Figures}
\Crefname{figure}{Figure}{Figures}

\newcommand{\Z}{\mathbb{Z}}

\newcommand{\Q}{\mathbb{Q}}

\newcommand{\sign}{\mathrm{sign}}

\newcommand{\End}{\mathop{\mathrm{End}}\nolimits}

\newcommand{\id}{\mathrm{id}}
\newcommand{\ind}{\mathop{\mathrm{ind}}\nolimits}

\newcommand{\A}{\mathcal A}
\newcommand{\B}{\mathcal B}
\newcommand{\G}{\mathcal G}

\newcommand{\su}{\mathfrak{su}(2)}

\newcommand{\s}{\mathfrak{s}}

\def\ri{\rightarrow}

\def\Om{\Omega}

\def\spinc{\text{spin}^c}
\newcommand{\N}{\mathbb N}
\newcommand{\R}{\mathbb R}

\newcommand{\ctext}[1]{\raise0.2ex\hbox{\textcircled{\scriptsize{#1}}}}

\def\dim{\operatorname{dim}}

\def\aut{\operatorname{Aut}}
\def\End{\operatorname{End}}

\def\id{\operatorname{id}}

\def\ind{\operatorname{ind}}

\def\grad{\operatorname{grad}}

\newcommand{\mbar}[1]{{\ooalign{\hfil#1\hfil\crcr\raise.167ex\hbox{--}}}}

\def\Tr{\operatorname{Tr}}

\def\wt{\widetilde}

\def\un{\underline}
\def\ov{\overline}

\def\CF {\mathit{CF}}
\def\HF {\mathit{HF}}

\newcommand\HFp {\HF^+}

\newcommand \CFm {\CF^-}
\newcommand \HFm {\HF^-}

\newcommand\CP{\smash{\mathbb{CP}^2}}
\newcommand\oCP{\smash{\ov{\mathbb{CP}}}^2}

\newcommand\RP{\smash{\mathbb{RP}^2}}

\newcommand\K{\mathfrak{K}}
\newcommand\E{\mathfrak{E}}

\title[]{Involutions and the Chern-Simons filtration in instanton Floer homology}
\author{Antonio Alfieri}
\address{Centre de Recherche Math\' ematiques, Montreal}
\email{alfieriantonio90@gmail.com}
\author{Irving Dai}
\address{Department of Mathematics, The University of Texas at Austin}
\email{irving.dai@math.utexas.edu}
\author{Abhishek Mallick}
\address{Department of Mathematics, Rutgers University - New Brunswick}
\email{abhishek.mallick@rutgers.edu}
\author{Masaki Taniguchi}
\address{Department of Mathematics, Graduate School of Science, Kyoto University, Kitashirakawa Oiwake-cho, Sakyo-ku, Kyoto 606-8502, Japan}
\email{taniguchi.masaki.7m@kyoto-u.ac.jp}

\begin{document}

\maketitle
\vspace{-0.5cm}

\begin{abstract}
Building on the work of Nozaki, Sato and Taniguchi, we develop an instanton-theoretic invariant aimed at studying strong corks and equivariant bounding. Our construction utilizes the Chern-Simons filtration and is qualitatively different from previous Floer-theoretic methods used to address these questions. As an application, we give an example of a cork whose boundary involution does not extend over any 4-manifold $X$ with $H_1(X, \Z_2) = 0$ and $b_2(X) \leq 1 $, and a strong cork which survives stabilization by either of $n\CP$ or $n\oCP$. We also prove that every nontrivial linear combination of $1/n$-surgeries on the strongly invertible knot $\ov{9}_{46}$ constitutes a strong cork. Although Yang-Mills theory has been used to study corks via the Donaldson invariant, this is the first instance where the critical values of the Chern-Simons functional have been utilized to produce such examples. Finally, we discuss the geography question for nonorientable surfaces in the case of extremal normal Euler number. 
\end{abstract}


\section{Introduction} \label{sec:1}

Let $Y$ be an integer homology $3$-sphere equipped with an involution $\tau$. If $Y$ bounds a smooth $4$-manifold $W$, then it is natural to ask whether $\tau$ extends as a diffeomorphism over $W$. By work of Akbulut \cite{Ak91_cork} and Akbulut-Ruberman \cite{AR16}, this question is closely related to the study of exotic phenomena in four dimensions. Indeed, in the case that $W$ is contractible, a theorem of Freedman \cite{Fr82} shows that $\tau$ always extends as a homeomorphism. This gives the notion of a \textit{cork} \cite{Ak91_cork}, which is known to capture the difference between any pair of exotic structures on the same simply-connected, closed $4$-manifold \cite{MAt, CFHS}. Investigating the extendability of $\tau$ has also led to new sliceness obstructions via branched coverings \cite{ASA20, DHM20} and (in particular) a recent proof that the $(2, 1)$-cable of the figure-eight is not slice \cite{dai20222}.

Several authors have studied the extension question through the lens of Heegaard Floer homology, monopole Floer homology, and Seiberg-Witten theory for families; see for example \cite{LRS18, ASA20, DHM20, dai20222, Kang2022, KMT23A}. These developments have led to a wide range of striking applications to exotica and concordance. The aim of the present work is to introduce new tools for obstructing the extendability of $\tau$ through the use of the Chern-Simons filtration from instanton Floer homology. As we will see, this will allow us to establish different results than were previously accessible via the aforementioned methods. 

In this paper, we present applications to equivariant bounding, corks and exotica, surgeries on symmetric knots, and nonorientable slice surfaces. Each of these is discussed below, along with topological consequences and motivations from Floer theory. To the best of the authors’ knowledge, this is the first instance in which the Chern-Simons filtration has been utilized to address questions of this nature. Indeed, while Yang-Mills theory has been heavily utilized in smooth 4-manifold topology, this is the first time that the critical \textit{values} of the Chern-Simons functional have had any bearing on corks or exotic phenomena. Our examples are obtained through a combined analysis of both the Chern-Simons filtration and an understanding of the behavior of the Donaldson invariant under certain cork twists.

\subsection{Statement of results}\label{sec:1.1}

Recent results of Daemi \cite{Da20} and Nozaki-Sato-Taniguchi \cite{NST19} have suggested a fundamental difference between the information contained in the Chern-Simons filtration and that of other Floer homologies. These applications include several surprising results regarding bordism and Dehn surgery which were previously unobtainable via Heegaard-Floer-theoretic methods. The present paper aims to develop similar techniques in the equivariant setting, which will have an additional connection to the theory of corks and equivariant bordism. 

Our main technical construction is an involutive refinement of the $r_s$-invariant from the work of Nozaki-Sato-Taniguchi \cite{NST19}:

\begin{thm}\label{thm:1.1}
Let $Y$ be an oriented integer homology $3$-sphere and $\tau$ be a smooth, orientation-preserving involution on $Y$. For any $s \in [-\infty, 0]$, we define a real number 
\[
r_s(Y, \tau) \in (0, \infty]
\]
which is an invariant of the diffeomorphism class of $(Y, \tau)$. Moreover, let $(W, \wt{\tau})$ be an equivariant negative-definite cobordism from $(Y, \tau)$ to $(Y', \tau')$ with $H_1(X, \Z_2 )=0$. Then
\[
r_s(Y,\tau) \leq r_s(Y', \tau'). 
\]
If $r_s(Y, \tau)$ is finite and $W$ is simply connected, then in fact
\[
r_s(Y,\tau) < r_s(Y', \tau').
\]
\end{thm}
\noindent
We refer to a cobordism $W$ from $(Y, \tau)$ to $(Y', \tau')$ as \textit{equivariant} if it is equipped with a self-diffeomorphism $\wt{\tau}$ which restricts to $\tau$ and $\tau'$ on $Y$ and $Y'$, respectively. The assumption that $\tau$ is an involution is not essential; our $r_s$-invariant can be defined for any orientation-preserving diffeomorphism on $Y$.


As an initial application, recall the definition of a \textit{strong cork} from the work of Lin-Ruberman-Saveliev \cite{LRS18}. This is a cork for which the boundary involution $\tau$ does not extend over \textit{any} homology ball $W$ that $Y$ bounds. Techniques for detecting strong corks via Heegaard Floer theory were developed by Dai-Hedden-Mallick in \cite{DHM20}; these led to many novel families of corks, some of which have recently been used in the construction of new closed $4$-manifold exotica \cite{LLP}. It follows from Theorem~\ref{thm:1.1} that if $\tau$ extends over some homology ball $W$ with $Y = \partial W$, then $r_s(Y, \tau) = \infty$ for all $s$, and hence $r_s(Y, \tau)$ can be used for strong cork detection. As we will see, the additional information of the Chern-Simons filtration will allow us to derive new results and examples that were previously out of reach using other Floer-theoretic techniques.

\subsubsection{Equivariant bounding}\label{sec:1.1.1}
An obvious extension of the notion of a strong cork is the problem of obstructing \textit{equivariant bounding}. Indeed, it is natural to ask whether there is a cork for which $\tau$ does not extend over any definite manifold (of either sign). This is surprisingly difficult to answer, since even in the nonequivariant case, there are few examples in the literature where Floer theory has been used to provide constraints on both positive- and negative-definite boundings of a homology sphere $Y$.\footnote{For \textit{rational} homology spheres, the fact that there are multiple $d$-invariants makes this problem much more approachable; see \cite{GL20}. However, for an \textit{integer} homology sphere, note that any straightforward approach using the $d$- or $h$-invariant necessarily fails.} Indeed, the first example of an integer homology sphere with no definite bounding was exhibited recently by Nozaki-Sato-Taniguchi \cite{NST19} using the Chern-Simons filtration.\footnote{We generally use ``definite bounding" to refer to a definite bounding $W$ with $H_1(W, \Z_2) = 0$.}

We provide a partial answer to this question by placing homological constraints on action of the extension $\wt{\tau}$. We say $\wt{\tau}$ is \textit{homology-fixing} if $\wt{\tau}_* = \id$ on $H_2(W, \Q)$ and \textit{homology-reversing} if $\wt{\tau}_* = -\id$ on $H_2(W, \Q)$. Combining the involutive $r_s$-invariant with techniques of \cite{DHM20}, we prove:

\begin{thm}\label{thm:1.2}
There exists a cork $Y = \partial W$ such that the boundary involution $\tau$:
\begin{enumerate}
\item Does not extend as a diffeomorphism over any negative-definite $4$-manifold $W^-$ with $H_1(W^-, \Z_2)=0$ bounded by $Y$; and,
\item Does not extend as a homology-fixing or homology-reversing diffeomorphism over any positive-definite $4$-manifold $W^+$ with $H_1(W^+, \Z_2)=0$ bounded by $Y$.
\end{enumerate}
\end{thm}

Note that all previous results regarding equivariant bounding have either been restricted to manifolds which are spin or concern definite manifolds of a fixed sign. The first part of Theorem~\ref{thm:1.2} is established using the methods of this paper, while the second part is a consequence of the Heegaard Floer-theoretic formalism developed in \cite{DHM20}. (This leads to the conditions in the second part of the theorem.) In principle, the involutive $r_s$-invariant is capable of establishing \textit{both} the negative- and positive-definite cases of Theorem~\ref{thm:1.2} without any restrictions on the action of $\widetilde{\tau}_*$; this would follow from finding a cork with $r_s(Y, \tau)$ and $r_s(-Y, \tau)$ both nontrivial. However, we presently lack the computational tools to exhibit such an example. 

Theorem~\ref{thm:1.2} trivially provides the following corollary: 

\begin{cor}\label{cor:1.3}
There exists a cork $Y = \partial W$ such that the boundary involution $\tau$ does not extend as a diffeomorphism over any $4$-manifold $X$ with $H_1(X, \Z_2) = 0$ and $b_2 (X) \leq 1$ bounded by $Y$.
\end{cor}
\noindent
Note that if $X$ is simply connected, then $\tau$ extends over $X$ as a homeomorphism by work of Freedman \cite{Fr82}. Corollary~\ref{cor:1.3} thus emphasizes the difference between smooth and continuous topology in the context of the extension problem for definite manifolds.

We are also able to strengthen Corollary~\ref{cor:1.3} in the case that $X$ is obtained via stabilizing a homology ball by $n \CP$ or $n \oCP$. The following should be thought of as establishing the existence of a strong cork which persists under such a stabilization:

\begin{cor}\label{cor:cp2}
There exists a cork $Y = \partial W$ such that the boundary involution $\tau$ does not extend as a diffeomorphism over $W \# n \CP$ or $W \# n \oCP$ for any $n$; or, more generally, over $W' \# n \CP$ or $W' \# n \oCP$ for any homology ball $W'$ which $Y$ bounds.
\end{cor}
\noindent
Corollary~\ref{cor:cp2} is somewhat similar in spirit to recent stabilization results regarding corks \cite{Kang2022}, except that we deal with stabilization by $n\CP$ and $n\oCP$ rather than spin manifolds such as $S^2 \times S^2$. Utilizing work of Akbulut-Ruberman \cite{AR16}, Corollary~\ref{cor:cp2} be applied to produce pairs of compact, contractible manifolds which remain absolutely exotic even after summing with either of $n\CP$ and $n\oCP$. However, this latter application can also be obtained in a straightforward manner using more standard techniques. The authors thank Anubhav Mukherjee and Kouichi Yasui for bringing this to their attention; see Remark~\ref{rem:7.4} and the techniques of \cite{Y19}.

\subsubsection{Surgeries on slice knots}\label{sec:1.1.2}
We also use our involutive $r_s$-invariant to establish various examples previously inaccessible via Heegaard Floer theory. In \cite{DHM20}, it was shown that if $K$ is a symmetric slice knot, then $1/n$-surgery on $K$ often constitutes a strong cork. This approach should be contrasted with previous methods for constructing corks in the literature, which emphasize the handle decomposition of the relevant $4$-manifold $W$. The formalism of \cite{DHM20} has led to many novel examples of corks, including surgeries on the stevedore. From the viewpoint of homology cobordism, it is natural to consider linear combinations of such examples, and especially linear combinations of $1/n$-surgeries on the same knot $K$.

Again, this turns out to be surprisingly difficult. Traditionally, Heegaard Floer homology has had difficulty distinguishing between different surgeries on the same knot up to homology cobordism. For example, the surgery formula of \cite{hendricks2020surgery} shows that if $K$ is a fixed knot, then (involutive) Heegaard Floer homology cannot be used to establish the linear independence (in $\Theta^3_\Z)$ of any infinite family of $1/n$-surgeries on $K$; see \cite[Proposition 22.9]{hendricks2020surgery}. This is in contrast to Yang-Mills theory, whose application to the linear independence of families such as $\smash{\{S^3_{1/n}(T_{p, q})\}_{n \in \N}}$ is well-known \cite{Fu90, FS90}. In our context, techniques such as the equivariant surgery formula and the use of Seiberg-Witten theory for families are similarly expected to fail.

In Theorem~\ref{thm:connectedsum}, we establish a connected sum inequality for the involutive $r_s$-invariant. (See Section~\ref{sec:2.1} for a discussion of connected sums.) Using this, we prove:

\begin{thm}\label{thm:1.4}
Let $K$ be any of the strongly invertible slice knots in Figure~\ref{fig:1.1}. Then any nontrivial linear combination of elements in $\{ (S^3_{1/n}(K), \tau) \}_{n \in \N} $ yields a strong cork.
\end{thm}
\noindent
The first member $(m = 0)$ of the family displayed in Figure~\ref{fig:1.1} is $\ov{9}_{46} = P(-3, 3, -3)$; note that $(+1)$-surgery on $\ov{9}_{46}$ gives the (boundary of the) Akbulut-Mazur cork. Historically, the Akbulut-Mazur cork was the first example of a cork presented in the literature, and was established using the Donaldson invariant by Akbulut \cite{Ak91_cork}.

\begin{figure}[h!]
\includegraphics[scale = 1.18]{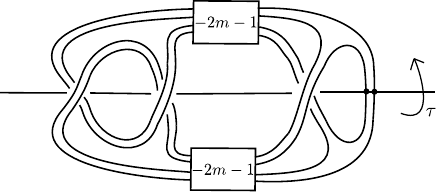}
\caption{The strongly invertible slice knots used in Theorems~\ref{thm:1.4} and \ref{thm:1.5} parameterized by $m \geq 0$. Figure~\ref{fig:1.1} is taken from \cite[Figure 12]{AKMR}. Here, $-2m - 1$ denotes the number of half twists.}\label{fig:1.1}
\end{figure}

\subsubsection{Simply-connected cobordisms}\label{sec:1.1.3}
Another advantage of Yang-Mills theory is that it can obstruct the existence of simply-connected cobordisms; see for example \cite{T87, Fu19, Da20, NST19, ADHLP22, Ta22}. 
This has its roots in a (variant of a) question of Akbulut \cite[Problem 4.95]{K78}, which asks (for example) whether there is a simply-connected cobordism from $\Sigma(2, 3, 5)$ to itself. (This was resolved negatively in \cite[Proposition 1.7]{T87}.) Here, we answer an equivariant version of this question: 

\begin{thm}\label{thm:1.5}
Let $K$ be any of the strongly invertible slice knots in Figure~\ref{fig:1.1}. Then for any $n \in \N$, there is no simply-connected, equivariant definite cobordism from $(S^3_{1/n}(K), \tau)$ to itself. 
\end{thm}

In particular, by considering $\smash{S^3_{1/n}(K) \# - S^3_{1/n}(K)}$ we obtain a cork $(Y, \tau)$ such that $\tau$ extends over some homology ball but not over any contractible manifold which $Y$ bounds. Note that this provides an example of a cork which is not strong; see \cite[Question 1.14]{DHM20} and \cite{HP20}.

\subsubsection{Nonorientable surfaces}\label{sec:1.1.4}
We now give some applications to the geography question for nonorientable surfaces. Let $K$ be a knot and $F$ be a nonorientable surface in $B^4$ with $K = \partial F$. There are two algebraic invariants associated to $F$: its \textit{nonorientable genus}, defined by $h(F) = b_1(F)$, and its \textit{normal Euler number} $e(F)$. The \textit{geography question} asks which pairs $(e, h)$ are realized by the set of (smooth) nonorientable slice surfaces for $K$. Work of Gordon-Litherland \cite{GL78} gives the well-known classical bound
\begin{equation}\label{eq:1.1}
\left | \sigma(K) - e(F)/2 \right | \leq h(F).
\end{equation}
A similar inequality was obtained by Ozsvath-Stipsicz-Szab\'o \cite{OSSunoriented}, who replaced the knot signature $\sigma(K)$ in (\ref{eq:1.1}) with the Floer-theoretic refinement $\upsilon(K)=2\Upsilon_K(1)$. Understanding the geography question in general is quite difficult; a complete answer is known only for a small handful of knots.  See \cite{GL11, MG18, Allen} for further results and discussion. 

In this paper, we consider the case where $F$ satisfies the equality
\begin{equation}\label{eq:1.2}
\left | \sigma(K) - e(F)/2 \right | = h(F).
\end{equation} 
We call such an $F$ an \textit{extremal surface}. As we explain in Section~\ref{sec:2.3}, (\ref{eq:1.2}) occurs precisely when the branched double cover $\Sigma_2(F)$ is definite. We thus apply our results regarding definite bounding to produce knots with no extremal slice surface. Note that such examples cannot be replicated by any single inequality of the same form as (\ref{eq:1.1}), including that of \cite{OSSunoriented}. Indeed, even combining the most commonly-used obstructions from \cite{Ba14} and \cite{OSSunoriented} necessarily fails to obstruct the set of all extremal surfaces; see Section~\ref{sec:7.5}. The authors are not aware of any previous example of a knot with no extremal slice surface appearing in the literature.





\begin{thm}\label{thm:1.6}
There exists a knot $J$ such that:
\begin{enumerate}
\item $J$ does not bound any extremal surface; and,
\item $\Sigma_2(J)$ bounds a contractible manifold.
\end{enumerate}
\end{thm}
\noindent
The knot $J$ is given by a certain linear combination of the knots $A_n$ and $B_n$ in Figure~\ref{fig:1.2}. These are constructed as follows: note that $\ov{9}_{46}$ admits two strong inversions, which we denote by $\tau$ and $\sigma$. (See for example Figure~\ref{fig:7.4}.) The knots $A_n$ and $B_n$ both have branched double cover $\smash{S^3_{1/n}(\ov{9}_{46})}$, with branching involutions corresponding to $\tau$ and $\sigma$, respectively.

\begin{figure}[h!]
\includegraphics[scale =0.6]{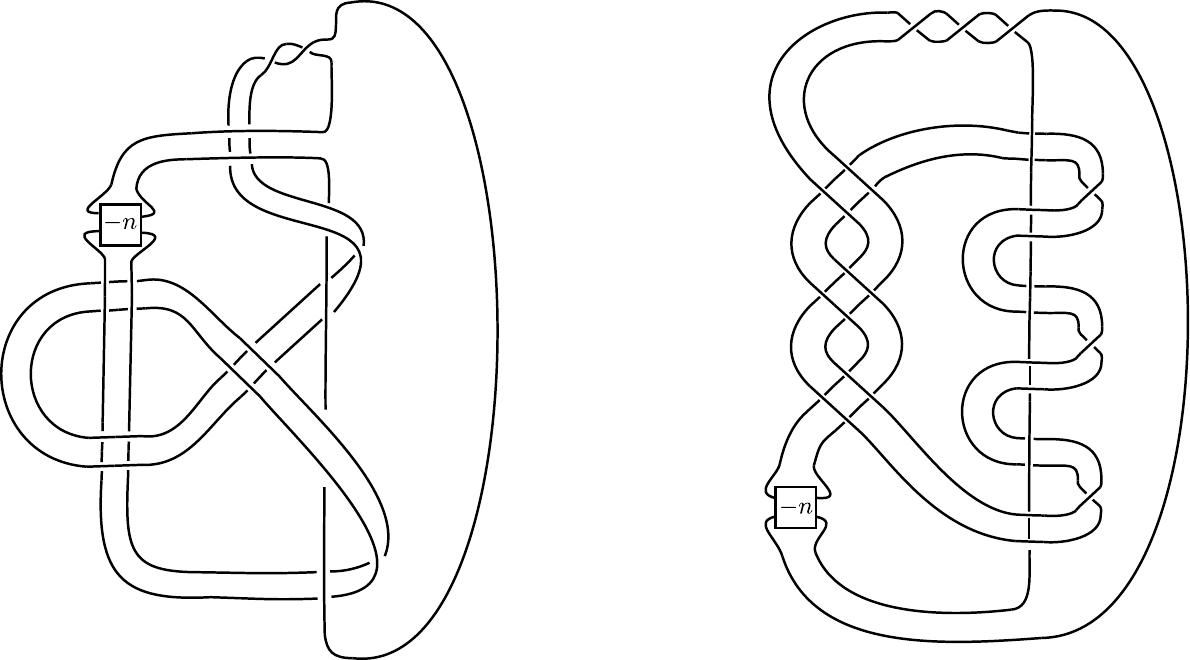}
\caption{Knots $A_n$ (left) and $B_n$ (right) with branched double cover $\smash{S^3_{1/n}(\ov{9}_{46})}$. The branching involution over $A_n$ corresponds to $\tau$; the branching involution over $B_n$ corresponds to $\sigma$. Here, $-n$ denotes the number of half twists.}\label{fig:1.2}
\end{figure}
If the second condition is removed, one can recover Theorem~\ref{thm:1.6} via the instanton-theoretic formalism of \cite{NST19} by taking any example for which $\Sigma_2(K)$ bounds no definite manifold. However, if $\Sigma_2(K)$ bounds a contractible manifold, then such a strategy clearly fails. In order to prove Theorem~\ref{thm:1.6}, we thus refine this approach by passing to the equivariant category. More precisely, note that if $K = \partial F$, then in fact $\Sigma_2(K) = \partial \Sigma_2(F)$ in the equivariant setting simply by remembering the branching involution over $F$. Theorem~\ref{thm:1.2} can then be leveraged to provide the desired examples. Similar ideas were used in \cite{ASA20, DHM20, dai20222} to define new sliceness obstructions; see Section~\ref{sec:2.3} for further discussion.

It is also natural to consider surfaces $F$ for which $\pi_1(B^4 - F)$ is as simple as possible; i.e., $\Z_2$. We refer to such an $F$ as a \textit{$\Z_2$-surface}. The condition on the fundamental group of $B^4 - F$ naturally arises in the setting of topological classification results; see \cite{COP}. One can refine the geography question by asking for which pairs $(e, h)$ we can find a $\Z_2$-surface for $K$.\footnote{Every $K$ bounds \textit{some} $\Z_2$-surface by taking the connected sum of a Seifert surface with an unknotted $\RP$.} It is easily checked that $\pi_1(\Sigma_2(F))$ is trivial for such $F$; thus, this problem may be approached by applying appropriate obstructions to simply-connected bounding. Here, we prove:

\begin{thm}\label{thm:1.7}
There exists a slice knot $J$ such that:
\begin{enumerate}
\item $J$ does not bound any extremal $\Z_2$-surface; and,
\item $\Sigma_2(J)$ bounds a contractible manifold.
\end{enumerate}
\end{thm}
\noindent
Explicitly, we may take $J = A_n \# -A_n$ for any $n \in \N$, where $A_n$ is the knot in Figure~\ref{fig:1.2}.

Note that since $J$ is slice, we have that $J$ does bound \textit{some} extremal surface (by taking the connected sum of any slice disk with $\RP$). However, the complement of this surface will have complicated fundamental group. Once again, if the second condition is removed, it is possible to use the results of \cite{NST19} to produce examples that bound an extremal surface but no extremal $\Z_2$-surface. Here, to prove Theorem~\ref{thm:1.7}, we use Theorem~\ref{thm:1.5} to obstruct simply-connected, equivariant definite boundings of $\Sigma_2(K)$. The fact that $\Sigma_2(K)$ bounds a contractible manifold shows that if we forget the equivariant category, then any such obstruction vanishes.

\subsubsection{Comparison to other invariants}
We stress that the examples obtained in the present work are qualitatively different than those obtained via Heegaard Floer homology \cite{ASA20, DHM20}, monopole Floer homology \cite{LRS18}, and Seiberg-Witten theory for families \cite{KMT23A}. For instance, as discussed in Section~\ref{sec:1.1.2}, these theories are ill-equipped to handle linear combinations of $1/n$-surgeries on the same knot $K$. Likewise, the TQFT nature of such invariants prohibit applications such as Theorem~\ref{thm:1.5}, since this requires distinguishing contractible manifolds from homology balls. Moreover, even for simple manifolds such as Brieskorn spheres, the information contained in our instanton-theoretic construction differs from the output of \cite{LRS18, ASA20, DHM20, KMT23A}; see Section~\ref{sec:brieskorn}. While the examples obtained using the latter are often qualitatively similar to each other, the usage of the Chern-Simons filtration in the present work appears to produce fundamentally new results in the study of equivariant bordism.

There are also several subtle differences between our involutive $r_s$-invariant and the work of Nozaki-Sato-Taniguchi \cite{NST19}. Our involutive instanton theory is related to studying the Chern-Simons functional on the mapping torus of the configuration space with respect to a diffeomorphism action on $Y$. Thus, the involutive $r_s$-invariant may be viewed as a $1$-parameter version of the usual $r_s$-invariant of \cite{NST19}. This is related to the fact that when we prove invariance under equivariant homology cobordism, we will need to count points in $1$-parameter families of ASD moduli spaces. In contrast, invariance of the original $r_s$-invariant can be established by counting points in usual (unparameterized) ASD-moduli spaces. 

\subsubsection{Instanton-theoretic aspects of the construction}
We close by discussing some of the technical difficulties regarding the use of instanton Floer theory in the present work. As is standard when defining homology cobordism invariants, it will be necessary to work with a formulation of instanton Floer homology which takes into account the reducible connection. The usual method for doing this is via the maps $D_1$ and $D_2$ of \cite{Do02}; see for example \cite{Do02, Fr02, Da20}. (This is usually what is meant by $SO(3)$-equivariant instanton theory.) While several such constructions are present in the literature, the standard versions are generally defined over fields such as $\Q$ \cite{Do02, Fr02, Da20}. See work of Miller Eismeier \cite{Mike19} and Daemi-Miller Eisemier \cite{DaMi22} for more general coefficients.

However, in order to define an involutive $r_s$-invariant, it will be necessary to work over $\Z_2$. (See Remark~\ref{rem:differentcoefficients}.) To do this, we consider a restricted subcomplex of the construction in \cite{NST19}, which corresponds to only using the $D_1$-map (or equivalently $D_2$-map) originally from \cite{Do02}. Unfortunately, this means that we do not have a complete dualization or tensor product formula for such complexes. This is partially why our current proof of \cref{thm:1.2} requires input from the Heegaard-Floer-theoretic side. We direct the reader to recent work of Fr\o yshov \cite{Fr23} regarding instanton Floer homology over $\Z_2$.

\begin{rem}\label{rem:1.8}
It is also possible to define the formalism of the current paper using instanton Floer homology with $\Z$-coefficients. However, none of the examples presented here differ significantly when working over $\Z$ rather than $\Z_2$, so we have chosen to work with the latter out of convenience.
\end{rem}

Finally, note that computing instanton Floer chain complexes (especially with a $\tau$-action) is quite difficult. While partial information can sometimes be obtained by an analysis of the $SU(2)$-representation variety of $Y$, in general there are few tools for constraining the action of $\tau$; see the work of Saveliev \cite{Sa03} and Ruberman-Saveliev \cite{RS04}. Here, we use the Donaldson invariant to constrain the involutive structure of the Akbulut-Mazur cork. We approach other examples via a topological argument involving equivariant negative-definite cobordisms, as in \cite{DHM20}. \\

\noindent
\textbf{Organization.} In Section~\ref{sec:2}, we introduce various topological definitions and constructions related to equivariant bordism. In Sections~\ref{sec:3} and \ref{sec:4}, we develop the main algebraic framework of the paper and define the involutive $r_s$-invariant. Analytic details of the construction are discussed in Section~\ref{sec:5}. In Section~\ref{sec:6}, we prove Theorem~\ref{thm:1.1} and give an extended list of properties of the $r_s$-invariant. Finally, in Section~\ref{sec:7} we establish the remaining applications listed in the introduction. \\

\noindent
\textbf{Acknowledgements.} The authors would like to thank Aliakbar Daemi, Hokuto Konno, Tomasz Mrowka, Daniel Ruberman, and Kouki Sato for helpful conversations. Part of this work was carried out at in the program entitled ``Floer homotopy theory" held at  MSRI/SL-Math (Fall 2022). Thus this work was supported by NSF DMS-1928930. The second author was partially supported by NSF DMS-1902746 and NSF DMS-2303823. The third author was partially supported by MSRI and NSF DMS-2019396. The fourth author was partially supported by JSPS KAKENHI Grant Number 20K22319, 22K13921, and RIKEN iTHEMS Program.

\section{Background and definitions}\label{sec:2}

In this section, we give several preliminary definitions and discuss equivariant bordism.

\subsection{Strong corks}\label{sec:2.1} We begin with the definition of a strong cork, as introduced by Lin, Ruberman, and Saveliev \cite[Section 1.2.2]{LRS18}:

\begin{defn}\label{def:2.1}
Let $Y$ be an integer homology $3$-sphere and $\tau$ be an orientation-preserving involution on $Y$. We say that $(Y, \tau)$ is a \textit{strong cork} if $\tau$ does not extend as a diffeomorphism over any homology ball which $Y$ bounds. We also require that $Y$ bound at least one contractible manifold, so that $Y$ is a cork boundary in the traditional sense. 
\end{defn}

We have the following notion of equivariant cobordism:

\begin{defn}\label{def:2.2}
Let $(Y, \tau)$ and $(Y, \tau')$ be two integer homology spheres equipped with orientation-preserving involutions. We say $(W, \wt{\tau})$ is an \textit{equivariant cobordism} from $(Y, \tau)$ to $(Y', \tau')$ if $W$ is a cobordism from $Y$ to $Y'$ and $\wt{\tau}$ is a self-diffeomorphism of $W$ which restricts to $\tau$ and $\tau'$ on $Y$ and $Y'$, respectively.
\end{defn}
\noindent
We may further specialize Definition~\ref{def:2.1} by requiring $W$ to be a homology cobordism, in which case we refer to $(W, \wt{\tau})$ as an \textit{equivariant homology cobordism}. This notion defines an equivalence relation on the set of pairs $(Y, \tau)$; we denote this by $\sim$. Clearly, $(Y, \tau)$ is a strong cork if and only if $Y$ bounds a contractible manifold and $(Y, \tau) \not\sim (S^3, \id)$. Note that by the third part of Theorem~\ref{thm:1.1}, $r_s(Y, \tau)$ is an invariant of the equivalence class of $(Y, \tau)$. 

As discussed in Section~\ref{sec:1.1.1}, we will sometimes have cause to place homological constraints on the extension $\wt{\tau}$. We define:

\begin{defn}\label{def:2.3}
Let $(Y, \tau)$ and $(Y, \tau')$ be two integer homology spheres equipped with orientation-preserving involutions and $(W, \wt{\tau})$ be an equivariant cobordism between them.
\begin{itemize}
\item We say that $\wt{\tau}$ is \textit{homology-fixing} if $\wt{\tau}_* = \id$ on $H_2(W, \Q)$.
\item We say that $\wt{\tau}$ is \textit{homology-reversing} if $\wt{\tau}_* = - \id$ on $H_2(W, \Q)$.\footnote{We could require that $\wt{\tau}_*$ act as $\pm \id$ on $H_2(W, \Z)$, but here we will use $H_2(W, \Q)$ due to Theorem~\ref{lem:2.9}.}
\end{itemize}
\end{defn}

If $(Y, \tau)$ and $(Y', \tau')$ are two equivariant homology spheres, one can attempt to form their equivariant connected sum. To this end, let $p \in Y$ and $p' \in Y'$ be fixed points for $\tau$ and $\tau'$, respectively, which have equivariant neighborhoods that are identified via a diffeomorphism intertwining $\tau$ and $\tau'$. Then there is an obvious equivariant connected sum operation
\[
(Y, \tau)\#(Y',\tau')=(Y\#Y',\tau\#\tau').
\]
This may depend on the choice (and existence) of $p$ and $p'$. In all of the examples of this paper, the fixed-point set of $\tau$ will be diffeomorphic to $S^1$, in which case it is clear that the connected sum operation is well-defined. 

\begin{rem}
We stress that requiring $\wt{\tau}$ to be an involution is quite different than requiring $\wt{\tau}$ to be a diffeomorphism. Although obstructing the latter clearly obstructs the former, there are many examples in which the two notions are not equivalent. For instance, let $\tau$ be the involution on any Brieskorn sphere $\Sigma(p, q, r)$ coming from the obvious $\Z_2$-subgroup of the $S^1$-action. In \cite[Theorem A]{AH21}, it is shown that $\tau$ cannot extend as an involution over any contractible manifold which $\Sigma(p, q, r)$ bounds. On the other hand, it is easy to produce an extension of $\tau$ as a diffeomorphism by using the fact that $\tau$ is isotopic to the identity. In the present paper, we study the case where $\wt{\tau}$ is required to be a diffeomorphism, so as to preserve the original motivation coming from the theory of corks.
\end{rem}


\subsection{Equivariant surgery}\label{sec:2.2}

One particularly flexible method for constructing a symmetric $3$-manifold is via surgery on an equivariant knot. As these will provide a robust source of examples in this paper, we include a brief discussion here.

Recall that by Smith theory, any orientation-preserving involution on $S^3$ is conjugate to rotation about a standard unknot. We say that a knot $K \subseteq S^3$ is \textit{equivariant} if it is preserved by such an involution $\tau$. If $\tau$ has two fixed points on $K$, then we say that $K$ is \textit{strongly invertible}. If $\tau$ has no fixed points on $K$, then we say that $K$ is \textit{periodic}. In \cite[Section 5.1]{DHM20}, it is shown that if $(K, \tau)$ is an equivariant knot, then any surgered manifold $\smash{S^3_{p/q}(K)}$ inherits an involution from the symmetry on $K$, which we also denote by $\tau$. Note that any symmetric $3$-manifold constructed in this way has fixed-point set diffeomorphic to $S^1$.

Such examples fit naturally into the formalism of this paper, as the following lemmas show:

\begin{lem}\label{lem:knotsA}
Let $K$ be a strongly invertible or periodic knot with symmetry $\tau$. For any $n \in \N$, there is a simply-connected, equivariant negative-definite cobordism $(W, \wt{\tau})$ from 
\[
(S^3_{1/(n+1)}(K), \tau) \quad \text{to} \quad (S^3_{1/n}(K), \tau).
\]
This cobordism has $H_2(W, \Z) = \Z$. Moreover:
\begin{itemize}
\item If $K$ is strongly invertible, then $\widetilde{\tau}$ is homology-reversing.
\item If $K$ is periodic, then $\widetilde{\tau}$ is homology-fixing.
\end{itemize}
\end{lem}
\begin{proof}
We explicitly construct a $2$-handle attachment cobordism from 
\[
Y_{n+1} = S^3_{1/(n+1)}(K) \quad \text{to} \quad Y_n = S^3_{1/n}(K)
\]
as follows. The reader may check that if $K$ is a strongly invertible or periodic knot, then we may find an equivariant Seifert framing $\lambda$ of $K$. View $\lambda$ as lying inside a solid torus neighborhood $N(K)$ of $K$. We remove $N(K)$ from $S^3$ and re-glue it along the matrix
\[
\left(\begin{array}{cc}1 & 0 \\n+1 & 1\end{array}\right)
\]
to obtain $Y_{n+1}$. For clarity, let $K'$ and $\lambda'$ denote the images of $K$ and $\lambda$, respectively, in $Y_{n+1}$.

We construct our cobordism by attaching a $2$-handle along $K'$, with framing $-1$ relative to $\lambda'$. This means that the outgoing boundary is the surgery along $K \subseteq S^3$ with surgery matrix
\[
\left(\begin{array}{cc}1 & 0 \\n+1 & 1\end{array}\right) \left(\begin{array}{cc}1 & 0 \\-1 & 1\end{array}\right) = \left(\begin{array}{cc}1 & 0 \\n & 1\end{array}\right);
\]
that is, we obtain $Y_n$. It is not hard to check that the intersection form of this cobordism is $(-1)$. Indeed, examining the surgery matrix for $Y_{n+1}$, note that $K'$ and $\lambda'$ may be pushed into $Y_{n+1} - N(K') = S^3 - N(K)$. After doing so, we obtain two Seifert framings for $K$ in $S^3$. Hence $K'$ and $\lambda'$ bound disjoint Seifert surfaces in $Y_{n+1}$. The fact that we attached our $2$-handle with framing $-1$ relative to $\lambda'$ shows that the intersection form is indeed $(-1)$.

Thus our cobordism is evidently simply-connected and negative-definite. If $K$ is strongly invertible, then it is not hard to check that the extension $\wt{\tau}$ discussed in \cite[Section 5.1]{DHM20} reverses orientation on the generator of $H_2(W, \Z)$, while if $K$ is periodic, then $\wt{\tau}$ sends the generator of $H_2(W, \Z)$ to itself. This completes the proof.
\end{proof}

This yields:

\begin{lem}\label{lem:knotsB}
Let $K$ be a strongly invertible or periodic knot with symmetry $\tau$. For any nonzero $n \in \Z$, the surgered manifold $\smash{(S^3_{1/n}(K), \tau)}$ is the boundary of a simply-connected, equivariant definite manifold $\smash{(W, \wt{\tau})}$. Moreover:
\begin{itemize}
\item If $n > 0$, then $W$ is positive definite.
\item If $n < 0$, then $W$ is negative definite.
\end{itemize}
and:
\begin{itemize}
\item If $K$ is strongly invertible, then $\widetilde{\tau}$ is homology-reversing.
\item If $K$ is periodic, then $\widetilde{\tau}$ is homology-fixing.
\end{itemize}
\end{lem}
\begin{proof}
Suppose $n > 0$. By repeatedly applying Lemma~\ref{lem:knotsA}, we obtain a simply-connected, equivariant negative-definite cobordism from
\[
(S^3_{1/n}(K), \tau) \quad \text{to} \quad (S^3_{+1}(K), \tau).
\]
Turning this around gives a positive-definite cobordism from $(+1)$-surgery to $1/n$-surgery on $K$. Moreover, $(S^3_{+1}(K), \tau)$ itself bounds a simply-connected, equivariant positive-definite manifold given by the trace of $(+1)$-surgery on $K$. If $K$ is strongly invertible, then $\wt{\tau}$ on each of these handle attachment cobordisms acts as multiplication by $-1$, and hence the composite cobordism clearly satisfies the desired properties. Similarly, if $K$ is periodic, then $\wt{\tau}$ acts as the identity on each handle attachment cobordism. The case $n < 0$ follows by mirroring and negating $n$.
\end{proof}

\subsection{Branched covers}\label{sec:2.3}
Another context in which equivariant manifolds naturally arise is the setting of branched covers. Recall that if $K$ is a knot with slice disk $D$, then the branched double cover $\Sigma_2(K)$ bounds a $\Z_2$-homology ball given by $\Sigma_2(D)$. This provides a well-known obstruction to the sliceness of $K$. 

Following \cite{ASA20, DHM20, dai20222}, we may refine this approach by remembering the branching involution on $\Sigma_2(K)$. This gives $\Sigma_2(K)$ the structure of an equivariant $\Z_2$-homology sphere. If $K$ bounds a slice disk $D$, then $\Sigma_2(D)$ is likewise equivariant and $\Sigma_2(K) = \partial \Sigma_2(D)$ in the sense of Definition~\ref{def:2.2}. Asking for the existence of an \textit{equivariant} homology ball with boundary $\Sigma_2(K)$ is (in principle) stronger than the corresponding obstruction in the nonequivariant category. In \cite{ASA20, DHM20, dai20222} this was used to obtain new linear independence and sliceness results with the help of Heegaard Floer homology; and, in particular, a proof that the $(2, 1)$-cable of the figure-eight knot is not slice \cite{dai20222}. In each case, the corresponding nonequivariant obstruction was trivial and the use of the branching involution proved to be essential.

As discussed in Section~\ref{sec:1.1.4}, we will also have cause to consider more general slice surfaces $F$ for $K$. We thus collect together various results regarding the topology of $\Sigma_2(F)$. We first have the following result of Gordon-Litherland \cite{GL78}: 

\begin{thm}\cite{GL78}\label{thm:2.7}
Let $F$ be a nonorientable slice surface for $K$ in $B^4$. Then
\[
\sigma(K) = \sign(\Sigma_2(F)) + e(F)/2.
\]
\end{thm}
\noindent
Next, we have the following facts regarding the basic algebraic topology of $\Sigma_2(F)$: 
\begin{lem}\label{lem:2.8}
Let $F$ be a nonorientable slice surface for $K$ in $B^4$. Then:
\[
H_1(\Sigma_2(F), \Z_2) = 0 \quad \text{and} \quad b_2(\Sigma_2(F)) = b_1(F). 
\]
\end{lem}
\begin{proof}
For the first part, see for example \cite[Lemma 3.8]{Na00}. For the second, see for example \cite[Lemma 2]{Massey}.
\end{proof}
\noindent
Note that since $| \sign(\Sigma_2(F)) | \leq b_2(\Sigma_2(F)) = b_1(F)$, Theorem~\ref{thm:2.7} then establishes the inequality (\ref{eq:1.1}) of Section~\ref{sec:1.1.4}. Moreover, equality occurs precisely when $| \sign(\Sigma_2(F)) | = b_2(\Sigma_2(F))$; that is, $\Sigma_2(F)$ is definite. Finally, we have: 
\begin{lem}\cite[Lemma 2]{Massey}\label{lem:2.9}
The branching action $\wt{\tau}$ on $\Sigma_2(F)$ acts as $- \id$ on $H_2(\Sigma_2(F), \Q)$.
\end{lem}

We thus have:
\begin{lem}
If $K$ bounds an extremal surface $F$, then $\Sigma_2(K)$ bounds a homology-reversing, equivariant definite manifold $W$ with $H_1(W, \Z_2) = 0$.
\end{lem}
\begin{proof}
Follows immediately from Lemmas~\ref{lem:2.8} and \ref{lem:2.9}.
\end{proof}

\section{Algebraic preliminaries}\label{sec:3}

In this section, we review the set-up of \cite{NST19} and re-cast it in an algebraic framework which will more easily generalize to the equivariant setting.

\subsection{Instanton-type complexes}\label{sec:3.1}

We begin with a modification to the usual instanton Floer complex \cite{Fl88, Do02} which takes into account the reducible connection, and the Chern-Simons filtration. Throughout the paper, we consider $\Z$-graded, $\R$-filtered chain complexes over the ring $\Z_2[y^{\pm 1}]$ of Laurent polynomials. We denote by $\deg_\Z$ the homological grading, and  by $\deg_I$ the filtration level. Under the differential, $\deg_\Z$ is decreased by one, and $\deg_I$ is nonincreasing. The gradings are so that multiplication by the variable $y$ has 
$\deg_\Z(y) = 8 \text{ and } \deg_I(y) = 1$.

When regarded as a chain complex  $\Z_2[y^{\pm 1}]$ is regarded as the trivial complex spanned by a single generator with $\deg_\Z = \deg_I = 0$, and no differential. We use brackets to denote a shift in homological grading, so that (for example) the generator $1 \in \Z_2[y^{\pm 1}][-3]$ has homological grading $\deg_\Z = -3$, and $\deg_I = 0$.

\begin{defn}\label{def:3.1}
Let $(\un{C},\un{d})$  be a $\Z$-graded, $\R$-filtered, finitely-generated, free chain complex over $\Z_2[y^{\pm1}]$. We say $(\un{C},\un{d})$ is an {\it instanton-type chain complex} if it is equipped with a subcomplex $C \subseteq \un{C}$ such that there is an graded, filtered isomorphism of $\Z_2[y^{\pm 1}]$-complexes
\[
\un{C}/C \cong \Z_2[y^{\pm 1}][-3].
\]
Denote
\[
\pi \colon \un{C} \rightarrow \un{C}/C \cong \Z_2[y^{\pm 1}][-3].
\]
We will often think of $C$ as defining a two-step filtration $C \subseteq \un{C}$. We denote by $\un{C}_i$ or $C_i$ the appropriate complex in homological grading $i$.
\end{defn}

If $\un{C}$ is an instanton-type complex, then we may choose a splitting of graded, filtered $\Z_2[y^{\pm 1}]$-modules
\[
\un{C} = C \oplus \Z_2[y^{\pm 1}][-3].
\]
With respect to this splitting, 
\[
\un{d} = d + D_2, 
\]
where $d$ is the restriction of $\un{d}$ to $C$ and $D_2$ is a filtered map from $\Z_2[y^{\pm 1}][-3]$ to $C$. While such a splitting is not canonical, we often assume that a particular splitting has been fixed whenever we discuss $\un{C}$. We then denote the generator $1 \in \Z_2[y^{\pm 1}][-3]$ by
\[
\theta = (0, 1).
\]
Note that $D_2$ is completely determined by $D_2(\theta) \in C_{-4}$.

\begin{rem} Definition~\ref{def:3.1} should be thought of as the analog of the fact that in Heegaard Floer homology  $\HFm(Y)/U\text{-torsion} \cong \Z_2[U]$ for any integer homology sphere $Y$. In that case we similarly have a noncanonical splitting $\HFm(Y) = (U\text{-torsion}) \oplus \Z_2[U]$. 
\end{rem}

\begin{rem}\label{rem:3.2}
We sketch the relevance of Definition~\ref{def:3.1} to instanton Floer homology. Recall that instanton Floer homology \cite{Fl88, Do02} associates to an integer homology sphere $Y$ equipped with a Riemannian metric a $\Z$-graded chain complex. This is done by looking at the flow lines of a suitable perturbation $f_\pi$ of the Chern-Simons functional
 \[
 f: \mathcal{A}(P)/\mathcal{G}_0(P) \to \R.
 \] 
Here $P\to Y$ denotes the trivial $SU(2)$-bundle over $Y$, $\mathcal{A}(P)$ the affine space of $SU(2)$-connections on $P$, and $\mathcal{G}_0(P)$ the space of degree zero gauge transformations of $P$.  In the usual setting, only irreducible connections on $P$ are considered. We denote the resulting  chain complex by $(C(Y,g,\pi), d)$, where $g$ denotes Riemannian metric on $Y$, and $\pi$ an admissible holonomy perturbation; see \cite[Section (1b)]{Fl88} and \cite[Section 5.5.1]{BD95}. Note that $C(Y,g,\pi)$ has the structure of a $\Z_2[y^{\pm 1}]$-module where multiplication by $y^k$ acts as the gauge action of a gauge transformation of degree $k\in \Z$. 

In order to take into account the reducible connection, we define
\[\un{C}(Y,g,\pi)= C(Y,g,\pi) \oplus \Z_2[y^{\pm1}] \cdot \theta,\]
where $\theta$ is a formal generator (in homological grading $-3$) representing the class of the trivial connection on $P$. The differential on $\un{C}(Y,g,\pi)$ is then $\un{d} = d + D_2$, where $d$ is the usual differential on the irreducible complex $C(Y,g,\pi)$ and $D_2$ counts flow lines out of the reducible connection; see \cite[Section 7.1]{Do02}. We put a filtration on $\un{C}(Y, g, \pi)$ by defining
\[
\deg_I(\bold{x})=f_\pi(\bold{x})
\]
for every irreducible critical point $\bold{x}\in \text{crit}(f_\pi)$ and setting the filtration of $\theta$ to be zero.
\end{rem}

\begin{rem}\label{rem:3.3}
Note that $\un{C}(Y,g,\pi)$ depends on the perturbation $\pi$ (even up to an appropriate notion of filtered homotopy equivalence) and thus is not an invariant of $Y$. In order to remedy this, we introduce the definition of an enriched complex in Section~\ref{sec:4.4}, which will capture the notion of taking a sequence of perturbations converging to zero. We gloss over this subtlety for now.
\end{rem}


The following auxiliary complexes capture information from the Chern-Simons filtration.

\begin{defn}\label{def:3.4}
Let $(\un{C}, \un{d})$ be an instanton-type complex. For any real number $s$, define
\[
(\un{C}^{[-\infty,s]}, \un{d}^{[-\infty,s]}) = (\{ \zeta \in \un{C} | \deg_I (\zeta ) \leq s\} , \un{d}|_{\un{C}^{[-\infty,s]}}). 
\]
This is a subcomplex of $\un{C}$. For a given pair of real numbers $r<s$, we may also consider the quotient complex
\[
\un{C}^{[r,s]} = \un{C}^{[-\infty,s]}/ \un{C}^{[-\infty,r]}.
\]
When discussing $\un{C}^{[r,s]}$, we will usually assume $r < 0 \leq s$. Define complexes $C^{[-\infty, s]}$ and $C^{[r, s]}$ similarly. Denote
\[
\pi^{[r, s]} \colon \un{C}^{[r, s]} \rightarrow \un{C}^{[r, s]}/C^{[r, s]} \cong \Z_2[y^{\pm 1}][-3]^{[r, s]}.
\]
Note that $\Z_2[y^{\pm 1}][-3]^{[r, s]}$ is generated by the cycles $y^{i}$ with $r < \deg_I(y^{i}) \leq s$. (Or, more precisely, the classes of these cycles in the quotient $\Z_2[y^{\pm 1}]^{[-\infty, r]}/\Z_2[y^{\pm 1}]^{[-\infty, s]}$.)
\end{defn}
As in the discussion following Definition~\ref{def:3.1}, we will often assume that a splitting of $\un{C}$ has been chosen. Denote
\[
\theta^{[r, s]} = \pi^{[r, s]}(\theta).
\]
This is nonzero if and only if $r < 0 \leq s$. Sometimes we will continue to write $\theta$ in place of $\theta^{[r, s]}$ when the context is clear.

We now discuss maps between instanton-type chain complexes. 

\begin{defn}\label{def:3.5}
Let $(\un{C}, \un{d})$ and $(\un{C}', \un{d}')$ be two instanton-type complexes. A \textit{morphism} 
\[
f \colon \un{C} \rightarrow \un{C}'
\]
is a $\deg_\Z$-preserving, $\Z_2[y^{\pm 1}]$-linear chain map that preserves the two-step filtration:
\[
f(C) \subseteq C'.
\]
This induces a map
\[
f \colon \un{C}/C \cong \Z_2[y^{\pm 1}][-3] \rightarrow \un{C}'/C' \cong \Z_2[y^{\pm 1}][-3],
\]
which by abuse of notation we also denote by $f$.
\end{defn}

We emphasize that an instanton-type complex has two filtrations: the $\deg_I$-filtration and the two-step filtration afforded by Definition~\ref{def:3.1}. As per Definition~\ref{def:3.5}, in the present work we will only ever consider maps which preserve the two-step filtration. We thus refer to a map as \textit{filtered} (with no qualification) to mean that it is filtered with respect to $\deg_I$. While we will primarily be interested in filtered maps, it will also be useful to have a notion of a chain map that increases the filtration by at most a fixed parameter:

\begin{defn}\label{def:3.6}
Let $(\un{C},\un{d}) $ and $(\un{C}',\un{d}')$ be two instanton-type chain complexes and fix any real number $\delta \geq 0$. We say that a morphism $f \colon \un{C} \rightarrow \un{C}'$ has \textit{level} $\delta$ if 
\[
\deg_I(f(\zeta)) \leq \deg_I(\zeta) + \delta
\]
for all $\zeta \in \un{C}$. Note that a level-zero morphism is filtered in the usual sense. Clearly, a level-$\delta$ morphism induces a morphism:
\[
f^{[r, s]} \colon \un{C}^{[r, s]} \rightarrow \un{C}'^{[r + \delta, s + \delta]}
\]
for any $r < s$.
\end{defn}

We have the obvious notion of homotopy equivalence:

\begin{defn}\label{def:3.7}
Let $(\un{C},\un{d}) $ and $(\un{C}',\un{d}')$ be two instanton-type chain complexes. We say that $\un{C}$ and $\un{C}'$ are \textit{homotopy equivalent} if there exist morphisms $f$ and $g$ between them and $\Z_2[y^{\pm 1}]$-linear homotopies
\[
H_{\un{C}} \colon \un{C}_* \rightarrow \un{C}_{* + 1} \quad \text{and} \quad H_{\un{C}'} \colon \un{C}'_* \rightarrow \un{C}'_{* + 1}
\]
such that
\[
gf + \id = \un{d}H_{\un{C}} + H_{\un{C}}\un{d} \quad \text{and} \quad fg + \id = \un{d}'H_{\un{C}'} + H_{\un{C}'}\un{d}'.
\]
We require that $H_{\un{C}}$ and $H_{\un{C}'}$ preserve the two-step filtrations on $\un{C}$ and $\un{C}'$, respectively. 
\begin{itemize}
\item If $f$, $g$, $H_{\un{C}}$, and $H_{\un{C}'}$ are filtered then we say that $\un{C}$ and $\un{C}'$ are \textit{filtered homotopy equivalent}.
\item If $f$, $g$, $H_{\un{C}}$, and $H_{\un{C}'}$ have level $\delta$, then we say that $\un{C}$ and $\un{C}'$ are \textit{level-$\delta$ homotopy equivalent}.
\end{itemize}
It is clear that filtered homotopy equivalence is an equivalence relation. Note, however, that for a fixed value of $\delta$, level-$\delta$ homotopy equivalence is \textit{not} an equivalence relation, since the composition of two level-$\delta$ morphisms need not have level $\delta$.
\end{defn}

We now turn to the following important notion: 

\begin{defn}\label{def:3.8}
Let $(\un{C},\un{d}) $ and $(\un{C}',\un{d}') $ be two instanton-type chain complexes. We say that a morphism $\lambda \colon \un{C} \rightarrow \un{C}'$ is a \textit{local map} if the induced map on the quotient
\[
\lambda \colon \un{C}/C \cong \Z_2[y^{\pm 1}][-3] \rightarrow \un{C}'/C' \cong \Z_2[y^{\pm 1}][-3] 
\]
is an isomorphism.
\begin{itemize}
\item If $\lambda$ is filtered, then we refer to $\lambda$ as a \textit{filtered local map}.
\item If $\lambda$ has level $\delta$, then we refer to $\lambda$ as a \textit{level-$\delta$ local map}.
\end{itemize}
If we have local maps in both directions between $\un{C}$ and $\un{C}'$, we say that $\un{C}$ and $\un{C}'$ are \textit{locally equivalent}. 
\begin{itemize}
\item If both maps are filtered, then we refer to this as a \textit{filtered local equivalence}.
\item If both maps have level $\delta$, then we refer to this as a \textit{level-$\delta$ local equivalence}.
\end{itemize}
It is clear that filtered local equivalence is an equivalence relation. Note, however, that for a fixed value of $\delta$, level-$\delta$ local equivalence is \textit{not} an equivalence relation, since the composition of two level-$\delta$ morphisms need not have level $\delta$.
\end{defn}

Definition~\ref{def:3.8} should be thought of as the rough analogue of the notion of a local map in Heegaard Floer homology. Recall that in Heegaard Floer theory, a map $f \colon CF^-(Y) \rightarrow CF^-(Y')$ is said to be local if it induces an isomorphism between the localizations
\[
U^{-1} \HFm(Y) \rightarrow U^{-1} \HFm(Y').
\]
We have thus borrowed this terminology, even though $\un{C}/C$ is not really the localization of $\un{C}$ with respect to any variable.

\subsection{$\theta$-supported cycles and the $r_s$-invariant}\label{sec:3.2}

We now review the definition of the $r_s$-invariant of \cite{NST19}. For this, we will need the notion of a $\theta$-supported cycle, which may be regarded as partial information of the special cycles introduced in \cite{DISST22}. Roughly speaking, a $\theta$-supported cycle should be regarded as the analog of a $U$-nontorsion element in $\HFm(Y)$.

\begin{defn}\label{def:3.9}
Let $(\un{C}, d)$ be an instanton-type complex. We say that a cycle $z \in \un{C}$ is a \emph{$\theta$-supported cycle} if $\pi_*([z]) = 1 \in \Z_2[y^{\pm 1}][-3]$, where
\[
\pi_* \colon H_*(\un{C}) \rightarrow H_*(\un{C}/C) \cong \Z_2[y^{\pm 1}][-3].
\]
Similarly, we say that a cycle $z \in \un{C}^{[r, s]}$ is a \emph{$\theta$-supported $[r, s]$-cycle} if $\pi_*^{[r, s]}([z]) = \pm 1 \in \Z_2[y^{\pm 1}][-3]^{[r, s]}$, where
\[
\pi_*^{[r, s]} \colon H_*(\un{C}^{[r, s]}) \rightarrow H_*(\un{C}^{[r, s]}/C^{[r, s]}) \cong \Z_2[y^{\pm 1}][-3]^{[r, s]}.
\]
Here, we require that $r < 0 \leq s$, so that $\pm 1$ is a nonzero element of $\Z_2[y^{\pm 1}][-3]^{[r, s]}$.
\end{defn}

It will be convenient to think of $\theta$-supported cycles explicitly by choosing a splitting of $\un{C}$ as in the discussion after Definition~\ref{def:3.1}. Then any element $z \in \un{C}$ can be written $z=\zeta+c\cdot \theta$ for some $\zeta \in C$ and $c\in\Z_2[y^{\pm1}]$. A $\theta$-supported cycle is a cycle $z$ for which $c =  1$.\footnote{When we work over $\Z$, $c$ is allowed to be $\pm1$.} While the splitting of $\un{C}$ is not canonical, it is straightforward to see that this property is equivalent to Definition~\ref{def:3.9} and thus independent of the choice of splitting. A similar characterization holds regarding $\theta^{[r, s]}$ and $\theta$-supported $[r, s]$-cycles. 

We define the $r_s$-invariant of an instanton-type chain complex using the notion of a filtered $\theta$-supported cycle:
\begin{defn} \label{def:3.10}
Let $\un{C}$ be an instanton-type complex and $s \in [-\infty, 0]$. Define  
\[
r_s ( \un{C})= - \inf_{r < 0} \{ r |
\text{ there exists a $\theta$-supported $[r, -s]$-cycle}\} \in [0, \infty],
\]
with the caveat that if the above set is empty, we set $r_s(\un{C}) = - \infty$.
\end{defn}

The signs in Definition~\ref{def:3.10} are for consistency with the conventions of \cite{NST19}. See Remark~\ref{rem:differenceinrs}.

\begin{ex}
Several examples of instanton-type complexes and graphs of their $r_s$-invariants are given in Figure~\ref{fig:3.1}. 
\begin{itemize}
\item In (a), $\theta$ itself is a $\theta$-supported $[r, -s]$-cycle for all values of $r$ and $s$; hence $r_s(\un{C}) = \infty$ for all $s$. 
\item In (b), there is a $\theta$-supported $[r, -s]$-cycle if and only if $r \geq \beta$; this is again given by $\theta$ itself. 
\item In (c), there is a $\theta$-supported $[r, -s]$-cycle for all values of $r$ and $s$. This is given by $\theta$ if $r \geq \beta$ and $\theta - y$ if $r < \beta$. 
\item The most nontrivial case occurs in (d). If $-s \geq \alpha$, then $\theta - y$ forms a $\theta$-supported $[r, -s]$-cycle for any value of $r$. If $-s < \alpha$, then there is a $\theta$-supported $[r, -s]$-cycle if and only if $r \geq \beta$, which is given by $\theta$.
\end{itemize}
\begin{figure}[h!]
\includegraphics[scale = 1]{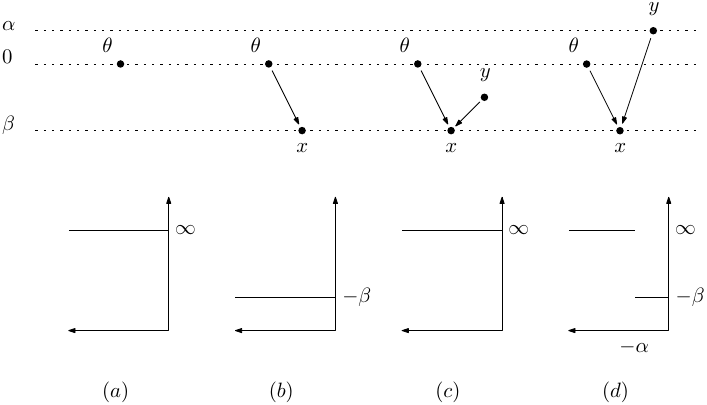}
\caption{Some instanton-type complexes and graphs of their $r_s$-invariants. Dotted lines represent the $\deg_I$-filtration. The differential is given by the black arrows; if no black arrow is drawn, the differential on a given generator is zero.}\label{fig:3.1}
\end{figure}
\end{ex}

We refer the reader to \cite{NST19} for further discussion of $r_s(Y)$.

\subsection{Properties} \label{sec:3.3}
 We now give several important properties of instanton-type complexes and the $r_s$-invariant.
 
\subsubsection{Local maps} In Section~\ref{sec:5.4}, we show that negative-definite cobordisms induces morphisms in instanton Floer homology. Hence it will be useful to have the following:

\begin{lem} \label{under local map filtration}
Let $(\un{C},\un{d}) $ and $(\un{C}',\un{d}')$ be instanton-type complexes and $\lambda \colon \un{C} \rightarrow \un{C}'$ be a local map. 
\begin{itemize}
\item If $z$ is a $\theta$-supported cycle in $\un{C}$, then $\lambda(z)$ is a $\theta$-supported cycle in $\un{C}'$. 
\item If $z$ is a $\theta$-supported $[r, s]$-cycle in $\un{C}$ and $\lambda$ is a level-$\delta$ local map, then $\lambda(z)$ is a $\theta$-supported $[r + \delta, s + \delta]$-cycle in $\un{C}'$.\footnote{Here, we assume that $r < 0 \leq s$ and $r + \delta < 0 \leq s + \delta$.}
\end{itemize}
\end{lem} 
\begin{proof}
Since $\lambda$ is a chain map, it maps cycles to cycles; the fact that a local map induces an isomorphism $\un{C}/C \rightarrow \un{C}'/C'$ shows that $\lambda(z)$ is also $\theta$-supported. The second part of the claim follows from the fact that if $\lambda$ has level-$\delta$, then it induces a chain map from $\un{C}^{[r,s]}$ to $\un{C}^{[r+ \delta, s+ \delta]}$.
\end{proof}

This immediately yields:

\begin{lem} \label{lem:3.13}
Let $\lambda \colon \un{C} \rightarrow \un{C}'$ be a level-$\delta$ local map. Then for any $s \in [-\infty, 0]$, we have
\[
r_s (\un{C}) - \delta \leq r_{s - \delta}(\un{C}').
\]
\end{lem}

\begin{proof}
This follows from \cref{under local map filtration} and Definition~\ref{def:3.10}. 
\end{proof}

\subsubsection{Tensor products} In order to understand connected sums of $3$-manifolds, we will have to understand tensor products of instanton-type complexes.

\begin{defn} \label{def:3.14}
Let $(\un{C},\un{d}) $ and $(\un{C}',\un{d}')$ be instanton-type complexes. We define their tensor product by taking the usual graded tensor product over $\Z_2[y^{\pm 1}]$, but with an upwards grading shift of three:
\[
\un{C} \otimes \un{C}' = (( \un{C} \otimes \un{C}') [3], \un{d}\otimes \un{d}').
\]
The $\R$- and two-step filtrations on $\un{C} \otimes \un{C}'$ are both given by the usual tensor product of filtrations on $\un{C}$ and $\un{C}'$, respectively. In the latter case, this means that the subcomplex of $\un{C} \otimes \un{C}'$ required by Definition~\ref{def:3.1} is given by
\[
(C \otimes \un{C'} + \un{C} \otimes C')[3].
\]
It is straightforward to check that this makes $\un{C} \otimes \un{C}'$ into an instanton-type complex. Explicitly, if we fix splittings for $\un{C}$ and $\un{C}'$, then the $\theta$-chain in $\un{C} \otimes \un{C}'$ is given by $\theta \otimes \theta'$. 
\end{defn}
Suppose that $z \in \un{C}^{[r, s]}$ and $z' \in \un{C}'^{[r', s']}$ are cycles. Then we may choose representatives (which we also call $z$ and $z'$) in $\un{C}$ and $\un{C}'$ such that 
\[
\deg_I(z) \leq s \text{ and } \deg_I(\un{d}z) \leq r
\]
and
\[
\deg_I(z') \leq s' \text{ and } \deg_I(\un{d}'z') \leq r'.
\]
Then $\deg_I(z \otimes z') \leq s + s'$. Since the boundary of $z \otimes z'$ is given by $(\un{d}z) \otimes z' \pm z \otimes (\un{d}'z')$, we see that the boundary of $z \otimes z'$ has filtration level at most $\max\{r + s', r' + s\}$. Hence $z \otimes z'$ is a cycle in $(\un{C} \otimes \un{C}')^{[\max\{r + s', r' + s\}, s + s']}$.

The behavior of $\theta$-supported cycles under tensor products is straightforward:

\begin{lem} \label{tensor product of theta supp cycle} 
Let $(\un{C},\un{d}) $ and $(\un{C}',\un{d}')$ be instanton-type complexes. 
\begin{itemize}
\item If $z \in \un{C}$ and $z' \in \un{C}'$ are $\theta$-supported cycles, then $z \otimes z'$ is likewise a $\theta$-supported cycle.
\item If $z \in \un{C}$ and $z' \in \un{C}'$ are $\theta$-supported $[r, s]$- and $[r', s']$-cycles, respectively, then $z \otimes z'$ is a $\theta$-supported $[\max\{r + s', r' + s\} , s + s']\text{-cycle}$.
\end{itemize}
\end{lem}
\begin{proof}
It is clear that the tensor product of two $\theta$-supported cycles is a $\theta$-supported cycle, since if $z= \theta + \zeta$ and $z'= \theta' + \zeta'$, we have $z \otimes z' = \theta \otimes \theta' + \theta\otimes \zeta' + \zeta \otimes \theta'$. This implies $z \otimes z'$ is a $\theta$-supported cycle in $\un{C} \otimes \un{C}'$. The filtered version of the claim then follows from the preceding paragraph.
\end{proof}

This immediately yields:
\begin{lem}\label{lem:3.15}
Let $(\un{C},\un{d}) $ and $(\un{C}',\un{d}')$ be instanton-type complexes. For any $s$ and $s'$ in $[-\infty, 0]$, we have
\[
r_{s + s'}(\un{C} \otimes \un{C}') \geq \min \{r_{s} (\un{C})  + s' , r_{s'} (\un{C}')  + s\}. 
\]
\end{lem}
\begin{proof}
This follows from \cref{tensor product of theta supp cycle} and Definition~\ref{def:3.10}.
\end{proof}

\subsubsection{Dualization} As discussed in Section~\ref{sec:3.1}, Definition~\ref{def:3.1} is motivated by counting flow lines out of the reducible connection $\theta$. One can instead take into account flow lines \textit{into} the reducible connection via the map $D_1$ of \cite[Section 7.1]{Do02}. There are thus actually two flavors of instanton-type complexes. When clarity is needed, we refer to a complex as in Definition~\ref{def:3.1} as a ($D_2$-) instanton-type complex. In contrast:

\begin{defn}\label{defc-ubar}
Let $(\ov{C},\ov{d})$ be a $\Z$-graded, $\R$-filtered, finitely-generated free chain complex over $\Z_2[y^{\pm1}]$. We say $(\ov{C},\ov{d})$ is a ($D_1$-) {\it instanton-type chain complex} if it is equipped with a subcomplex isomorphic to $\Z_2[y^{\pm 1}]$. We think of this as again defining a two-step filtration $\Z_2[y^{\pm 1}] \subseteq \un{C}$; we denote the quotient complex by $C = \ov{C}/\Z_2[y^{\pm 1}]$.
\end{defn}

In this situation, we may again choose a splitting of graded, filtered $\Z_2[y^{\pm 1}]$-modules
\[
\ov{C} = C \oplus \Z_2[y^{\pm 1}],
\]
where the first summand should be interpreted as a graded, filtered $\Z_2[y^{\pm 1}]$-module isomorphic to $C$. With respect to this splitting, 
\[
\ov{d} = d + D_1, 
\]
where $d$ is the differential on $C$ and $D_1$ is some filtered map from $C$ to $\Z_2[y^{\pm 1}]$. Note that $\ov{d}$ is zero on $\Z_2[y^{\pm 1}]$. We again write $\theta = (0, 1)$ for the generator of $\Z_2[y^{\pm 1}]$.

It is not hard to see that up to a grading shift, the dual of a ($D_2$-) instanton type complex may be given the structure of a ($D_1$-) instanton type complex, and vice-versa. Although we will not discuss this here, it turns out that (setting aside the dependence on $g$ and $\pi$ discussed in Remark~\ref{rem:3.3}), 
\[
\ov{C}(Y) \cong \un{C}(-Y)^*.
\]
This mirrors the Heegaard Floer setting, where $\HFp(Y)$ is isomorphic to the dual of $\HFm(-Y)$. However, it turns out that $\ov{C}(Y)$ cannot in general be determined from $\un{C}(Y)$. Indeed, the differential in $\ov{C}$ counts flows into the reducible connection on $Y$, while the differential in $\un{C}$ counts flows out of the reducible connection. In this paper, we will almost exclusively work with $\un{C}$, rather than $\ov{C}$. We will thus often consider both $\un{C}(Y)$ and $\un{C}(-Y)$, which do not contain the same information. See Section~\ref{sec:7.1} for further discussion. \\

\subsubsection{Local triviality} We will often be interested in whether a complex $\un{C}$ admits maps to or from the trivial complex. It turns out that the former is trivial, while the existence of a local map in the other direction is characterized by $r_0(\un{C})$: 

\begin{lem}\label{lem:3.18}
Let $\un{C}$ be an instanton-type complex. Then:
\begin{itemize}
\item There always exists a filtered local map from $\un{C}$ to the trivial complex $\Z_2[y^{\pm 1}][-3]$. 
\item There exists a filtered local map from $\Z_2[y^{\pm 1}][-3]$ to $\un{C}$ if and only if $r_0(\un{C}) = \infty$.
\end{itemize}
\end{lem}
\begin{proof}
The first part of the lemma is trivial, as the quotient map
\[
\pi \colon \un{C} \rightarrow \un{C}/C \cong \Z_2[y^{\pm 1}][-3]
\]
constitutes the desired local map. For the second part, assume that we have a filtered local map $\lambda \colon \Z_2[y^{\pm 1}][-3] \rightarrow \un{C}$. Then $\lambda(1)$ is a $\theta$-supported cycle in $\un{C}$ of filtration level zero; this shows that $r_0(\un{C}) = \infty$. Conversely, suppose that $r_0(\un{C}) = \infty$. This easily implies that there is a $\theta$-supported cycle of filtration level zero. (Note that this is in fact a cycle in $\un{C}$, not just a particular quotient of $\un{C}$.) We construct the desired filtered local map  
\[
\lambda \colon \Z_2[y^{\pm 1}][-3] \to \un{C}
\]
by sending the generator of the left-hand side to the aforementioned cycle. 
\end{proof}

\section{Involutive instanton complexes}\label{sec:4}

We now introduce the main invariants discussed in this paper.

\subsection{Involutive complexes}\label{sec:4.1}

In Section~\ref{sec:5.3}, we will show that an involution on $Y$ induces a homotopy involution on the instanton complex $\un{C}(Y)$. (We set aside the dependence on $g$ and $\pi$ for now; see Section~\ref{sec:4.4}.) We make this notion precise in the following definition:

\begin{defn}\label{def:4.1}
An {\it involutive instanton-type complex}  is a pair $(\un{C}, \tau)$, where $\un{C}$ is an instanton-type complex and 
\[
\tau: \un{C} \to \un{C}
\]
is a filtered morphism satisfying the following:
\begin{itemize}
    \item The induced map on the quotient 
    \[
    	\tau \colon \un{C}/C \cong \Z_2[y^{\pm 1}][-3] \rightarrow \un{C}/C \cong \Z_2[y^{\pm 1}][-3] 
    \]
    is the identity. Note here that we use the same isomorphism in the domain and the range afforded by Definition~\ref{def:3.1}.
    \item $\tau$ is a homotopy involution; that is, there exists a $\Z_2[y^{\pm 1}]$-linear map 
    \[
    H : \un{C}_* \to \un{C}_{*+1}
    \]
    such that
    \[ 
    \un{d} H + H \un{d} = \tau^2+ \id. 
    \]
    We require $H$ to be filtered with respect to both $\deg_I$ and the two-step filtration.
    \end{itemize}
Note that since $\tau$ is filtered with respect to $\deg_I$, we also have the morphism 
\[
\tau^{[r, s]} \colon \un{C}^{[r, s]} \rightarrow \un{C}^{[r, s]}. 
\]
This satisfies the same properties as above with $\un{C}$ replaced by $\un{C}^{[r, s]}$. We will sometimes continue to write $\tau$ in place of $\tau^{[r, s]}$ when the meaning is clear.
\end{defn}

If we choose a splitting $\un{C} = C \oplus \Z_2[y^{\pm 1}]$ as in Section~\ref{sec:3.1}, then it is clear that $\tau \theta = \xi + \theta$ for some $\xi \in C$. Maps between involutive instanton-type complexes are exactly as in Section~\ref{sec:3.1}, but with the additional requirement that they homotopy commute with $\tau$. More precisely:

\begin{defn}\label{def:4.2}
Let $(\un{C}, \tau)$ and $(\un{C}', \tau')$ be two involutive complexes. We say that a morphism $f \colon \un{C} \rightarrow \un{C}'$ is \textit{equivariant} if there exists a $\Z_2[y^{\pm 1}]$-linear map 
 \[
    H : \un{C}_* \to \un{C}'_{*+1}
    \]
    such that
    \[ 
    \un{d}' H + H \un{d} = \tau' f +  f \tau.
    \]
We require $H$ to be filtered with respect to the two-step filtration. If $f$ is filtered with respect to $\deg_I$, then we require $H$ to also be filtered with respect to $\deg_I$; if $f$ has level $\delta$, then we require $H$ to also have level $\delta$. 
\end{defn}

This gives the obvious notion of an equivariant local map, an equivariant homotopy equivalence, and so on. For the sake of being explicit, note that a level-$\delta$ equivariant homotopy equivalence is formed by a pair of level-$\delta$ equivariant morphisms $f$ and $g$, such that $f \circ g \simeq \id$ and $g \circ f \simeq \id$ by level-$\delta$ homotopies. This means that the maps $f$ and $g$, as well as the homotopies witnessing 
\[
\tau' f \simeq f \tau, \quad \tau g \simeq g \tau', \quad f \circ g \simeq \id \quad \text{and} \quad g \circ f \simeq \id,
\]
all have level $\delta$.

\subsection{Equivariant $\theta$-supported cycles and the involutive $r_s$-invariant}\label{sec:4.2}

We now turn to the involutive analog of the $r_s$-invariant. We begin with the following:

\begin{defn}\label{def:4.3}
Let $(\un{C},\tau)$ be an involutive instanton-type complex. A chain $z \in \un{C}$ is called an \textit{equivariant $\theta$-supported cycle} if $z$ is a $\theta$-supported cycle in the sense of Definition~\ref{def:3.9} and $[z]$ is fixed by $\tau_*$; that is, 
\[
\tau_* [z]=[z] \in H_*(\un{C}).
\]
 Similarly, a chain $z \in \un{C}^{[r, s]}$ is called an {\it equivariant $\theta$-supported $[r, s]$-cycle} if $z$ is a $\theta$-supported $[r, s]$-cycle and 
 \[
 \tau^{[r, s]}_* [z] = [z] \in H_* (  \un{C}^{[r,s]}) . 
 \]
\end{defn}

Thus an involutive $\theta$-supported cycle is just a $\theta$-supported cycle whose homology class is fixed by $\tau_*$. This leads immediately to an involutive analog of the $r_s$-invariant: 

\begin{defn}\label{def:4.4}
Let $(\un{C}, \tau)$ be an involutive instanton-type complex and $s \in [- \infty, 0]$. Define
\[
r_s ( \un{C}, \tau)= -\inf_{r < 0} \{ r |
\text{ there exists an equivariant $\theta$-supported $[r, -s]$-cycle} \} \in [0,\infty],
\]
with the caveat that if the above set is empty, we set $r_s(\un{C}, \tau) = - \infty$.
\end{defn}

\begin{ex}\label{ex:4.5}
Several examples of involutive instanton-type complexes and graphs of their involutive $r_s$-invariants are given in Figure~\ref{fig:4.1}. In each case, the noninvolutive $r_s$-invariant is $\infty$ for all $s$. 
\begin{itemize}
\item An archetypal example is given in $(a)$, where $\theta$ itself is a $\theta$-supported $[r, -s]$-cycle for all values of $r$ and $s$, but is equivariant only when $r \geq \beta$. 
\item In $(b)$, we present an example which is nontrivial even though $\theta$ itself is fixed by $\tau$. Indeed, if $r < \beta$ then $\theta$ is not a cycle: the $\theta$-supported $[r, -s]$-cycles are instead $\theta - y$ and $\theta - z$, which are interchanged by $\tau$. If $r \geq \beta$, then $\theta$ forms an equivariant $\theta$-supported $[r, -s]$-cycle. 
\item Finally, in $(c)$, we present an example which depends on $s$. If $-s \geq \alpha$, then $x$ is homologically trivial, so $\theta$ forms an equivariant $\theta$-supported $[r, -s]$-cycle. Otherwise, the analysis of $(c)$ is the same as that of $(a)$.
\end{itemize}

\begin{figure}[h!]
\includegraphics[scale = 1]{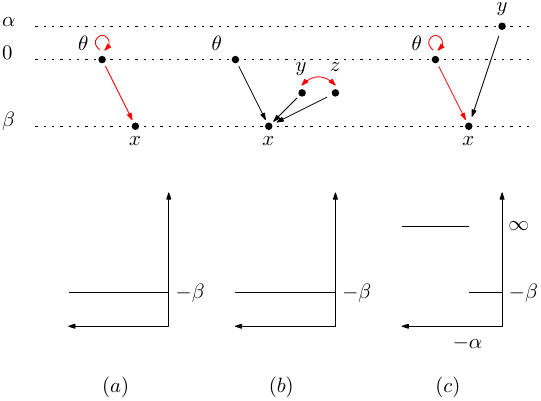}
\caption{Some involutive instanton-type complexes and graphs of their involutive $r_s$-invariants. Dotted lines represent the $\deg_I$-filtration. The differential is given by the black arrows; if no black arrow is drawn, the differential on a given generator is zero. The action of $\tau$ is given by the (sum of) the red arrows; if no red arrow is drawn, the action of $\tau$ on a given generator is the identity.}\label{fig:4.1}
\end{figure}
\end{ex}


\subsection{Properties}\label{sec:4.3}
 The properties of involutive instanton-complexes and the involutive $r_s$-invariant are largely the same as those discussed in Section~\ref{sec:3.3}.
 
 \subsubsection{Local maps} The involutive $r_s$-invariant is functorial under equivariant local maps:

\begin{lem}\label{lem:4.6}
Let $(\un{C}, \tau)$ and $(\un{C}', \tau')$ be two involutive complexes and $\lambda \colon \un{C} \rightarrow \un{C}'$ be an equivariant local map. Then:
\begin{itemize}
\item If $z$ is an equivariant $\theta$-supported cycle in $\un{C}$, then $\lambda(z)$ is an equivariant  $\theta$-supported cycle in $\un{C}'$. 
\item If $z$ is an equivariant $\theta$-supported $[r, s]$-cycle in $\un{C}$ and $\lambda$ has level $\delta$, then $\lambda(z)$ is an equivariant $\theta$-supported $[r + \delta, s + \delta]$-cycle in $\un{C}'$. 
\end{itemize}
\end{lem} 
\begin{proof}
If $[z]$ is $\tau_*$-invariant, then clearly $\lambda_*[z]$ is $\tau'_*$-invariant. The claim then follows from the proof of Lemma~\ref{under local map filtration}.
\end{proof}

This immediately yields:

\begin{lem}\label{lem:4.7}
Let $(\un{C}, \tau)$ and $(\un{C}', \tau')$ be two involutive complexes and $\lambda \colon \un{C} \rightarrow \un{C}'$ be an equivariant level-$\delta$ local map. Then for any $s \in [-\infty, 0]$, we have
\[
r_s (\un{C}, \tau) - \delta \leq r_{s - \delta}(\un{C}', \tau').
\]
\end{lem}

\begin{proof}
This follows from Lemma~\ref{lem:4.6} and Definition~\ref{def:4.4}. 
\end{proof}

\subsubsection{Tensor products} As in Section~\ref{sec:3.3}, we have tensor products of involutive complexes:

\begin{defn}\label{def:4.8}
Let $(\un{C}, \tau) $ and $(\un{C}', \tau')$ be involutive complexes. We define their tensor product by taking the tensor product as in Definition~\ref{def:3.14}, equipped with the obvious tensor product of $\tau$ and $\tau'$:
\[
(\un{C}, \tau) \otimes (\un{C}', \tau') = (( \un{C} \otimes \un{C}') [3], \tau \otimes \tau').
\]
It is straightforward to check that $\tau \otimes \tau'$ satisfies all the conditions of Definition~\ref{def:4.1}.
\end{defn}

We have:

\begin{lem} \label{lem:4.9} 
Let $(\un{C}, \tau) $ and $(\un{C}', \tau')$ be involutive complexes. 
\begin{itemize}
\item If $z \in \un{C}$ and $z' \in \un{C}'$ are equivariant $\theta$-supported cycles, then $z \otimes z'$ is likewise an equivariant $\theta$-supported cycle. 
\item If $z \in \un{C}$ and $z' \in \un{C}'$ are equivariant $\theta$-supported $[r, s]$- and $[r', s']$-cycles, respectively, then $z \otimes z'$ is an equivariant $\theta$-supported $[\max\{r + s', r' + s\} , s + s']\text{-cycle}$.
\end{itemize}
\end{lem}
\begin{proof}
If $\tau_*[z] = [z]$ and $\tau'_*[z'] = [z']$, then $(\tau \otimes \tau')_*[z \otimes z'] = [z \otimes z']$. The claim then follows from the proof of Lemma~\ref{tensor product of theta supp cycle}.
\end{proof}

This immediately yields:
\begin{lem}\label{lem:4.10}
Let $(\un{C},\un{d}) $ and $(\un{C}',\un{d}')$ be instanton-type complexes. For any $s$ and $s'$ in $[-\infty, 0]$, we have
\[
r_{s + s'}(\un{C} \otimes \un{C}') \geq \min \{r_{s} (\un{C})  + s' , r_{s'} (\un{C}')  + s\}. 
\]
\end{lem}
\begin{proof}
This follows from Lemma~\ref{lem:4.9} and Definition~\ref{def:4.4}.
\end{proof}

\subsubsection{Dualization} As discussed in Section~\ref{sec:3.3}, the dual of ($D_2$-) instanton-type complex is not a ($D_2$-) instanton-type complex, but rather a ($D_1$-) instanton-type complex.

\begin{defn}\label{def:4.11}
A {\it involutive $(D_1$-$)$ instanton-type complex}  is a pair $(\ov{C}, \tau)$, where $\ov{C}$ is a $(D_1$-$)$ instanton-type chain complex and $\tau: \ov{C} \to \ov{C}$ is a grading-preserving, $\Z_2[y^{\pm 1}]$-linear chain map satisfying the following:
\begin{itemize}
    \item $\tau$ is filtered with respect to $\deg_I$.
    \item $\tau$ preserves the two-step filtration; that is, it sends the subcomplex $\Z_2[y^{\pm 1}][-3] \subseteq \ov{C}$ to itself. We moreover require that $\tau$ act as the identity on $\Z_2[y^{\pm 1}][-3]$.
    \item $\tau$ is a homotopy involution; that is, there exists a $\Z_2[y^{\pm 1}]$-linear map 
    \[
    H : \un{C}_* \to \un{C}_{*+1}
    \]
    such that
    \[ 
    \ov{d} H + H \ov{d} = \tau^2+\id. 
    \]
    We require $H$ to be filtered with respect to both $\deg_I$ and the two-step filtration.
    \end{itemize}
\end{defn}

Note that in the $D_1$-setting, the action of $\tau$ always fixes $\theta$, whereas in the $D_2$-setting, we may have $\tau \theta = \xi + \theta$ for some $\xi \in C$. Conversely, in the $D_1$-setting the action of $\tau$ may send a chain in $C$ to a chain which is supported by $\theta$, whereas in the $D_2$ setting, the action of $\tau$ preserves $C$. Once again, the dual of a ($D_2$-) involutive complex may be given the structure of a ($D_1$-) involutive complex, and 
\[
\ov{C}(Y) \cong \un{C}(-Y)^*.
\]
as involutive complexes.

\subsubsection{Local triviality} We have the following analog of Lemma~\ref{lem:3.18}:

\begin{lem}\label{lem:4.12}
Let $(\un{C}, \tau)$ be an involutive complex. Then:
\begin{itemize}
\item There always exists a filtered equivariant local map from $(\un{C}, \tau)$ to the trivial complex $(\Z_2[y^{\pm 1}][-3], \id)$. 
\item There exists a filtered equivariant local map from $(\Z_2[y^{\pm 1}][-3], \id)$ to $(\un{C}, \tau)$ if and only if $r_0(\un{C}, \tau) = \infty$.
\end{itemize}
\end{lem}
\begin{proof}
For the first part of the lemma, observe that the quotient map
\[
\pi \colon \un{C} \rightarrow \un{C}/C \cong \Z_2[y^{\pm 1}][-3]
\]
still constitutes the desired local map. The fact that $\tau$ acts as the identity on $\un{C}/C$ shows that $\pi$ is equivariant. The second part of the lemma follows as in Lemma~\ref{lem:3.18}, noting that all cycles in this context are equivariant.
\end{proof}

\subsection{Enriched involutive complexes} \label{sec:4.4}
The formalism discussed so far adequately reflects the analytic situation in the setting where no perturbation of the Chern-Simons functional is needed. In general, however, several technical modifications to the preceding sections are required. 

Firstly, it turns out that the action of $\tau$ may not quite be filtered with respect to $\deg_I$. This necessitates the following mild modification of Definition~\ref{def:4.1} in which $\tau$ and all related homotopies are only required to be level-$\delta$ maps:

\begin{defn}\label{def:4.13}
A {\it level-$\delta$ involutive instanton-type complex}  is a pair $(\un{C}, \tau)$, where $\un{C}$ is an instanton-type complex and 
\[
\tau: \un{C} \to \un{C}
\]
is a level-$\delta$ morphism satisfying the following:
\begin{itemize}
    \item The induced map on the quotient 
    \[
    	\tau \colon \un{C}/C \cong \Z_2[y^{\pm 1}][-3] \rightarrow \un{C}/C \cong \Z_2[y^{\pm 1}][-3] 
    \]
    is the identity. Note here that we use the same isomorphism in the domain and the range afforded by Definition~\ref{def:3.1}.
    \item $\tau$ is a homotopy involution; that is, there exists a $\Z_2[y^{\pm 1}]$-linear map 
    \[
    H : \un{C}_* \to \un{C}_{*+1}
    \]
    such that
    \[ 
    \un{d} H + H \un{d} = \tau^2+\id. 
    \]
    We require $H$ to be filtered with respect to the two-step filtration and have level $\delta$ with respect to $\deg_I$.
    \end{itemize}
A level-zero involutive complex is of course an involutive complex in the previous sense of Definition~\ref{def:4.1}.
\end{defn}

We may still speak of equivariant morphisms between involutive instanton-type complexes of different levels. For example, let $(\un{C}_1, \tau_1)$ be a level-$\delta_1$ involutive complex and $(\un{C}_2, \tau_2)$ be a level-$\delta_2$ involutive complex. A level-$\delta$ equivariant morphism 
\[
f \colon \un{C}_1 \rightarrow \un{C}_2
\]
is still simply a level-$\delta$ morphism in the sense of Definition~\ref{def:3.5} which commutes with $\tau_1$ and $\tau_2$ up to a level-$\delta$ homotopy $H$. Note that the parameter $\delta$ is independent from $\delta_1$ and $\delta_2$, even though the homotopies which make $\tau_1$ and $\tau_2$ into homotopy involutions have levels $\delta_1$ and $\delta_2$. In practice, it will not really be crucial to keep track of the different level shifts; here, we are explicit for the sake of completeness.

More importantly, as discussed in Remark~\ref{rem:3.2}, instead of associating a single involutive complex to $(Y, \tau)$, we will need to associate a sequence of complexes which represents taking a sequence of perturbations converging to zero. The following definition captures this notion:

\begin{defn}\label{def:4.14}
An {\it enriched involutive instanton complex} $\un{\mathfrak{E}}_\tau$ is a sequence of involutive instanton complexes (of varying levels)
\[
(\un{C}_i,\un{d}_i, \tau_i) \text{ of level } \delta_i
\]
for $i \in \Z^{\geq 0}$, together with a sequence of equivariant local maps (of varying levels) 
\[
\psi^j_i : (\un{C}_i,\un{d}_i, \tau_i) \to (\un{C}_j,\un{d}_j, \tau_j) \text{ of level } \delta_{i,j}
\]
for $i, j \in \Z^{\geq 0}$, satisfying the following conditions: 

\begin{itemize}

\item (Clustering condition): The $\deg_I$-levels of homogenous chains in the $\un{C}_i$ cluster around a discrete subset of $\R$. More precisely, we require that there exists a discrete subset $\mathfrak{K} \subseteq \R$ such that for any $\delta>0$, there exists $N$ such that $\forall i >N$ and $\zeta \in \un{C}_i$, 
    \[
    \deg_I(\zeta) \in B_\delta(\mathfrak{K}).
    \]
Note that necessarily $0 \in \mathfrak{K}$.
\item (Composition of maps): Each $\psi_i^i$ is the identity and each $\psi_j^k \circ \psi_i^j$ is homotopic to $\psi^k_i$ via a homotopy of level $\smash{\delta_{i,j,k}}$.

\item (Perturbations converging to zero): We have 
\[
\delta_i \rightarrow 0, \quad \delta_{i,j} \rightarrow 0, \quad \text{and} \quad \delta_{i,j,k} \rightarrow 0
\]
as $i, j, k \rightarrow \infty$. More precisely, for any $\delta > 0$, there exists $N$ such that $\forall i, j, k > N$, we have $\delta_i, \delta_{i,j}, \delta_{i,j,k} < \delta$.
\end{itemize}
\end{defn}
\noindent
Usually we will suppress writing the subscripts on the differentials $\un{d}_i$. There is an unfortunate collision of notation, in that the subscript of $\un{C}_i$ has (up to this point) meant the $\Z$-grading on $\un{C}$, while we now use it to refer to the index of $\un{C}_i$ in a sequence of involutive complexes. However, the former will be comparatively rare moving forward; the distinction will be clear from context.

One can similarly define a \textit{$(D_1$-$)$ enriched complex} $\ov{\mathfrak{E}}_\tau$ by replacing $\un{C}$ with $\ov{C}$. The import of Definition~\ref{def:4.14} is that for sufficiently large indices, all of the maps and homotopies considered in Definition~\ref{def:4.14} are ``almost" filtered.


We generalize the notion of homotopy equivalences and local maps: 
\begin{defn} \label{def:4.15}
An {\it (enriched) homotopy equivalence} between two enriched involutive complexes $\un{\mathfrak{E}}_\tau$ and $\un{\mathfrak{E}}'_\tau$ is a sequence of equivariant homotopy equivalences $f_i$ and $g_i$ making
\[
\un{C}_i \simeq \un{C}'_i \text{ of level } \epsilon_i
\]
for $i \in \Z$, such that $\epsilon_i \rightarrow 0$ as $i \rightarrow \infty$. We require $f_j \psi_i^j \simeq (\psi')_i^j f_i$ via a chain homotopy whose level goes to zero as $i, j \rightarrow \infty$, and similarly for the $g_i$.
\end{defn}

\begin{defn} \label{def:4.16}
An {\it (enriched) local map} between two involutive complexes $\un{\mathfrak{E}}_\tau$ to $\un{\mathfrak{E}}'_\tau$ is a sequence of equivariant local maps 
\[
\lambda_i :\un{C}_i \to \un{C}'_i \text{ of level } \epsilon_i
\]
for $i \in \Z$, such that $\epsilon_i \rightarrow 0$ as $i \rightarrow \infty$. 
If enriched local maps in both directions exist, then we say $\mathfrak{E}_\tau$ and $\mathfrak{E}'_\tau$ are \textit{(enriched) locally equivalent}. Note that we do not require any commutation requirement with the $\psi_i^j$.
\end{defn}

Finally, we define the tensor product and dualization operations. The following are easily checked to be enriched complexes:

\begin{defn} \label{def:4.17}
Let $\un{\mathfrak{E}}_\tau$ and $\un{\mathfrak{E}'}_\tau$ be two enriched complexes. Then we define
\[
\un{\mathfrak{E}}_\tau \otimes \un{\mathfrak{E}}_\tau' = \{ \un{C}_i\otimes  \un{C}'_i, \tau\otimes \tau'  \}
\]
with the following data: 
\begin{itemize}
\item the maps $\psi^j_i \otimes (\psi')^j_i$
\item the discrete set $\mathfrak{K}^\otimes $ is given by 
\[
\mathfrak{K}^\otimes = \{ r_1 + r_2 | r_1 \in \mathfrak{K}, r_2 \in \mathfrak{K}'\}. 
\]
\item various homotopies obtained as tensor products of the homotopies from $\un{\mathfrak{E}}_\tau$ and $\un{\mathfrak{E}'}_\tau$. 
\end{itemize}
\end{defn}

\begin{defn} \label{def:4.18}
Let $\ov{\mathfrak{E}}_\tau$ be a $(D_1$-$)$ enriched complex. Then we define 
\[
\ov{\mathfrak{E}}_\tau^* = \{ (\ov{C}_i)^* , \tau^* \} 
\]
with the following data: 
\begin{itemize}
\item the maps $(\psi^j_i)^*$
\item the discrete set $\mathfrak{K}^* $ is given by 
\[
\mathfrak{K}^* = \{ -r | r \in \mathfrak{K}\}. 
\]
\item various homotopies obtained as duals of the homotopies from $\ov{\mathfrak{E}}_\tau$. 
\end{itemize}
\end{defn}

\subsection{The involutive $r_s$-invariant for an enriched complex} \label{sec:4.5}

We now define the involutive $r_s$-invariant for an enriched complex. Note that if $(\un{C}, \tau)$ is a level-$\delta$ involutive complex with $\delta > 0$, then $\tau$ does \textit{not} induce an automorphism of $\un{C}^{[r, s]}$ and hence Definition~\ref{def:4.4} is not quite valid. We thus need the following important lemma: 

\begin{lem} \label{lem:4.19} 
Let $\un{\mathfrak{E}}_\tau$ be an enriched involutive complex. Fix any $r, s \in [-\infty, \infty] \setminus \mathfrak{K}$. Then for all $i$ sufficiently large, $\tau_i$ induces a homotopy involution
\[
\tau_i^{[r, s]} \colon \un{C}_i^{[r, s]} \rightarrow \un{C}_i^{[r, s]}.
\]
Moreover, the equivariant chain homotopy type of 
\[
(\un{C}_i^{[r,s]}, \tau_i^{[r, s]})
\]
is independent of $i$ for $i$ sufficiently large. 
\end{lem}
\begin{proof}
Let $r \in [-\infty, \infty] \setminus \mathfrak{K}$. Then $\tau_i$ induces a map
\begin{equation}\label{eq:taui}
\tau_i \colon \un{C}_i^{[-\infty, r]} \rightarrow \un{C}_i^{[-\infty, r + \delta_i]}.
\end{equation}
Let $\delta = d(r, \mathfrak{K}) > 0$. By the clustering condition, we know that for $i$ sufficiently large, every homogenous chain in $\un{C}_i$ lies within distance $\delta/2$ of $\mathfrak{K}$ and hence is at least distance $\delta/2$ from $r$. In this situation, it follows that 
\[
\smash{\un{C}_i^{[-\infty, r]} = \un{C}_i^{[-\infty, r - \delta/2]}}. 
\]
If moreover $\delta_i < \delta/2$, then clearly the image of (\ref{eq:taui}) lies in $\smash{\un{C}_i^{[-\infty, r]}}$. For $i$ sufficiently large, we thus have that $\tau_i$ may be considered as a map
\[
\tau_i \colon \un{C}_i^{[-\infty, r]} \rightarrow \un{C}_i^{[-\infty, r]}.
\]
Note that $\tau_i$ is a homotopy involution; by taking $i$ sufficiently large, we may likewise assume that the homotopy in question sends $\smash{\un{C}_i^{[-\infty, r]}}$ to itself. Since 
\[
\un{C}_i^{[r, s]} = \un{C}_i^{[- \infty, s]}/\un{C}_i^{[- \infty, r]},
\]
a similar argument for $s$ gives the desired construction of $\smash{\tau_i^{[r, s]}}$.

The second part of the lemma is proven similarly. From Definition~\ref{def:4.14}, we have equivariant morphisms 
\[
\psi_i^j : \un{C}_i \to  \un{C}_j 
\]
such that $\psi_j^i \psi_i^j \simeq \id$ and $\psi_i^j \psi_j^i \simeq \id$. Although these morphisms and homotopies are not filtered, by taking $i$ and $j$ sufficiently large and applying the same argument as the previous paragraph, we may assume that they in fact map $\un{C}_i^{[r, s]}$ and $\un{C}_j^{[r, s]}$ to themselves/each other.
\end{proof}

For $i$ sufficiently large, we thus obtain a complex $\un{C}_i^{[r, s]}$ with a well-defined involution $\tau_i^{[r, s]}$. (At least under the assumption that $r, s \in [-\infty, \infty] \setminus \mathfrak{K}$.) Moreover, the homotopy type of this pair stabilizes as $i \rightarrow \infty$. We denote this stable value by:

\begin{defn}\label{def:4.20}
Let $\un{\mathfrak{E}}_\tau$ be an enriched involutive complex. For $r, s \in [-\infty, \infty] \setminus \mathfrak{K}$, define the equivariant chain homotopy type
\[
\un{\mathfrak{E}}_\tau^{[r, s]} = (\un{C}_i^{[r, s]}, \tau_i^{[r, s]})
\]
for $i$ sufficiently large as in Lemma~\ref{lem:4.19}.
\end{defn}

It will be helpful to have the following lemma:

\begin{lem}\label{lem:4.21}
Let $\un{\mathfrak{E}}_\tau$ be an enriched involutive complex. Suppose 
\[
[r, r'] \subseteq [-\infty, \infty] \setminus \K \quad \text{and} \quad [s, s'] \subseteq [-\infty, \infty] \setminus \K. 
\]
Then
\[
\un{\mathfrak{E}}_\tau^{[r, s]} \simeq \un{\mathfrak{E}}_\tau^{[r', s']}.
\]
\end{lem}
\begin{proof}
The same argument as in Lemma~\ref{lem:4.19} shows that for $i$ sufficiently large, we have
\[
\un{C}_i^{[- \infty, r]} = \un{C}_i^{[- \infty, r']}.
\]
The analogous observation for $s$ and $s'$ gives the claim.
\end{proof}

Note that $\un{\mathfrak{E}}_\tau^{[r, s]}$ inherits a two-step filtration, since all of our maps are filtered with respect to the two-step filtration on each $\un{C}_i$. If $r < 0 < s$, this means that we still have the notion of an equivariant $\theta$-supported cycle in $\smash{\un{\mathfrak{E}}_\tau^{[r, s]}}$. As before, we use this to define the $r_s$-invariant:

\begin{defn}\label{def:4.22}
Let $\un{\mathfrak{E}}_\tau$ be an enriched involutive complex and $s \in [- \infty, 0]$. If $-s \notin \mathfrak{K}$, define 
\[
r_s( \un{\mathfrak{E}}_\tau) = -\inf_{r < 0 \text{ and } r \notin \mathfrak{K}} \{ r |
\text{ there exists an equivariant $\theta$-supported cycle in }\mathfrak{E}_\tau^{[r, -s]} \} \in [0,\infty],
\]
with the caveat that if the above set is empty, we set $r_s(\un{\mathfrak{E}}_\tau) = - \infty$. For $-s \in \mathfrak{K}$, we define 
\[
r_s( \un{\mathfrak{E}}_\tau) = \lim_{t\to s^-}r_t( \un{\mathfrak{E}}_\tau).
\]
Note that the right-hand side is eventually constant due to Lemma~\ref{lem:4.21}.
\end{defn}

By Lemma~\ref{lem:4.21}, it is clear that $r_s(\un{\E}_\tau)$ is valued in $-\K \cup \{\pm \infty\}$. Moreover, it is not hard to see that (as a function of $s$) $r_s$ is continuous from the right and is constant on each connected component of $[- \infty, 0] \setminus - \K$. Note that due to Lemma~\ref{lem:4.21}, we may exclude any discrete collection of points from the infimum in the definition of $\smash{r_s( \un{\mathfrak{E}}_\tau)}$ without changing its value. The reader may check that if $\un{\E}_\tau$ consists of a constant sequence of level-zero involutive complexes, then Definition~\ref{def:4.22} coincides with Definition~\ref{def:4.4}. We have the analogs of Lemmas~\ref{lem:4.7} and \ref{lem:4.10}:

\begin{lem}\label{monotonicity of involtuive rs}\label{lem:4.23}
If there is an enriched local map from $\un{\mathfrak{E}}_\tau$ to $\un{\mathfrak{E}}'_\tau$, then 
\[
r_s( \un{\mathfrak{E}}_\tau) \leq r_s( \un{\mathfrak{E}}'_\tau)
\]
for every $s \in [-\infty, 0]$.
\end{lem}
\begin{proof}
Fix any $r, s\in [-\infty, \infty] \setminus (\K \cup \K')$. Let $i$ be large enough so that
\[
\E_\tau^{[r, s]} = (\un{C}_i^{[r, s]}, \tau_i^{[r, s]}) \quad \text{and} \quad (\E'_\tau)^{[r, s]} = ((\un{C}'_i)^{[r, s]}, (\tau'_i)^{[r, s]}).
\] 
By the same argument as in the proof of Lemma~\ref{lem:4.19}, by increasing $i$ we may in fact assume that $\lambda_i$ maps
\[
\lambda_i \colon \un{C}_i^{[r, s]} \rightarrow (\un{C}'_i)^{[r, s]}.
\]
Thus, if $\E_\tau^{[r, s]}$ admits an equivariant $\theta$-supported cycle, then $\smash{(\E'_\tau)^{[r, s]}}$ admits an equivariant $\theta$-supported cycle. It follows that if $-s \in [-\infty, \infty] \setminus (\K \cup \K')$, then 
\[
r_s( \un{\mathfrak{E}}_\tau) \leq r_s( \un{\mathfrak{E}}'_\tau).
\]
Here, there is a slight technicality: we have assumed $r \in [- \infty, \infty] \setminus (\K \cup \K')$ throughout, but in the definition of $r_s$ on each side of the above inequality, $r$ ranges over $[-\infty, \infty] \setminus \K$ or $[-\infty, \infty] \setminus \K'$, respectively. However, due to Lemma~\ref{lem:4.21}, we may exclude any discrete collection of points from the infimum in the definition of $\smash{r_s( \un{\mathfrak{E}}_\tau)}$ without changing its value. For $-s \in \mathfrak{K} \cup \mathfrak{K}'$, a limiting argument with $t \rightarrow s^-$ gives the same inequality. 
\end{proof}

\begin{lem}\label{conn sum of rs for inv enriched}\label{lem:4.24}
Let $\un{\mathfrak{E}}_\tau$ and $\un{\mathfrak{E}'}_\tau$ be two enriched complexes. For any $s$ and $s'$ in $[-\infty, 0]$, we have
 \[
r_{s+s'}( \un{\mathfrak{E}}_ \tau \otimes \un{\mathfrak{E}}_ \tau') \geq \min \{r_{s} (\un{\mathfrak{E}}_ \tau) + s' , r_{s'} (\un{\mathfrak{E}}_ \tau')  + s\}. 
\]
\end{lem}
\begin{proof}
First suppose 
\[
r, s \in [-\infty, \infty] \setminus \K, \quad r', s' \in [-\infty, \infty] \setminus \K', \quad \text{and} \quad r + s', s + r', s + s' \in [-\infty, \infty] \setminus \K^\otimes. 
\]
By taking $i$ sufficiently large, we obtain the conclusion of Lemma~\ref{lem:4.9} with $\un{C}$ and $\un{C}'$ replaced by $\un{\E}_\tau$ and $\un{\E}'_\tau$. This gives the desired inequality when $s, s'$, and $s + s'$ are not in $\K \cup \K' \cup \K^\otimes$; otherwise, we apply a limiting argument as in Lemma~\ref{lem:4.23}.
\end{proof}

For completeness, we record the following formal definition:

\begin{defn}\label{def:4.25}
Let
\[
\Theta^\mathfrak{E}_\tau = \{\text{all enriched complexes}\}/\text{local equivalence}.
\]
By Lemma~\ref{lem:4.23}, $r_s$ defines a function
\[
r_s: \Theta^\mathfrak{E}_\tau \to [0, \infty].
\]
The operation of $\otimes$ makes $\Theta^\mathfrak{E}_\tau$ into a commutative monoid, with the identity element being the trivial enriched complex $\mathfrak{T}$ consisting of the constant sequence $\Z_2[y^{\pm}][-3]$ and each $\psi^j_i = \id$. We call $\smash{\Theta^\mathfrak{E}_\tau}$ the \textit{(enriched) local equivalence monoid}.
\end{defn}

Finally, we have the analog of Lemma~\ref{lem:4.12}:

\begin{lem}\label{lem:4.26}
Let $\un{\E}_\tau$ be an enriched complex. Then:
\begin{itemize}
\item There always exists an enriched local map from $\un{\E}_\tau$ to $\mathfrak{T}$.
\item There exists an enriched local map from $\mathfrak{T}$ to $\un{\E}_\tau$ if and only if $r_0(\un{\E}_\tau) = \infty$.
\end{itemize}
\end{lem}

\begin{proof}
The first part of the lemma is obvious, as the sequence of filtered local maps
\[
\pi_i \colon \un{C}_i \rightarrow \un{C}_i/C_i \cong \Z_2[y^{\pm 1}][-3]
\]
gives the claim. For the second part of the lemma, assume that there is an enriched local map from $\mathfrak{T}$ to $\un{\E}_\tau$. This means that we have a sequence of equivariant local maps
\[
\lambda_i \colon \Z_2[y^{\pm 1}][-3] \rightarrow \un{C}_i \text{ of level } \epsilon_i.
\]
Then $\lambda_i(1)$ is an equivariant $\theta$-supported cycle in $\un{C}_i^{[-\infty, \epsilon_i]}$ for each $i$. Now fix any $-s \notin \K$. Since in particular $-s > 0$, there is some $i$ such that $\epsilon_i < -s$. For $i$ sufficiently large, we thus obtain an equivariant $\theta$-supported cycle in
\[
\un{\E}_\tau^{[-\infty, -s]} = \un{C}_i^{[-\infty, -s]}.
\]
This shows $r_s(\un{\E}_\tau) = \infty$ for all such $s$, and hence $r_0(\un{\E}_\tau) = \infty$ by Definition~\ref{def:4.22}.

Conversely, assume $r_0(\un{\E}_\tau) = \infty$. Then we have a sequence $t_k \rightarrow 0^-$ with each $t_k \notin \K$ and $r_{t_k}(\un{\E}_\tau) = \infty$. For each $k$, select an index $i_k$ sufficiently large such that
\[
\un{\E}_\tau^{[-\infty, t_k]} = \un{C}_{i_k}^{[-\infty, t_k]}
\]
as in Definition~\ref{def:4.20}. Then there is an equivariant $\theta$-supported cycle in $\smash{\un{C}_{i_k}^{[-\infty, t_k]}}$. Construct an equivariant local map
\[
\lambda_{i_k} \colon \Z_2[y^{\pm 1}][-3] \rightarrow \un{C}_{i_k} \text{ of level } t_k
\]
by setting $\lambda_{i_k}(1)$ equal to this cycle. This partially defines an enriched local map from $\mathfrak{T}$ to $\un{\E}_\tau$, in the sense that we have defined local maps $\lambda_{i_k}$ for all $k$. Without loss of generality, we may assume the $i_k$ form an increasing sequence $\mathfrak{S}$. To define a local map $\lambda_i$ for every $i$, recall that we have local maps
\[
\psi_i^j \colon \un{C}_i \rightarrow \un{C}_j \text{ of level } \delta_{i,j}.
\]
For arbitrary $i$, let $i_k$ thus be the greatest element of $\mathfrak{S}$ which is less than or equal to $i$ and set
\[
\lambda_i = \psi_{i_k}^i \circ \lambda_{i_k}.
\]
This is an equivariant local map of level $t_k + \delta_{i_k, i}$. Since $t_k \rightarrow 0$ and the $\delta_{i,j} \rightarrow 0$, it is clear that the set of $\lambda_i$ constitute an enriched local map, as desired.
\end{proof}


\section{The analytic construction}\label{sec:5}
In this section, we review the construction of instanton Floer homology and show that a homology sphere equipped with an involution gives rise to an enriched involutive complex. We then discuss equivariant cobordisms and give some results involving connected sums. 

\subsection{Notation} \label{sec:5.1}
We begin with some notation. Our discussion here is from \cite[Section 2]{NST19}.

\subsubsection{Holonomy perturbations} \label{sec:5.1.1}

Let $Y$ by an oriented integer homology $3$-sphere. Denote:
\begin{itemize}
\item the product $SU(2)$-bundle by $P_Y$,
\item the product connection on $P_Y$ by $\theta$; and,
\item the set of $SU(2)$-connections on $P_Y$ by $\A(Y)$.
\end{itemize}
We fix a preferred trivialization of $P_Y$ in order to form the product connection $\theta$. Recall that the \textit{gauge group} $\mathrm{Map}(Y, SU(2))$ is the set of smooth maps from $Y$ into $SU(2)$. We have the usual gauge action of $\mathrm{Map}(Y, SU(2))$ on $\A(Y)$ given by 
\[
a\cdot g = g^{-1} dg + g^{-1} ag. 
\]
The \textit{degree-zero gauge group} is the subgroup $\mathrm{Map}_0(Y, SU(2))$ of $\mathrm{Map}(Y, SU(2))$ consisting of gauge transformations with mapping degree zero. In this paper, we consider the quotient of $\A(Y)$ by the degree-zero gauge group, rather than the full gauge group. Note that former is in $\Z$-to-$1$ correspondence with the latter. Denote:
\begin{itemize}
\item $\smash{\wt{B}(Y) = \A(Y) /\mathrm{Map}_0(Y, SU(2))}$; and,
\item $\wt{B}^*(Y) =\{[a] \in  \wt{B}(Y) \ | \ a \text{ is irreducible} \}$. 
\end{itemize}
Here, a connection is said to be \textit{irreducible} if its stabilizer under the action of $\mathrm{Map}_0(Y, SU(2))$ consists of the two constant gauge transformations $\pm I$.

Given a fixed trivialization of $P_Y$, the {\it Chern-Simons functional on $\A(Y)$} is the map $cs_Y$ from $\A(Y)$ to $\R$ defined by 
\[
cs_Y(a) = \frac{1}{8\pi^2}\int_Y \Tr\left(a\wedge da +\frac{2}{3}a\wedge a\wedge a\right).
\]
It is a standard fact that 
\begin{equation}\label{eq:csgauge}
cs_Y(a \cdot g) - cs(a)= \text{deg} (g)
\end{equation}
for any $g \in \mathrm{Map}(Y,SU(2))$, where $\text{deg} (g)$ is the degree of $g$. Thus $cs_Y$ descends to a map 
\[
cs_Y \colon \wt{B}(Y)\ri \R, 
\]
which we also denote by $cs_Y$. We write $\mathfrak{K}_Y$ for the set of critical values of $cs_Y$.

\begin{defn}
For any $d \in \Z^{\geq 0}$ and fixed $l \gg 2$, define the set of orientation-preserving embeddings of $d$ disjoint solid tori into $Y$:
\[
\mathcal{F}_d=  \left\{ (f_i \colon S^1\times D^2 \hookrightarrow Y )_{i = 1}^d 
\  \middle| \  f_i(S^1 \times D^2) 
\text{ are mutually disjoint} \right\}
\]
and denote by $C^{l}(SU(2)^d,\R)_{\mathrm{ad}}$ 
the set of adjoint-invariant $C^l$-class functions on $SU(2)^d$. The \it{set of holonomy perturbations on $Y$} is defined by
\[
\mathcal{P}(Y)= \bigcup_{d \in \Z^{\geq 0}}\mathcal{F}_d\times C_{\mathrm{ad}}^{l}(SU(2)^d,\R).
\]
\end{defn}

A holonomy perturbation gives rise to a perturbation of $cs_Y$, constructed as follows: 

\begin{defn}
Fix a 2-form $d\mathcal{S}$ on $D^2$ supported in the interior of $D^2$ with $\int_{D^2}d\mathcal{S}=1$. For any $\pi = (f,h)\in \mathcal{P}(Y)$, 
we define the {\it $\pi$-perturbed Chern-Simons functional}
\[
cs_{Y,\pi} \colon\widetilde{\B}(Y) \ri \R
\]
by
\[
cs_{Y,\pi}(a)= cs_Y(a) + h_\pi(a) = cs_Y(a)
+ \int_{x \in D^2} h(\mathrm{hol}_{f_1(-,x)}(a), \dots, \mathrm{hol}_{f_d(-,x)}(a)) d\mathcal{S},
\]
where $\mathrm{hol}_{f_i(-,x)}(a)$ is the holonomy around the loop $s \mapsto f_i(s,x)$ for each $1 \leq i \leq d$. 
\end{defn}

We denote $\|h\|_{C^l}$ by $\|\pi\|$. 


\subsubsection{Gradient-line trajectories} \label{sec:5.1.2}

The gradient-line equation of $cs_{Y,\pi}$ with respect to the $L^2$-metric is given by: 
\begin{align}\label{grad}
 \frac{\partial}{\partial t} a_t=\grad_{a_t} cs_{Y,\pi} = *_{g_Y}F(a_t) + \grad_{a_t} h_\pi ,
\end{align}
where $*_{g_Y}$ is the Hodge star operator. For the precise form of $\grad_{a_t} h_\pi $, see for instance \cite[Section 2.1.3]{NST19}. Let
\[
\widetilde{R}_\pi(Y) = \left\{a \in \widetilde{\B}(Y) \Biggm |F(a)+\grad_a h_\pi  =0  \right\}
 \]
 and 
 \[
 \widetilde{R}_\pi^*(Y) = \widetilde{R}_\pi(Y) \cap \widetilde{\B}^*(Y).
 \]
We furthermore assume that $\pi$ has been chosen such that $\widetilde{R}_\pi(Y) = \widetilde{R}_\pi^*(Y) \cup ([\theta] \times \Z)$, i.e. we shall take a perturbation $h$ so that $h$ is $0$ near a neighborhood of $(\id, \ldots, \id)$. 

\begin{defn}
The solutions of \eqref{grad} correspond to $SU(2)$-connections $A$ on the trivial $SU(2)$-bundle over $\R \times Y$ which satisfy the \textit{perturbed ASD equations}:
\begin{align}\label{pASD}
F^+(A)+ \pi^+(A) = 0,
\end{align}
where $\pi(A)$ is a particular $\su$-valued $2$-form over $Y \times \R$. See \cite[Section 2.1.3]{NST19} for the explicit form of $\pi(A)$. The superscript $+$ is the self-dual part of a 2-form with respect to the product metric on $\R \times Y$; that is, $(1+*)/2$ where $*$ is the Hodge star operator. 
\end{defn}

Fix a holonomy perturbation $\pi$ on $Y$. For $a$ and $b$ in  $\wt{R}^*_\pi(Y)$, define the \textit{moduli space of trajectories} $M_\pi(a,b)$ as follows. Let $A_{a,b}$ be an $SU(2)$-connection on $Y \times \R$ satisfying 
\[
A_{a,b}|_{Y\times (-\infty,-1]}=p^*a \quad \text{and} \quad A_{a,b}|_{Y\times [1,\infty)}=p^*b,
\]
where $p$ is the projection $\R \times Y \ri Y$. Then we define
\begin{align}\label{*}
M_\pi(a,b) =\left\{A = A_{a,b}+c \ \middle | \ c \in \Omega^1(\R \times Y)\otimes \su_{L^2_q}\text{ with } A \text{ satisfying } \eqref{pASD} \right\}/ \G(a,b),
\end{align}
where $\G(a,b)$ is given by 
\[
\G(a,b)=\left\{ g \in \aut(P_{\R \times Y})\subset { \Gamma (\R \times Y; \underline{\End(\mathbb{C}^2)})    }_{L^2_{q+1,\text{loc}}} \ \middle | \ \nabla_{A_{a,b}}(g) \in L^2_q \right\}. 
\]
Here, 
 \[
 \|f\|^2_{L^2_q}=\sum_{0\leq j \leq q} \int_{\R \times Y} |\nabla^j_{A_{a,b}}f|^2
 \]
for $f \in \Om^i(\R \times Y) \otimes \su$ with compact support, where $|-|$ is the product metric on $\R \times Y$, $q \geq 3$ and $\underline{\End(\mathbb{C}^2)}$ is the product bundle whose fiber is $\End(\mathbb{C}^2)$ on $ \R \times Y$. The action of $g \in \G(a,b)$ on the numerator of (\ref{*}) is given by pulling back connections along $g$. 

We also allow $a = \theta$ so long as $b$ is irreducible. In this setting, we construct a slightly different moduli space $\smash{M_{\pi, \delta}(\theta, b)}$ by replacing $A_{a, b}$ with a similarly-defined reference connection $A_{\theta, b}$ and using the $\smash{L^2_{q,\delta}}$-norm in (\ref{*}) instead of $L^2_q$-norm.
The $\smash{L^2_{q,\delta}}$-norm is given by
\[
  \|f\|^2_{L^2_{q,\delta}}=\sum_{0\leq j \leq q} \int_{\R \times Y}e^{\delta \sigma} |\nabla^j_{A_{\theta, b}}f|^2
\]
for some small $\delta > 0$. Here, $\sigma \colon \R \times Y \ri \R$ is a smooth function with 
\[
\sigma(y, t) = 
\begin{cases}
- t &\text{ for } t < 0\\
0 &\text{ for } t > 1.
\end{cases}
\]
For convenience of notation, we continue to write $\smash{M_{\pi}(a, b)}$ in place of $\smash{M_{\pi, \delta}(\theta, b)}$ when $a = \theta$. A similar construction holds when $a$ is a $y^i$-multiple of $\theta$, or when $a$ is irreducible and $b$ is a $y^i$-multiple of $\theta$, although we will not need the latter. In each case, we have an $\R$-action on $\smash{M_{\pi}(a, b)}$ given by translation.


We say that a perturbation $\pi$ is \textit{nondegenerate} if at each critical point of $cs_{Y, \pi}$, there is no kernel in the formal Hessian of $cs_{Y, \pi}$. We say that a perturbation $\pi$ is \textit{regular} if for any pair of critical points $a$ and $b$ and $A \in  M_\pi(a,b)$, the linearization 
\begin{align*}
D_A(F^+(A)+ \pi^+(A)) : \Omega^1(\R \times Y)\otimes \su_{L^2_{q} } \to \Omega^+ (\R \times Y)\otimes \su_{L^2_{q-1}}
\end{align*}
is surjective, with the understanding that the $\smash{L^2_q}$- and $\smash{L^2_{q-1}}$-norms should be replaced with the $\smash{L^2_{q, \delta}}$- and $\smash{L^2_{q-1, \delta}}$-norms if $a = \theta$. 
For more detailed arguments involving holonomy perturbations (such as questions regarding smoothness), see \cite[Proposition 7]{Kr05} and \cite[Proposition D.1]{SaWe08}. 


\subsubsection{ASD moduli spaces} \label{sec:5.1.3}
Now let $W$ be a negative-definite, connected cobordism from $Y$ to $Y'$ with $b_1(W)=0$. Suppose $Y$ and $Y'$ are connected. 
 Throughout, we fix nondegenerate regular holonomy perturbations $\pi$ and $\pi'$ on $Y$ and $Y'$, respectively.
 Let $W^*$ be the cylindrical-end $4$-manifold
\[
W^*= \left(   (-\infty, 0]  \times Y  \right) \cup W \cup \left( [0, \infty)  \times Y' \right).  
\]
Choose a metric $g_{W^*}$ on $W^*$ which coincides with the product metric on $Y \times (-\infty, 0] $ and $Y \times [0, \infty)$. 

\begin{defn}
For any $d \in \Z^{\geq 0}$ and fixed $l \gg 2$, define the set of orientation-preserving embeddings of the set of $d$ disjoint copies of $S^1 \times D^3$ into $W$: 
\[
\mathcal{F}_d(W) =  \left\{ (f_i \colon S^1\times D^3 \hookrightarrow W )_{i = 1}^d 
\  \middle| \  f_i(S^1 \times D^3)
\text{ are mutually disjoint} \right\}.
\]
The \textit{set of holonomy perturbations on $W$} is given by: 
\[
\mathcal{P}^* (W) = \bigcup_{d \in \Z^{\geq 0}} \mathcal{F}_d(W) \times C^l_{\text{ad}}(SU(2), \R)^d \times C^l (\R^d, \R).
\]
\end{defn}

Given $\pi_W \in \mathcal{P}^*(W)$, one can define a perturbation $2$-form $\pi_W(A)$ of the usual ASD equations on $W$, just as in Section~\ref{sec:5.1.2}. See \cite[Equation (14)]{Ta22} for the explicit form of $\pi_W(A)$. More generally, we consider perturbations of the ASD equations on $W^*$ taking into account a choice of holonomy perturbation on the ends:

\begin{defn}
Let $\pi_W$ be a holonomy perturbation on $W$ and $\pi$ and $\pi'$ be holonomy perturbations on $Y$ and $Y'$. We define the \textit{perturbed ASD equations on $W^*$} to be 
\begin{equation}\label{cob}
F^+(A)+ \pi_W^+(A) + \rho_- \cdot \pi^+(A) +   \rho_+ \cdot (\pi')^+(A) =0.
\end{equation}
Here, $\rho_\pm$ are cutoff functions on $Y \times (-\infty, 0]$ and $Y \times [0, \infty)$, respectively, satisfying 
\begin{align*}
\rho_- ( t,y)=
\begin{cases} 
1 \text{ if } t<-1\\ 
0 \text{ if } -\frac{1}{2} <t \leq 0 
\end{cases} 
\text{and} \quad
\rho_+ ( t,y)= 
\begin{cases} 
0 \text{ if }  0 \leq t < \frac{1}{2} \\
1 \text{ if } 1<t.
\end{cases}  
\end{align*} 
The terms $\pi^+(A)$ and $(\pi')^+(A)$ are the gradient-line perturbations discussed in Section~\ref{sec:5.1.2}, applied to the restriction of $A$ over the ends $Y \times (-\infty, 0]$ and $Y \times [0, \infty)$. We refer to the triple $(\pi_W, \pi, \pi')$ as a \textit{holonomy perturbation on $W^*$ with ends $\pi$ and $\pi'$}. We often only write $\pi_W$, leaving $\pi$ and $\pi'$ implicit. Note that the perturbation term $\smash{\pi_W^+(A)}$ is supported in the interior of $W$, while the terms $\smash{\rho_- \cdot \pi^+(A)}$ and $\smash{\rho_+ \cdot (\pi')^+(A)}$ are supported on the ends of $W^*$.
\end{defn}

As in the case of holonomy perturbations on $Y$, we define the norm $\| \pi_W \|$ of $\pi_W$ in terms of the induced perturbation of the ASD equations on $W$.  
When we speak of the norm of a holonomy perturbation on $W^*$, we will mean the maximum of $\| \pi_W \|$, $\| \pi \|$, and $\| \pi' \|$, which are defined as the $C^k$ norms of $ j_W \circ h_W$, $h$ and $h'$ if $\pi_W= (g_W, h_W, j_W)$, $\pi= (f, h)$ and $\pi'=(f', h')$.


Fix a holonomy perturbation on $W^*$. For $a\in \wt{R}_{\pi}(Y)$ and $b\in  \wt{R}_{\pi'}(Y')$, define the \textit{ASD moduli space} by
\[
{M}_{\pi_W} (a,W^* ,b)= \left\{ A = A_{a,b} + c \ \middle | \ c \in \Om^1(W^*) \otimes \su _{L^2_q} \text{ with } A \text{ satisfying } \eqref{cob} \right\} /\G (a,W^*, b).
\]
Here, we have suppressed the data of $\pi$ and $\pi'$ in the subscript. The reference connection $A_{a,b}$ and the group $\G (a,W^*, b)$ are defined in a similar way as in the gradient-line case. As before, we allow $a$ to be a $y^i$-multiple of $\theta$, in which case we must replace the $\smash{L^2_{q}}$-norm with the $\smash{L^2_{q, \delta}}$-norm.

We say that a given moduli space ${M}_{\pi_W} (a,W^* ,b)$ is \textit{regular} if for any point $A \in {M} (a,W^* ,b)$, the linearization 
\begin{align*}
D_{A}( F^+(A)+ \pi_W^+(A) + \rho_- \pi(A)^ + +   \rho_+ \pi'(A)^ + ) : &\\
\Omega^1 (W^*)\otimes \su_{L^2_q}  \to &  \Omega^+ (W^*)\otimes \su_{L^2_{q-1}}
\end{align*}
is surjective. If $a$ is reducible, then we use the weighted $\smash{L^2_{q, \delta}}$-norm instead of the $\smash{L^2_q}$-norm, as in the case of $\R \times Y$. 

We will also need to consider instanton moduli spaces associated to a family of perturbations. Let $B$ be a smooth manifold with boundary or corners; we will usually have $B=[0,1]$ or $B= [0,1]^2$. For the family setting, we usually fix a collection of embeddings ${\bf f} = (f_i) \in \mathcal{F}_d(W)$. 
For a smooth map  ${\bf h}: B\to C^l_{\text{ad}}(SU(2), \R)^d \times C^l (\R^d, \R)$, define the \textit{moduli space with family $B$} by 
\[
{\bf M}_{  \pi_W = ({\bf f}, {\bf h} ) } (a, W^*, b) = \bigcup_{ ({\bf f} ,{\bf h}(b) ) \  b \in B} M_{({\bf f} , {\bf h}(b) )} (a, W^*, b)  , 
\]
where, for each $b \in  B$, we consider the moduli space of the solutions to
\[
F^+(A)+ \pi_{W, b}^+(A) + \rho_- \cdot \pi^+(A) +   \rho_+ \cdot (\pi')^+(A) =0
\]
and $\pi_{W, b}^+$ is the 4-dimensional perturbation determined by $({\bf f} ,{ \bf h}(b) )$. 
We say that $(b, A) \in {\bf M}_{\bf \pi} (a, W^*, b)$ is \textit{regular} if the linearization 
\begin{align*}
D_{(b,A)}( F^+(A)+ \pi_W^+ (A) + \rho_- \pi(A)^ + +   \rho_+ \pi'(A)^ + ) : & \\
\Om^1 (W^*) \otimes \su_{L^2_q} \oplus T_{b} B   \to & \Om^+ (W^*) \otimes \su_{L^2_{q-1}} 
\end{align*}
is surjective. If every point in ${\bf M}_{\bf \pi} (a, W^*, b)$ is regular, we say that ${\bf M}_{\bf \pi} (a, W^*, b)$ is regular. 
Note that the formal dimension of ${\bf M}_{\bf \pi} (a, W^*, b)$ is $\dim B + \ind(a)-\ind(b)$ when $a$ and $b$ are irreducible and  $\dim B -3 -\ind(b)$ when $a$ is reducible. 

\begin{lem}\label{exi of family per} Let $C$ be a positive real number. 
Let $\pi$ and $\pi'$ be nondegenerate regular perturbations on $Y$ and $Y'$, respectively. Let $B$ be a manifold with boundary or corners and 
\[
\pi_{W, \partial}  = ({\bf f_\partial } = (f_i)_{(1 \leq i \leq n)} , {\bf h})  : \partial B\to \mathcal{P}^* (W)
\]
 be a regular perturbation as a family for ${\bf M}_{\pi_{W, \partial}   } (a, W^*, b)$ with respect to critical points $a$ and $b$ of $cs_{Y, \pi}$ and $cs_{Y', \pi'}$ respectively satisfying $\ind(a) - \ind(b) \leq C$. (We assume one of $a$ and $b$ is irreducible. )

Then, there is an extension 
\[
\pi_W =  ({\bf f } = (f_i)_{(1 \leq i \leq m)} ,  {\bf h},  {\bf h} ') : B \to \mathcal{P}^* (W)
\]
such that $\pi_{W, \partial, b }^+(A) =  \pi_{W, b }^+(A)$ for $b \in \partial B$ and the moduli space  ${\bf M}_{\pi_W} (a, W^*, b)$ is regular for given $a$ and $b$ satisfying $\ind(a) - \ind(b) \leq C$, where $n \leq m$.
\end{lem}
The bound $\leq C$ is not essential, but in this paper, we only use finite components of moduli spaces, so we only use this weaker statement. (For the transversality for infinite many moduli spaces using holonomy perturbations, see \cite[Definition 9]{Kr05}.) 
Since every point in the moduli spaces is irreducible, the proof is essentially the same as the poof of the usual transversality argument to prove invariance of chain homotopy type of instanton complexes, see for example \cite[Corollary 14]{Kr05} and \cite[Theorem 8.3]{SaWe08}. Also, the family ASD-moduli spaces with family holonomy perturbations are treated in several preceding studies, for example, see \cite{KM11, Sc15}. 


\subsection{Instanton Floer homology} \label{sec:5.2}
We now show that choosing a holonomy perturbation on $Y$ gives rise to an instanton-type complex in the sense of Definition~\ref{def:3.1}. This construction is due to Donaldson and may be found in \cite[Section 7]{Do02}. The involvement of the Chern-Simons filtration parallels the formalism of enriched instanton knot Floer theory developed in \cite{DS19}. 

\subsubsection{The instanton chain complex.} \label{sec:5.2.1} Let $Y$ be an oriented homology 3-sphere equipped with a Riemannian metric. Fix a nondegenerate regular holonomy perturbation $\pi$ on $Y$. We define the \textit{(irreducible) instanton chain group} to be the formal $\Z_2$-span of points in $\smash{\wt{R}_\pi^*(Y)}$: 
\[
C(Y, \pi) = \mathrm{span}_{\Z_2} \{ [a] \in \wt{R}_{\pi}^*(Y) \}.
\]
The relative index difference $\ind(a) - \ind(b) = \dim M_\pi(a, b)$ remains well-defined after quotienting out by the degree-zero gauge group and hence descends to an absolute $\Z$-grading on $C(Y, \pi)$.
We convert this to an absolute $\Z$-grading by setting $\ind_- \theta = -3$\footnote{There are two abusolute $\Z$-degrees for $\theta$, $\ind_+ (\theta)=0$ and $\ind_- (\theta)=-3$. The dimension of the moduli space $M(a, \theta)$ can be computed by $\ind(a) - \ind_+(\theta) = \ind (a)$.} and defining $\ind_-(\theta) - \ind(b) = \dim M_{\pi}(\theta, b)$. We thus denote $\ind$ by $\deg_\Z$. The differential is given by
\[
da = \sum_{\deg_\Z (a) - \deg_\Z(b) =1} \# (M_{\pi} (a, b)/\R) \cdot b,
\]
extending $\Z_2$-linearly. When we work over $\Z$, we must orient each moduli space $M_{\pi} (a, b)/\R$ to obtain a signed count of points; see \cite[Section 2.2]{NST19} for details.
There is a free action of $\Z_2[y^{\pm 1}]$ on $C(Y, \pi)$ where $y^{\pm 1}$ represents a degree-$(\pm 1)$ gauge transformation; it is well-known that $\deg_\Z(y) = 8$. It is not hard to check that $d$ is in fact $\Z_2[y^{\pm 1}]$-equivariant and hence $C(Y, \pi)$ is $\Z/8\Z$-periodic. Finally, we also have a map $D_2: \Z_2[y^{\pm 1}][-3] \to C_{-4} (Y, \pi)$ given by 
\[
D_2 (1) = \sum_{\deg_\Z(b) =-4} \# (M_{\pi} (\theta , b)/\R) \cdot b. 
\]

\begin{defn} Define
\[
\un{C}(Y, \pi)  = C(Y, \pi) \oplus \Z_2[y^{\pm 1}][-3] \quad \text{and} \quad \un{d} = d + D_2.
\]
We put a filtration on $C(Y, \pi)$ by setting $\deg_I(a) = cs_{\pi} (a)$; it follows from (\ref{eq:csgauge}) that $\deg_I(y) = 1$. Similarly, we define $\deg_I(y^k) = k$ on $\Z_2[y^{\pm 1}][-3]$. This gives an instanton-type complex in the sense of Definition~\ref{def:3.1}, with $C(Y, \pi) \subseteq \un{C}(Y, \pi)$ forming the desired subcomplex.
\end{defn}

Now let $W$ be a negative-definite cobordism from $Y$ to $Y'$ with $b_1(W) = 0$. Choose a regular holonomy perturbation $\pi_W$ on $W^*$ with ends $\pi$ and $\pi'$. We obtain a cobordism map
\[
F_{\pi_W} : \un{C} (Y, \pi) \to \un{C} (Y', \pi')
\] 
as follows. For any $a \in \un{C}(Y, \pi)$, we let
\[
F_{\pi_W} a = \sum_{\substack{b\in \un{C}(Y', \pi') \\ \deg_\Z(a) = \deg_\Z(b)}} \# M_{\pi_W}(a,W^*, b) \cdot b
\]
with the additional convention that
\[
\# M_{\pi_W}(a, [0,1] \times Y ^*, b) =
\begin{cases}
| H_1(W, \Z) |  \operatorname{mod} 2 & \text{ if } a = \theta \text{ and } b = \theta \\
0 & \text{ if } a \neq \theta \text{ and } b = \theta.
\end{cases}
\]
The case where $a = \theta$ and $b$ is irreducible was already covered in the discussion of Section~\ref{sec:5.1.3}. In order to see $F_{\pi_W}$ is a two step filtered chan map, we need to identify $M_{\pi_W}(\theta, [0,1] \times Y ^*, \theta)$ with sufficiently perturbed flat connections over $W$. This is already done in the non-equivariant setting. See \cite[Section 2.2]{Da20} and \cite{NST19} for more precise arguments.  
Note that in the present formalism, we thus have $\smash{F_{\pi_W} \theta = z + | H_1(W, \Z) | \cdot \theta}$ for some $z \in C(Y', \pi')$, while $\theta$ does not appear in the image of any generator in $C(Y, \pi)$. It is thus clear that if $H_1(W, \Z_2) = 0$, then $| H_1(W, \Z)| $ is odd, therefore $\smash{F_{\pi_W}}$ is a local map. It can be shown that the level of $\smash{F_{\pi_W}}$ is a function of $\| \pi_W \|$ (and $\| \pi \|$ and $\| \pi' \|$), which goes to zero as $\| \pi_W \|$ (and $\| \pi \|$ and $\| \pi' \|$) go to zero. Again the norm of a holonomy perturbation on $W^*$, we mean the maximum of $\| \pi_W \|$, $\| \pi \|$, and $\| \pi' \|$, which are defined as the $C^k$ norms of $ j_W \circ h_W$, $h$ and $h'$ if $\pi_W= (g_W, h_W, j_W)$, $\pi= (f, h)$ and $\pi'=(f', h')$.

\subsubsection{Enriched complexes.} \label{sec:5.2.2} Now suppose that we have chosen a sequence of nondegenerate regular holonomy perturbations $\pi_i$ on $Y$ with $\| \pi_i \| \rightarrow 0$. For any pair of perturbations in this sequence, let $\pi^j_i$ be a regular holonomy perturbation on $Y \times [0,1]^*$ with ends $\pi_i$ and $\pi_j$. Treating $Y \times [0, 1]$ as a cobordism from $Y$ to itself, this gives a local map
\[
\psi_i^j: \un{C}(Y, \pi_i) \to \un{C}(Y, \pi_j ).
\]
We may also assume that $\|\pi_i^j\| \to 0$ as $i, j \rightarrow \infty$, so that the level of $\smash{\psi_i^j}$ goes to zero. Setting aside the action of $\tau$, it is shown in \cite[Section 2.3]{NST19} that the sequence $\un{C}(Y, \pi_i)$ together with the maps $\smash{\psi_i^j}$ defines an enriched complex.

We similarly claim that a cobordism $W$ from $Y$ to $Y'$ induces a map of enriched complexes. Let $\pi_i$ and $\pi'_i$ be sequences of nondegenerate regular holonomy perturbations on $Y$ and $Y'$ with $\| \pi_i \|, \| \pi'_i \| \rightarrow 0$. For each $i$, let $(\pi_W)_i$ be a regular holonomy perturbation on $W^*$ with ends $\pi_i$ and $\pi'_i$. Then we obtain a sequence of cobordism maps
\[
F_i = F_{(\pi_W)_i} : \un{C} (Y, \pi_i) \to \un{C} (Y', \pi'_i).
\]
Moreover, we may choose $(\pi_W)_i$ such that $\|(\pi_W)_i \| \rightarrow 0$. Setting aside the action of $\tau$, this defines a map of enriched complexes in the sense of \cref{def:4.16}. 

\subsection{Construction of $\tau$} \label{sec:5.3}

Now suppose that $Y$ is equipped with an orientation-preserving involution $\tau$. Henceforth, we assume that we have chosen a $\tau$-invariant metric on $Y$. Let $\pi_i$ be a sequence of holonomy perturbations on $Y$ as in the previous subsection. For each $i$, let $\pi_i^\tau$ be a holonomy perturbation on $[0, 1] \times Y ^*$ with ends $\pi_i$ and $\tau^* \pi_i$. We construct a chain map
\begin{equation}\label{eq:deftau}
\tau_i: \un{C}(Y, \pi_i)  \to \un{C}(Y, \tau^* \pi_i) = \un{C}(Y, \pi_i) 
\end{equation}
as follows. The map from $\un{C}(Y, \pi_i)$ to $\un{C}(Y, \tau^* \pi_i)$, which by abuse of notation we also denote by $\tau_i$, is just the cobordism map associated to $ [0, 1] \times Y $ with the perturbation $\pi^\tau_i$. Explicitly, for any $a \in \un{C}(Y, \pi_i)$, let
\[
\tau_i a = \sum_{\substack{b \in \un{C}(Y, \tau^* \pi_i) \\ \deg_\Z(a) = \deg_\Z(b)}} \# M_{\pi^\tau_i}(a, Y \times [0,1]^*, b) \cdot b
\]
with the additional convention that
\[
\# M_{\pi^\tau_i}(a, Y \times [0,1]^*, b) =
\begin{cases}
1 & \text{ if } a = \theta \text{ and } b = \theta \\
0 & \text{ if } a \neq \theta \text{ and } b = \theta.
\end{cases}
\]
The case where $a = \theta$ and $b$ is irreducible was already covered in the discussion of Section~\ref{sec:5.1.3}. The sign of the moduli spaces are also given as the usual cobordism map. 
Note that this means $\tau_i \theta = z + \theta$ for some $z \in C(Y, \tau^* \pi_i)$, while $\theta$ does not appear in the image of any generator in $C(Y, \pi_i)$. We then complete (\ref{eq:deftau}) by composing with the tautological identification $\un{C}(Y, \tau^* \pi_i) = \un{C}(Y, \pi_i)$ given by the pullback of connections. If we furthermore assume $\| \pi_i^\tau \| \rightarrow 0$, then we see that $\tau_i$ is a local map whose level goes to zero as $i \rightarrow \infty$

We now show that the sequence $\tau_i$ makes the family $\un{C}(Y, \pi_i)$ into an enriched complex in the sense of Definition~\ref{def:4.14}. Let $\smash{\pi_i^j}$ and $\smash{\psi^j_i}$ be defined as in the previous subsection. We claim:

\begin{lem}\label{ex prf of enriched local inv str}
The following hold:
\begin{itemize}
\item [(i)]
For each $i$ and $j$, we have
\[
\psi^j_i \tau_i \simeq \tau_j \psi^j_i
\]
via a chain homotopy $H_i^j$ of level $\delta_{i, j}$, where $\delta_{i, j} \rightarrow 0$.
\item[(ii)] For each $i$, we have
\[
\tau_i^2 \simeq \id
\]
via a chain homotopy $H_i$ of level  $\delta_i$, where $\delta_i \rightarrow 0$.
\end{itemize}
\end{lem}

\begin{proof}
Let $a \in \un{C}(Y, \pi_i)$ and $c \in \un{C}(Y, \pi_j)$. The coefficient of $c$ in $(\tau_j \psi^j_i)(a)$ is easily seen to be the number of points in the product
\begin{equation}\label{eq:581}
\bigcup_{b \in \un{C}(Y, \pi_j) } M_{\pi_i ^j} (a,Y \times [0,1]^*, b) \times  M_{\pi_j^\tau} (b, Y \times [0,1]^* , \tau^*c),
\end{equation}
where $\tau^* c$ is the pullback of $c$. Likewise, the coefficient of $c$ in $(\psi^j_i \tau_i)(a)$ is easily seen to be the number of points in the product
\begin{equation}\label{eq:582}
\bigcup_{b \in \un{C}(Y, \tau^* \pi_i)} M_{\pi_i ^\tau} (a,Y \times [0,1]^*, b) \times  M_{\pi_i^j} (\tau^* b, Y \times [0,1]^* , c),
\end{equation}
where $\tau^* b$ is the pullback of $b$. Note that we clearly have a bijection between 
\[
M_{\pi_i^j} (\tau^* b, Y \times [0,1]^* , c) \quad \text{and} \quad M_{\tau^* \pi_i^j} (b, Y \times [0,1]^* , \tau^*c)
\]
by taking the pullback of the entire moduli space along $\tau \times \id$ on $Y \times [0, 1]$, which by abuse of notation we also denote by $\tau^*$. Hence (\ref{eq:582}) is in bijection with
\begin{equation}\label{eq:583}
\bigcup_{b \in \un{C}(Y, \tau^* \pi_i)} M_{\pi_i ^\tau} (a,Y \times [0,1]^*, b) \times  M_{\tau^* \pi_i^j} (b, Y \times [0,1]^* , \tau^*c).
\end{equation}
By gluing theory, (\ref{eq:581}) and (\ref{eq:583}) are in bijection with 
\[
M_{\pi_i^j \# \pi_j ^\tau} (a,Y \times [0,1]^*, \tau^* c) \quad \text{ and } \quad M_{\pi_i^\tau \# \tau^* \pi_i^j} (a,Y \times [0,1]^*, \tau^*c),
\]
respectively, where the perturbations $\pi_i^j \# \pi_j ^\tau$ and $\pi_i^\tau \# \tau^* \pi_i^j$ are defined by 
\[
\pi_i^j \# \pi_j ^\tau|_{Y \times (-\infty , -T_0]} = \pi_i^j \quad \text{and} \quad \pi_i^j \# \pi_j ^\tau|_{Y \times [T_0, \infty)}  = \pi_j ^\tau
 \]
 and
\[
\pi_i^\tau \# \tau^* \pi_i^j|_{Y \times (-\infty , -T_0]} = \pi_i \quad \text{and} \quad  \pi_i^\tau \# \tau^* \pi_i^j|_{Y \times [T_0, \infty)}  = \tau^*\pi_i^j
\]
for sufficiently large $T_0$. Here, we mean (for example) that $\smash{\pi_i^j \# \pi_j ^\tau}$ agrees with a $t$-shifted copy of $\pi_i^j$ for $t \ll 0$ and a $t$-shifted copy of $\pi_j^\tau$ for $t \gg 0$. Strictly speaking, this means that $\smash{\pi_i^j \# \pi_j ^\tau}$ is a holonomy perturbation on some $Y \times [-T, T]^*$, but we continue to write $Y \times [0,1]^*$.


Since both $\smash{\pi_i^j \# \pi_j ^\tau}$ and $\smash{\pi_i^\tau \# \tau^* \pi_i^j}$ have ends $\pi_i$ and $\tau^* \pi_j$, we may take a one-parameter family of holonomy perturbations on $Y \times [0, 1]^*$ that interpolates between them and has fixed ends. Denote this family by $\pi(s)$, where
\[
\pi(0) = \pi_i^j \# \pi_j ^\tau \quad \text{and} \quad \pi(1) = \pi_i^\tau \# \tau^* \pi_i^j.
\] 
For $a \in \un{C}(Y, \pi_i)$ and $c \in \un{C}(Y, \pi_j)$, we have the instanton moduli space associated to the family $\{\pi(s)\}$ discussed in Section~\ref{sec:5.1.3}:
\[
{\bf M} _{\{\pi(s)\}}(a,Y \times [0,1]^*,  \tau^* c ) = \bigcup_{t \in [ 0,1] } M_{\pi(s)}(a,Y \times [0,1]^*,  \tau^*c ). 
\]
If $\deg_\Z(a) = \deg_\Z(c)$, then generically this moduli space has the structure of a one-dimensional manifold. Gluing theory tells us that after compactifying and orienting, there are four kinds of endpoints:
\[
\displaystyle \bigcup_{\substack{b \in \un{C}(Y, \tau^* \pi_j) \\ \deg_\Z(b) = \deg_\Z(a) + 1}} \mathbf{M}_{\{\pi(s)\}}(a, Y \times [0,1]^*, b) \times \left(M_{\tau^* \pi_j}(b, \tau^*c) /\R\right)
\]
and
\[
\displaystyle \bigcup_{\substack{b \in \un{C}(Y, \pi_i) \\ \deg_\Z(b) = \deg_\Z(c) - 1}} \left(M_{\tau_i}(a, b)/\R\right) \times \mathbf{M}_{\{\pi(s)\}}(b, Y \times [0,1]^*, \tau^* c) 
\]
together with
\[
M_{\pi(0)}(a, Y \times [0,1]^*, \tau^*c ) \quad \text{and} \quad M_{\pi(1)}(a, Y \times [0,1]^*, \tau^*c). 
\]
After appropriately introducing signs, this gives the equality
\[
\psi^j_i \tau_i + \tau_j \psi^j_i = \un{d} H_i^j + H_i^j \un{d}, 
\]
where the homotopy $H_i^j$ is defined by counting points in $\smash{{\bf M} _{\{\pi(t)\}}(a,Y \times [0,1]^*,  b)}$ whenever $\deg_\Z(b) = \deg_\Z(a) + 1$. We may moreover assume that $\| \{\pi(t)\} \| \rightarrow 0$ as $i, j \rightarrow \infty$. This means that the level of $H_i^j$ goes to zero, as desired.

The proof of the second part of the lemma is similar. Let $a$ and $c$ be in $\un{C}(Y, \pi_i)$. The coefficient of $c$ in $\tau_i^2 a$ is easily seen to be the number of points in the product
\[
\bigcup_{b \in \un{C}(Y, \tau^* \pi_i) } M_{\pi_i ^\tau} (a,Y \times [0,1]^*, b) \times  M_{\pi_i^\tau} (\tau^* b, Y \times [0,1]^* , \tau^* c),
\]
which is in bijection with 
\[
\bigcup_{b \in \un{C}(Y, \tau^* \pi_i) } M_{\pi_i ^\tau} (a,Y \times [0,1]^*, b) \times  M_{\tau^* \pi_i^\tau} (b, Y \times [0,1]^* , c).
\]
The family of homotopies between $\tau_i^2$ and $\id$ is obtained by taking an interpolating family of perturbations between $\pi_i^\tau  \# \tau^* \pi_i^\tau$ and the constant family $\pi_i$. 
\end{proof}

Putting everything together, we obtain:

\begin{defn}
Let $Y$ be an oriented homology 3-sphere equipped with an orientation-preserving involution $\tau$. We obtain an enriched complex $\un{\E}(Y, \tau)$  by taking any sequence of nondegenerate regular holonomy perturbations $\pi_i$ on $Y$ with $\| \pi_i \| \rightarrow 0$ and considering the family 
\[
(\un{C}(Y, \pi_i), \un{d}, \tau_i)
\]
together with the maps $\smash{\psi^j_i}$ of Section~\ref{sec:5.2} and the homotopies of Lemma~\ref{ex prf of enriched local inv str}. The clustering subset is given by the set of critical points $\mathfrak{K}_Y$ of the Chern-Simons functional. We refer to $\un{\E}(Y, \tau)$ as the \textit{enriched involutive complex associated to $(Y, \tau)$}.
\end{defn}

\begin{defn}
Let $Y$ be an oriented homology 3-sphere equipped with an orientation-preserving involution $\tau$. Define
\[
r_s (Y, \tau) =r_s  (\un{\E}(Y, \tau)). 
\]
\end{defn} 

\begin{rem} \label{rem:rangeofrs}
According to Definition~\ref{def:4.22}, the abstract $r_s$-invariant takes values in $[0, \infty] \cup \{- \infty\}$. However, it is not hard to show that in fact $r_s(Y, \tau) > 0$ by essentially the same argument as in \cite{NST19, DISST22}. The point here is that the component of $\tau$ from the reducible to irreducible part of $\un{C}$ strictly decreases the Chern-Simons filtration. The same holds for $D_2$, which should be thought of as the component of $\un{d}$ from the reducible to irreducible part of $\un{C}$. It follows from this that there is always some $\epsilon < 0$ such that $\theta$ itself constitutes an equivariant $\theta$-supported cycle in $\un{C}^{[\epsilon, -s]}$.
\end{rem}

\subsection{Cobordism maps} \label{sec:5.4}

We now show that an equivariant cobordism induces a map of enriched involutive complexes. Let $(W, \wt{\tau})$ be an equivariant negative-definite cobordism from $(Y, \tau)$ to $(Y', \tau')$ with $b_1(W) = 0$. Choose holonomy perturbations
\[
\pi_i^\tau \text{ on } Y \times [0, 1]^* \text{ with ends } \pi_i \text{ and } \tau^* \pi_i 
\]
and
\[
\pi_i^{\tau'} \text{ on } Y' \times [0, 1]^* \text{ with ends } \pi'_i \text{ and } (\tau')^* \pi'_i
\]
which define $\tau_i$ and $\tau'_i$, respectively. In addition, choose a sequence of holonomy perturbations 
\[
\pi_{W, i} \text{ on } W^* \text{ with ends } \pi_i \text{ and }\pi'_i.
\]
This defines a sequence of cobordism maps $F_i = F_{W, i}$. The norms of all perturbations go to zero as $i \rightarrow \infty$.

\begin{lem}\label{involutive local map} We have
\[
F_i \tau_i \simeq \tau'_i F_i
\]
via a homotopy $H_i$ of level $\delta_i$, where $\delta_i \rightarrow 0$ as $i \rightarrow \infty$. 
\end{lem} 
\begin{proof}
The proof is similar to that of \cref{ex prf of enriched local inv str}, so we restrict ourselves to listing the moduli spaces involved. Let $a \in \un{C}(Y, \pi_i)$ and $c \in \un{C}(Y', \pi'_i)$. Then the coefficient of $c$ in $(\tau'_i F_i)(a)$ is given by counting points in
\begin{equation}\label{eq:5101}
\bigcup_{b \in \un{C}(Y', \pi'_i)} M_{\pi_{W, i}} (a,W^*, b) \times  M_{\pi_i^{\tau'}} (b, Y \times [0,1]^* , (\tau')^* c),
\end{equation}
while the coefficient of $c$ in $(F_i \tau_i)(a)$ is given by counting points in
\begin{equation}\label{eq:5102}
\bigcup_{b \in \un{C}(Y, \tau^* \pi_i )} M_{\pi_i^\tau} (a,Y \times [0,1]^*, b) \times  M_{\pi_{W, i}} (\tau^* b, W^* , c).
\end{equation}
Pulling back under the self-diffeomorphism $\wt{\tau}$ on $W$, we have a bijection
\[
M_{\pi_{W, i}} (\tau^* b, W^* , c) = M_{\wt{\tau}^* \pi_{W, i}} (b, W^* , (\tau')^* c).
\]
Thus (\ref{eq:5102}) is in bijection with
\begin{equation}\label{eq:5103}
\bigcup_{b \in \un{C}(Y, \tau^* \pi_i )} M_{\pi_i^\tau} (a,Y \times [0,1]^*, b) \times  M_{\wt{\tau}^* \pi_{W, i}} (b, W^* , (\tau')^*c).
\end{equation}
We stress that this requires the existence of $\wt{\tau}$. As in the proof of Lemma~\ref{ex prf of enriched local inv str}, the moduli spaces (\ref{eq:5101}) and (\ref{eq:5103}) are then in bijection with ASD moduli spaces over $W^*$ corresponding to regular holonomy perturbations
\[
\pi_{W, i} \# \pi_i^{\tau'} \quad \text{and} \quad \pi_i^\tau \# \wt{\tau}^* \pi_{W, i},
\]
respectively. These are defined in a similar way as the perturbations in Lemma~\ref{ex prf of enriched local inv str}. Take a one-parameter family of holonomy perturbations on $W^*$ which interpolates between $\smash{\pi_{W, i} \# \pi_i^{\tau'}}$ and $\smash{\pi_i^\tau \# \wt{\tau}^* \pi_{W, i}}$. The same argument as in Lemma~\ref{ex prf of enriched local inv str} produces the desired homotopy.
\end{proof}

In particular:

\begin{lem}\label{lem:localmaps}
Let $(W, \wt{\tau})$ be an equivariant negative-definite cobordism from $(Y, \tau)$ to $(Y', \tau')$ with $H_1(W, \Z_2) = 0$. Then we obtain an enriched local map
\[
\lambda_W \colon \un{\E}(Y, \tau) \rightarrow \un{\E}(Y', \tau').
\]
\end{lem}
\begin{proof}
Immediate from the definitions.
\end{proof}

\subsection{Connected sums} \label{sec:5.5}
 
We now study the behavior of our complexes under connected sums. Let $(Y, \tau)$ and $(Y', \tau')$ be two equivariant homology spheres such that the equivariant connected sum $(Y \# Y', \tau \# \tau')$ is defined. Then we may form the pair-of-pants cobordism $W^\#$ from $Y\# Y'$ to $Y \cup Y'$ obtained by attaching a $3$-handle along the connected sum sphere of $Y \# Y' $. It is clear that $W^\#$ may be made equivariant. 
For 4-manifolds with disconnected boundary, we can also similarly define ASD-moduli spaces as in the discussion in \cref{sec:5.1.3}. However, to obtain a local map, we to need care about the number of components, see \cref{local pants}. 
 
\begin{lem}\label{connected sum local map }
Let $(Y, \tau)$ and $(Y', \tau')$ be two equivariant homology spheres such that the equivariant connected sum $(Y \# Y', \tau \# \tau')$ is defined. Then we have enriched local maps: 
\[
\ov{\lambda}_{W^\#} : \ov{\mathfrak{E}}(Y\# Y', \tau\# \tau' ) \to  \ov{\mathfrak{E}}(Y,\tau ) \otimes   \ov{\mathfrak{E}}( Y' ,  \tau' )
\]
and 
\[
\un{\lambda}_{W^\#}  :  \un{\mathfrak{E}}(Y,\tau ) \otimes   \un{\mathfrak{E}}( Y' ,  \tau' )
 \to \un{\mathfrak{E}}(Y\# Y', \tau\# \tau' ). 
\]
\end{lem}

\begin{proof}
The existence of the first map was essentially established in \cite{NST19}, in which it is shown that $W^\#$ induces an enriched local map 
\[
\ov{\lambda}_{W^\#}:  \ov{\mathfrak{E}}(Y\# Y') \rightarrow \ov{\mathfrak{E}}(Y ) \otimes   \ov{\mathfrak{E}}( Y'  ). 
\]
The verification that $\ov{\lambda}_{W^\#}$ is homotopy equivariant is essentially the same as the proof of \cref{involutive local map}. The second map is obtained as the dual of the first map after reversing orientation of the boundaries. 
\end{proof}

\begin{rem}\label{local pants}
Note that we do \textit{not} have a local map 
\[
\overline{\lambda}_{W^\#}  : \overline{\mathfrak{E}}(Y,\tau ) \otimes   \overline{\mathfrak{E}}( Y' ,  \tau' ) \to   \overline{\mathfrak{E}}(Y\# Y', \tau\# \tau' ) . 
\]
This is due to certain subtleties involving the notion of a reducible connection in the case of a disconnected manifold. 
 The essential reason why we cannot have a local map is: the formal dimension of the reducible solution on the upside-down cobordism $(W^\#)^\dagger$ of $W^\#$ is $-6$, although the dimension of the stabilizer is $3$. Therefore, the moduli space is obstructed in this case and there is no natural map associated with the upside-down cobordism $(W^\#)^\dagger$. 
It turns out that this allows for good behavior for cobordism maps with domain $\overline{\mathfrak{E}}(Y,\tau ) \otimes   \overline{\mathfrak{E}}( Y' ,  \tau' )$, but not range $\overline{\mathfrak{E}}(Y,\tau ) \otimes   \overline{\mathfrak{E}}( Y' ,  \tau' )$. 
\end{rem}

\section{Properties of $r_s(Y, \tau)$} \label{sec:6}

We now prove Theorem~\ref{thm:1.1}, together with several other properties of the involutive $r_s$-invariant. We then explain the connection between the action of $\tau$ and the effect of the cork twist on the Donaldson polynomial. 

\subsection{Proof of Theorem~\ref{thm:1.1}} \label{sec:6.1}

Having constructed the involutive $r_s$-invariant, it remains to establish its claimed behavior under negative-definite cobordisms. For the convenience of the reader, we recall: \\
\\
\noindent
\textbf{Theorem 1.1.} \textit{
Let $Y$ be an oriented integer homology $3$-sphere and $\tau$ be a smooth, orientation-preserving involution on $Y$. For any $s \in [-\infty, 0]$, we define a real number 
\[
r_s(Y, \tau) \in (0, \infty]
\]
which is an invariant of the diffeomorphism class of $(Y, \tau)$. Moreover, let $(W, \wt{\tau})$ be an equivariant negative-definite cobordism from $(Y, \tau)$ to $(Y', \tau')$ with $H_1(X, \Z_2)=0$. Then
\[
r_s(Y,\tau) \leq r_s(Y', \tau'). 
\]
If $r_s(Y, \tau)$ is finite and $W$ is simply connected, then in fact
\[
r_s(Y,\tau) < r_s(Y', \tau').
\]}
\begin{proof}
The first inequality is straightforward. Since there is an equivariant negative-definite cobordism $(W, \wt{\tau})$ from $(Y, \tau)$ to $(Y', \tau')$ with $H_1(W; \Z_2)=0$, 
\cref{involutive local map} gives an enriched local map from $\un{\mathfrak{E}}(Y, \tau)$ to $\un{\mathfrak{E}}(Y', \tau')$. \cref{monotonicity of involtuive rs} then implies $r_s(Y ,\tau) \leq r_s(Y', \tau')$, as desired. 

The strict inequality is more subtle. The proof is an equivariant analog of \cite[Theorem 3]{Da20} and \cite[Theorem 1.1 (1)]{NST19}; we give a sketch here for the reader. Let $W$ be simply connected and suppose for the sake of contradiction that $r_s(Y,\tau) = r_s(Y', \tau') < \infty$. 

Fix any sequence of analytic data used to define the enriched complexes $\un{\E}(Y, \tau)$ and $\un{\E}(Y', \tau')$ and the enriched map between them. This consists of:

\begin{enumerate}
\item A Riemannian metric $g$ (resp.\ $g'$) on $Y$ (resp.\ $Y'$);
\item A cylindrical-end Riemannian metric on $W^*$, as in Section~\ref{sec:5.1.3};
\item A sequence of nondegenerate regular perturbations $\pi_i$ (resp.\ $\pi'_i$) such that $\|\pi_i\| \to 0$ (resp.\ $\|\pi'_i\| \to 0$). This gives sequences of complexes
\[
\un{C}(Y, \pi_i) \quad \text{and} \quad \un{C}(Y', \pi'_i)
\]
with homotopy involutions $\tau_i$ and $\tau'_i$; and,
\item A sequence of perturbations $\pi_i^W$ of the ASD equations on $W^*$ such that $\|\pi_i^W\| \to 0$. This gives a sequence of local maps 
\[
\lambda_i \colon \un{C}(Y, \pi_i) \rightarrow \un{C}(Y', \pi'_i) \text{ of level } \delta_i
\]
with homotopies $H_i$ (also of level $\delta_i$) such that
\[
\lambda_i \tau_i + \tau'_i \lambda_i = \un{d} H_i + H_i \un{d}
\]
for each $i$, where $\delta_i \rightarrow 0$.
\end{enumerate}

In Lemma~\ref{lem:technical} below, we construct an increasing sequence $(n_k)_{k = 1}^\infty$ such that for each $i$ in $(n_k)_{k = 1}^\infty$, we have a pair of chains
\[
\alpha_i \in C^{[-\infty, s]}(Y, \pi_i) \quad \text{and} \quad \alpha'_i \in C^{[-\infty, s]}(Y', \pi'_i)
\]
for which the following hold:
\begin{enumerate}
\item \label{item:lem611} We have
\[
\lim \deg_I(\alpha_i) = - r_s(Y, \tau) = -r_s(Y', \tau') = \lim \deg_I(\alpha'_i)
\]
along $(n_k)_{k = 1}^\infty$; and,
\item\label{item:lem612} Either 
\[
\alpha'_i = \lambda_i \alpha_i \text{ along } (n_k)_{k = 1}^\infty \quad \text{or} \quad \alpha'_i = H_i \alpha_i \text{ along } (n_k)_{k = 1}^\infty.
\]
\end{enumerate}
For convenience of notation, we restrict our attention to $(n_k)_{k = 1}^\infty$ and speak of the above holding for every $i$.

The proof and motivation behind Lemma~\ref{lem:technical} is rather involved, so we defer it until after the proof of Theorem~\ref{thm:1.1}. However, a modification of the proof of \cite[Theorem 1.1 (1)]{NST19} quickly completes the argument. Indeed, recall that $\alpha_i$ (respectively, $\alpha'_i$) is a linear combination of generators corresponding to critical points of the (perturbed) Chern-Simons functional. By (\ref{item:lem612}) of Lemma~\ref{lem:technical}, each generator $a_i'$ from $\alpha'_i$ must appear in the image $\lambda_i a_i$ or $H_i a_i$ of some generator $a_i$ from $\alpha_i$. In particular, it is then easily checked that
\begin{equation} \label{eq:11proofA}
M(a_i, W^*, a_i') \neq \emptyset
\end{equation}
for some sequence of generators $a_i \in C^{[-\infty, s]}(Y, \pi_i)$ and $a_i' \in C^{[-\infty, s]}(Y', \pi'_i)$ satisfying
\begin{equation} \label{eq:11proofB}
\lim \deg_I(a_i) = -r_s(Y, \tau) \quad \text{and} \quad \lim \deg_I(a_i') = -r_s(Y', \tau'). 
\end{equation}
Here, by $M(a_i, W^*, a_i')$ we either mean the usual ASD-moduli space used to define $\lambda_i$, or we mean the 1-parameter family moduli space used to define $H_i$ in \cref{involutive local map}, depending on which case of (\ref{item:lem612}) is applicable. Note that since $r_s(Y, \tau) = r_s(Y', \tau') > 0$, in the limit $a_i$ and $a'_i$ have filtration level bounded above by a nonzero negative constant. 

For each $i$, now choose any $A_i \in M(a_i, W^*, a'_i)$ using (\ref{eq:11proofA}). Then (\ref{eq:11proofB}), together with the fact that $r_s(Y, \tau) = r_s(Y', \tau')$, shows that the energy of $A_i$ goes to zero. After possibly passing to a subsequence, it follows that the $A_i$ converge to a flat connection with respect to the $\smash{L^2_{k, loc}}$-topology (for arbitrarily large $k$), up to a sequence of gauge transformations obtained by pasting local Coulomb gauge transformations. Denote this limit by $A_\infty$. Then $A_\infty$ is a flat $SU(2)$-connection which we claim is irreducible.

To see this, first observe that we may select a finite number of disjoint neighborhoods $\{U_\beta\}$ of the components of $R(Y)$ in $\B(Y)$ such that:
\begin{itemize}
\item Each $U_\beta$ is open with respect to the $C^\infty$-topology; and, 
\item There exists $\epsilon > 0$ such that any $a \in \B(Y)$ with $\|F_{a}\|_{L^2} \leq \epsilon$ must lie in some $U_\beta$. 
\end{itemize}
The existence of $\{U_\beta\}$ follows from Uhlenbeck's compactness theorem. We denote by $U_{\beta_0}$ the neighborhood containing $\theta$; by passing to a subsequence, we may assume all of the $a_i$ and $a'_i$ from (\ref{eq:11proofA}) lie in a single $U_{\beta}$ for some fixed $\beta \neq \beta_0$. (Note that  $a_i$ and $a'_i$ are certainly irreducible, as they have filtration level which is bounded away from zero.) Next, we claim that there is some $\epsilon' >0$ such that: \begin{itemize}
\item For any $SU(2)$-connection $A$ on $Y \times [0,1]$ in temporal gauge and any sufficiently small perturbation $\pi$, if $A$ satisfies
\[
F^+_A + \pi^+(A) =0 \quad \text{and} \quad \|F_A\|_{L^2} \leq \epsilon', 
\]
then $A|_{\{t\}\times Y} \in \bigcup_\beta U_\beta$ for each $t$. 
\end{itemize}
This claim also follows from Uhlenbeck's compactness theorem. Writing each $A_i$ in temporal gauge, we thus have that 
\[
A_i |_{[-T_0, -T_0 +1]\times Y  } \in \bigcup_\beta U_\beta
\]
In fact, since $A_i$ limits to $a_i$, we have
\[
A_i |_{[-T_0, -T_0 +1]\times Y  } \in \displaystyle U_\beta \quad \text{for} \quad \beta \neq \beta_0.
\]
This implies $A_\infty$ is irreducible on $Y \times (-\infty, 0]$. The holonomy of $A_\infty$ thus gives rise to an irreducible $SU(2)$-representation of $\pi_1(W)$. This contradicts the assumption that $W$ is simply connected.
\end{proof}

We now establish Lemma~\ref{lem:technical}. Fix analytic data as in the proof of Theorem~\ref{thm:1.1}. 



\begin{lem}\label{lem:technical}
Let $(W, \wt{\tau})$ be an equivariant negative-definite cobordism from $(Y, \tau)$ to $(Y', \tau')$ with $H_1(W, \Z_2)=0$. Suppose that there is some $s$ such that 
\[
r_s(Y, \tau) = r_s(Y', \tau') < \infty. 
\]
Then there exists an increasing sequence $(n_k)_{k = 1}^\infty$ such that for each $i$ in $(n_k)_{k = 1}^\infty$, we have chains
\[
\alpha_i \in C^{[-\infty, s]}(Y, \pi_i) \quad \text{and} \quad \alpha'_i \in C^{[-\infty, s]}(Y', \pi'_i)
\]
for which the following hold:
\begin{enumerate}
\item We have
\[
\lim \deg_I(\alpha_i) = - r_s(Y, \tau) = -r_s(Y', \tau') = \lim \deg_I(\alpha'_i)
\]
along $(n_k)_{k = 1}^\infty$; and,
\item Either 
\[
\alpha'_i = \lambda_i \alpha_i \text{ along } (n_k)_{k = 1}^\infty \quad \text{or} \quad \alpha'_i = H_i \alpha_i \text{ along } (n_k)_{k = 1}^\infty.
\]
\end{enumerate}
\end{lem}
\begin{proof}
Roughly speaking, the idea for producing $\alpha_i$ and $\alpha'_i$ is the following: we start by selecting a sequence of chains $\smash{z_i \in \un{C}^{[-\infty, s]}(Y, \pi_i)}$ realizing $r_s(Y, \tau)$. As we explain below, this means that the obstruction to $z_i$ being an equivariant cycle will be a chain $\alpha_i$ lying in some filtration level $\rho_i$, where $\rho_i$ limits to $- r_s(Y, \tau)$ as $i \rightarrow \infty$. We then consider the sequence of chains $z'_i = \lambda_i z_i$. The obstruction to $z_i'$ being an equivariant cycle will be a chain $\alpha'_i$ lying in some filtration level $\rho'_i$; we show that $\rho_i'$ likewise limits to $- r_s(Y', \tau')$. We then prove that $\alpha'_i$ is the image of $\alpha_i$ under $\lambda_i$ or $H_i$, giving the claim. 

We now make this precise. For simplicity, assume $-s \notin \K \cup \K'$; the general case follows by a limiting argument. By Definition~\ref{def:4.22} we have a sequence of negative real numbers $r_i \notin \K$ with
\[
\displaystyle \lim_{i\to \infty} r_i = - r_s(Y, \tau)
\]
from above, such that each stable complex $\un{\E}^{[r_i, -s]}(Y, \tau)$ admits an equivariant $\theta$-supported cycle in the sense of Definition~\ref{def:4.22}. For each $i$, fix a pertubation $\pi_{n_i}$ realizing $\smash{\un{\E}^{[r_i, -s]}(Y, \tau)}$, so that
\[
\un{\E}^{[r_i, -s]}(Y, \tau) = \un{C}^{[r_i, -s]}(Y, \pi_{n_i}).
\]
Without loss of generality, assume that the sequence $n_i$ is increasing with $i$. For each $i$, we have an explicit $\theta$-supported cycle $z_i$ in $\un{C}^{[r_i, -s]}(Y, \pi_i)$ whose homology class is fixed by $\smash{(\tau_i^{[r_i, -s]})_*}$. View $z_i$ as a chain in the complex $\un{C}^{[-\infty, -s]}(Y, \pi_{n_i})$ with
\[
\deg_I(\un{d} z_i) \leq r_i.\footnote{Here and throughout, we use $\deg_I(x)$ when $x$ is not homogenous to refer to the filtration level of $x$.}
\]
To say that the homology class of $z_i \in \un{C}^{[r_i, -s]}(Y, \pi_i)$ is fixed by $\smash{(\tau_i^{[r_i, -s]})_*}$ means that we have $h_i$ and $\xi_i$ in $\un{C}^{[-\infty, -s]}(Y, \pi_{n_i})$ such that 
\begin{equation}\label{eq:xi}
\un{d} h_i = (z_i + \tau_i z_i) + \xi_i \quad \text{with} \quad \deg_I(\xi_i) \leq r_i. 
\end{equation}
Let 
\[
\rho_i = \max\{ \deg_I(\un{d} z_i), \deg_I(\xi_i) \}.\footnote{By convention, $\deg_I(0) = - \infty$.} 
\]
We claim that
\begin{equation}\label{eq:6.1a}
\limsup \rho_i = -r_s(Y, \tau).
\end{equation}
Indeed, suppose not. Then for some positive $\delta > 0$, there exists an infinite sequence of the $z_i$ with $\deg_I(\un{d}z_i) \leq -r_s(Y, \tau) - \delta$ and $\deg_I(\xi_i) \leq -r_s(Y, \tau) - \delta$. It is straightforward to check that this produces an equivariant $\theta$-supported cycle in the stable complex $\un{\E}^{[-r(Y, \tau) - \delta, -s]}(Y, \tau)$, contradicting the infimum in Definition~\ref{def:4.22}. For notational convenience, henceforth we assume that $n_i = i$ and we restrict to a subsequence so that in fact $\lim \rho_i = -r_s(Y, \tau)$.

Now let
\[
z'_i = \lambda_i z_i.
\]
Then we may compute
\[
\un{d} z'_i = \un{d} \lambda_i z_i = \lambda_i \un{d} z_i
\]
and
\begin{align*}
z'_i + \tau'_i z'_i &= \lambda_i z_i + \tau'_i \lambda_i z_i \\
&= \lambda_i z_i + \lambda_i \tau_i z_i + (\un{d} H_i + H_i\un{d}) (z_i) \\
&= \un{d} H_i z_i + \left( \lambda_i \xi_i +  H_i \un{d}z_i \right).
\end{align*}
If we write
\begin{equation} \label{eq:6.1b}
\xi'_i = \lambda_i \xi_i + H_i \un{d} z_i, 
\end{equation}
then the resulting equation $z'_i + \tau'_i z'_i = \un{d} H_i z_i + \xi'_i$ is analogous to the defining equation (\ref{eq:xi}) for $\xi$. Define
\[
\rho'_i = \max\{ \deg_I(\un{d} z'_i), \deg_I(\xi'_i) \}. 
\]
The same argument as in the previous paragraph then shows that $\limsup \rho'_i \geq - r_s(Y', \tau')$. On the other hand, $\rho'_i \leq \rho_i + \delta_i$. Since $\delta_i \rightarrow 0$, this shows that $\limsup \rho'_i \leq \limsup \rho_i = -r_s(Y, \tau) = -r_s(Y', \tau')$. Hence
\begin{equation} \label{eq:6.1c}
\limsup \rho'_i = -r_s(Y', \tau').
\end{equation}
For notational convenience, we again pass to a subsequence so that $\lim \rho'_i = -r_s(Y', \tau')$.

With (\ref{eq:6.1c}) in hand, observe that up to passing to a further subsequence, we either have
\begin{enumerate}
\item \label{item:61alt1} $\lim \deg_I(\un{d} z_i') = -r_s(Y', \tau')$; or,
\item \label{item:61alt2} $\lim \deg_I(\xi_i') = -r_s(Y', \tau')$.
\end{enumerate}
Suppose that (\ref{item:61alt1}) is the case. Since
\[
\deg_I(\un{d} z_i') = \deg_I (\lambda_i \un{d} z_i) \leq \deg_I(\un{d} z_i) + \delta_i
\]
with $\delta_i \rightarrow 0$, the fact that (\ref{eq:6.1a}) and (\ref{eq:6.1c}) are equal shows that $\lim \deg_I(\un{d} z_i) = -r_s(Y, \tau)$. Now suppose (\ref{item:61alt2}) is the case. Again up to passing to a subsequence, examining (\ref{eq:6.1b}) we see that either
\[
\lim \deg_I ( \lambda_i \xi_i) = -r_s(Y', \tau') \quad \text{or} \quad \lim \deg_I (H_i \un{d} z_i) = -r_s(Y', \tau').
\]
In the former case, the same argument as in (\ref{item:61alt1}) shows that $\lim \deg_I( \xi_i ) = - r_s(Y, \tau)$. In the latter, the fact that $H_i$ is of level $\delta_i$ shows that $\lim \deg_I( \un{d} z_i ) = -r_s(Y, \tau)$. 

In either case, we obtain the chains claimed in the statement of the lemma. That is, there is some increasing sequence $(n_k)_{k = 1}^\infty$ such that for all $i$ in $(n_k)_{k = 1}^\infty$, we may define chains $\alpha_i \in \un{C}(Y, \pi_i)$ and $\alpha_i \in \un{C}(Y', \pi'_i)$ such that
\[
\lim \deg_I(\alpha_i) = - r_s(Y, \tau) = -r_s(Y', \tau') = \lim \deg_I(\alpha'_i)
\]
and
\[
f_i(\alpha_i) = \alpha'_i
\]
along $(n_k)_{k = 1}^\infty$, where $f_i$ is formed by counting instantons on $W^*$. Indeed, either $\alpha_i = \un{d}z_i$ and $\alpha'_i = \lambda_i \un{d}z_i$; or $\alpha_i = \xi_i$ and $\alpha'_i = \lambda_i \xi_i$; or $\alpha_i = \un{d}z_i$ and $\alpha'_i = H_i \un{d}z_i$. The reader may refer to Figure~\ref{fig:11cases} for a schematic depiction of these three cases.
\end{proof}

\begin{ex}
Figure~\ref{fig:11cases} gives a schematic depiction of the three possibilities in the proof of Lemma~\ref{lem:technical}. In each case, we display two complexes $C$ and $C'$ with the same involutive $r_s$-invariant, together with examples of maps $\lambda$ and $H$ between them. 
\begin{itemize}
\item In $(a)$, the obstruction to $z$ being a cycle is given by $\un{d} z = x$. The obstruction to $z'$ being a cycle is likewise given by $\un{d} z' = x'$; and we have $\lambda x = x'$. 
\item In $(b)$, the obstruction to $z$ being equivariant is given by $z + \tau z = y$. The obstruction to $z'$ being equivariant is likewise given by $z' + \tau' z' = y'$; and we have $\lambda y = y'$. For simplicity, we suppose $h = 0$ in (\ref{eq:xi}), so that $\xi = z + \tau z$.
\item In $(c)$, we have displayed the most complicated situation. Here, the obstruction to $z$ being an equivariant cycle is given by $\un{d} z = x$, while the obstruction to $z'$ being an equivariant cycle is given by $z' + \tau' z' = y'$. Unlike in $(a)$ and $(b)$, however, these two cycles are not related by $\lambda$, but rather by $H$. Note that in this example, $\lambda$ is not $\tau$-equivariant, as
\[
\lambda \tau z = \lambda (z + y) = z' \quad \text{while} \quad \tau' \lambda z = \tau' z' = z' + y'.
\]
These two are equal only up to addition of the term $(\un{d} H + H \un{d})(z) = H \un{d} z = H x = y'$.
\end{itemize}

\begin{figure}[h!]
\includegraphics[scale = 1]{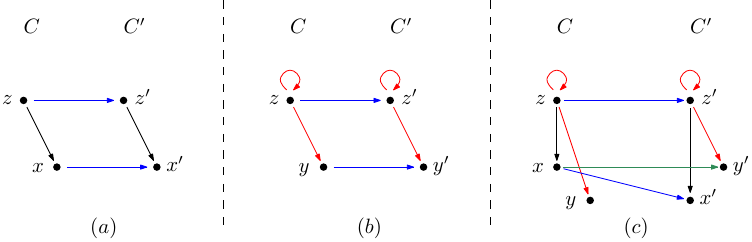}
\caption{Schematic depiction of the possibilities in the proof of Lemma~\ref{lem:technical}. Generators are arranged by height according to their filtration level. The differential is given by the black arrows; if no black arrow is drawn, the differential on a given generator is zero. The action of $\tau$ is given by the (sum of) the red arrows; if no red arrow is drawn, the action of $\tau$ on a given generator is the identity. The map $\lambda$ is in blue and the map $H$ is in green.}\label{fig:11cases}
\end{figure}
\end{ex}

We end this subsection by establishing several other properties of the involutive $r_s$-invariant. We first record the following straightforward corollary of Theorem~\ref{thm:1.1}:

\begin{thm}[Equivariant definite bounding]\label{thm:definitebounding}
Let $(Y, \tau)$ be an equivariant homology sphere. Suppose that $(Y, \tau)$ bounds an equivariant definite manifold $(W, \tilde{\tau})$ with $H_1(W, \Z_2) = 0$.
\begin{itemize}
\item If $W$ is positive definite, then $r_s(-Y, \tau) = \infty$ for all $s$.
\item If $W$ is negative definite, then $r_s(Y, \tau) = \infty$ for all $s$.
\end{itemize}
In particular, suppose that $(Y, \tau)$ is obtained by $1/n$-surgery on an equivariant knot.
\begin{itemize}
\item If $n > 0$, then $r_s(-Y, \tau) = \infty$ for all $s$.
\item If $n < 0$, then $r_s(Y, \tau) = \infty$ for all $s$.
\end{itemize}
\end{thm}

\begin{proof}
Assume $W$ is negative definite. Isotope $\tilde{\tau}$ so that it fixes a ball in the interior of $W$; puncturing $W$ then gives an equivariant cobordism from $(S^3, \id)$ to $(Y, \tau)$. It is easily checked that $r_s(S^3, \id) = \infty$ for all $s$. Applying Theorem~\ref{thm:1.1} then establishes the claim for $r_s(Y, \tau)$ and reversing orientation gives the positive-definite case. The last two claims then follow from applying Lemma~\ref{lem:knotsB}.
\end{proof}

We now have the connected sum inequality:

\begin{thm}[Connected sum inequality]\label{thm:connectedsum}
Let $(Y, \tau)$ and $(Y', \tau')$ be two equivariant homology spheres such that the equivariant connected sum $(Y \# Y', \tau \# \tau')$ is defined. Then for any $s$ and $s'$ in $[-\infty, 0]$, we have
 \[
r_{s+s'}( Y \# Y', \tau \# \tau' ) \geq \min \{r_{s} (Y, \tau) + s' , r_{s'} (Y', \tau')  + s\}. 
\]
In particular,
\[
r_0(Y \# Y', \tau \# \tau') \geq \min \{ r_0(Y, \tau), r_0(Y', \tau') \}.
\]
\end{thm}
\begin{proof}
\cref{connected sum local map } implies there is an enriched local map 
\[
\lambda_{W^\#}^*  :  \un{\mathfrak{E}}(Y,\tau ) \otimes   \un{\mathfrak{E}}( Y' ,  \tau' )
 \to \un{\mathfrak{E}}(Y\# Y', \tau\# \tau' ). 
 \]
Thus, \cref{monotonicity of involtuive rs} implies 
\[
r_s ( \un{\mathfrak{E}}(Y,\tau ) \otimes   \un{\mathfrak{E}}( Y' ,  \tau' ) ) \leq  r_s( \un{\mathfrak{E}}(Y\# Y', \tau\# \tau' ) ). 
\]
Combining this with \cref{conn sum of rs for inv enriched} gives the claim.  
\end{proof}

We also clarify the range of the $r_s$-invariant:

\begin{thm}[Range of $r_s$]\label{thm:range}
The involutive $r_s$-invariant is a nonincreasing function of $s$. Moreover, the range of the involutive $r_s$-invariant is a subset of 
\[
-\{\text{critical values of }cs_Y\} \cup \{\infty\}, 
\]
where $cs_Y$ is the $SU(2)$-Chern-Simons functional.
\end{thm}
\begin{proof}
The fact that $r_s$ is nonincreasing easily follows from the observation that there is a local map 
\[
{\un{\E}_\tau^{[r, s]} \rightarrow \un{\E}_\tau^{[r, s']}} 
\]
whenever $s \leq s'$. For the second claim, note that Lemma~\ref{lem:4.21} implies $r_s(\un{\E}_\tau)$ is valued in $-\K \cup \{\pm \infty\}$. In our case, $\mathfrak{K} = \mathfrak{K}_Y$ is the set of critical values of the $SU(2)$-Chern-Simons functional. Combined with Remark~\ref{rem:rangeofrs}, this completes the proof.
\end{proof}

Finally, we compare the involutive $r_s$-invariant with its usual (nonequivariant) counterpart. While technically we have only defined the latter for abstract instanton-type complexes, it is straightforward to carry out the enriched complex construction of Section~\ref{sec:4.5} to obtain an invariant for homology $3$-spheres, which we denote by $r_s(Y)$.

\begin{thm}[Comparison with the nonequivariant $r_s$-invariant]\label{thm:nonequivariantrs}
Let $r_s(Y)$ be the nonequivariant $r_s$-invariant constructed using Definition~\ref{def:3.10}. Then
\[
r_s(Y, \tau) \leq r_s(Y)
\]
for every $s$. Moreover, if $\tau$ is isotopic to the identity, then
\[
r_s(Y, \tau) = r_s(Y). 
\]
\end{thm}
\begin{proof}
It is clear that if $\tau$ is isotopic to the identity, then $r_s(Y, \tau) = r_s(Y, \id) = r_s(Y)$. Indeed, the isotopy-invariance of the induced action of $\tau$ can easily be shown directly. Alternatively, note that if $\tau_1$ is isotopic to $\tau_2$, then the track of isotopy gives an equivariant homology cobordism from $(Y, \tau_1)$ to $(Y, \tau_2)$. For the more general claim, note that a comparison of Definitions~\ref{def:3.10} and \ref{def:4.4} shows $r_s(\un{C}, \tau) \leq r_s(\un{C})$.
\end{proof}

\begin{rem}\label{rem:differenceinrs}
Note, however, that the $r_s$-invariant constructed using Definition~\ref{def:3.10} is not quite the same as the $r_s$-invariant of \cite[Definition 3.2]{NST19}: the two agree when working over field coefficients. (It is straightforward to carry out the present paper using more general coefficient rings.) To see this, we re-write Definition~\ref{def:3.10} as
\begin{align*}
r_s ( \un{C}) & =  -\inf \{   \deg_I (d \alpha - D_2 (1) ) \ | \ \alpha \in C_*, \deg_I(\alpha) \leq -s \} \\ 
& = \sup \{ \deg_I (f ) \ | \ f \in C^*(-Y), \deg_I(\alpha) \leq -s, d^* f(\alpha) - D_1^* (f) \neq 0 \}\\
& =   \inf \{ r \leq0 \ | \ 0\neq [D_1] \in H^* (C^{\leq -s}(-Y) / C^{\leq r} (-Y) )   \} \\ 
& =   \sup \{ r \leq0 \ | \ 0 =  [D_1] \in H^* (C^{\leq -s}(-Y) / C^{\leq r} (-Y) )   \}.
\end{align*}
In the second line, we have used the fact that we are working with field coefficients to invoke the duality $\smash{(C_*^{\leq s }(Y, \tau) )^* = C^*_{\geq  -s} (-Y, \tau)}$ and $\smash{D_2^* = D_1}$. The last line is the definition of the $r_s$-invariant in \cite[Definition 3.2]{NST19} for $-Y$. However, in \cite{NST19}, the Chern-Simons functional is defined with an additional negative sign in comparison to Section~\ref{sec:5.1.1}. Hence Definition~\ref{def:3.10} agrees with that of \cite[Definition 3.2]{NST19}. 
\end{rem}

\begin{rem}\label{rem:differentcoefficients}
While the formalism of this paper may be carried out over any coefficient ring, in order to obtain a nontrivial equivariant invariant it is best to either use $\Z$ or a field with characteristic equal to the order of $\tau$. For instance, suppose that $\tau$ is of order three and suppose we work over $\Z_2$. Then if $z$ is any $\theta$-supported cycle, we may average over the orbit of $\tau$ to form the equivariant $\theta$-supported cycle $z + \tau z + \tau^2 z$. Hence in this case there is no distinction between the existence of a $\theta$-supported cycle and an equivariant $\theta$-supported cycle.
\end{rem}

\subsection{Relation with the Donaldson invariant} \label{sec:6.2}
We now discuss the relation between the action of $\tau$ and the $SO(3)$-Donaldson polynomial. Let $X$ be a closed, simply-connected oriented 4-manifold with $b^+(X) \geq 3$. Assume $X=X_1 \cup_Y X_2$ along an oriented homology 3-sphere $Y$, where:
\begin{itemize}
    \item[(i)] $X_1$ is a negative-definite 4-manifold with $H_1(X_1; \Z_2)=0$,   
    \item[(ii)] $X_2$ is a 4-manifold with $b^+(X_2)>1$. 
\end{itemize}
Fix a cohomology class $w \in H^2(X; \Z_2 )$; this will correspond to the second Stiefel-Whitney class of an $SO(3)$-bundle. We will assume $w^2 \neq 0 \bmod 4$ in order to guarantee compactness of various moduli spaces. 

Denote the $SO(3)$-Donaldson polynomial invariant by 
\[
\Psi (X) : \Lambda^* H_2(X;\Z) \to \Z.
\]
This is formally defined by 
\[
\Psi (X)(x_1, \cdots , x_n ) =  \int_{\ov{M(X, g, m,w)} } \mu ( x_1) \wedge \cdots \wedge \mu (x_n)  , 
\]
where $\mu : H_2 (X) \to H^2 (\B_{m,w} (X))$ is as in \cite[Section 6.3]{Do02} and $\ov{M(X, g, m,w)}$ is the compactified ASD-moduli space with respect to the $SO(3)$-bundle $P_{m,w}$ satisfying $p_1=m$ and $w_2=w$. Here, $\B_{m,w} (X)$ is the space of $SO(3)$ connections over $P_{m,w}$ divided by gauge transformations and $m$ is chosen so that 
\[
-2m-3(1+b^+(X)) = 2n.
\]
In order to make this definition rigorous, we realize $\mu ( x_1), \ldots,  \mu (x_n)$ as divisors $V(x_1), \ldots , V(x_n)$ in the ASD-moduli space $\ov{M_k(X, g, m,  \omega)}$ and set
\[
 \int_{\ov{M(X, g, m,w)} } \mu ( x_1) \wedge \cdots \wedge \mu (x_n)   = \# \left( V(x_1 ) \cap \cdots \cap V(x_n) \cap \ov{M(X, g, m,w)} \right).
 \]

In this paper, we focus on the case $X=K3$ or $X=3\CP \# 20 \oCP$. 

\begin{rem}
Although we have used $SU(2)$ as the structure group in our definition of instanton Floer homology, one can equally well use $SO(3)$. For integer homology spheres, the groups defined in this way are naturally isomorphic; hence we will sometimes identify $SO(3)$- and $SU(2)$-instanton Floer homology.
\end{rem}

Counting relative ASD moduli spaces over $X_1$ defines a cycle $\un{\psi}(X_1)$ in the chain complex $\un{C}(Y)$, this defines 
\[
  \un{\Psi}(X_1): \Lambda^* H_2(X_1; \Z) \to \un{I}(Y),
\]
where $\un{I}_*(Y)$ is regarded as absolutely $\Z/8\Z$-graded. We similarly have 
\[
\ov{\Psi}(X_2): \Lambda^* H_2(X_2; \Z) \to \ov{I}(Y). 
\]
This gives the pairing formula:

\begin{thm}\cite[Section 7]{Do02}
In the above setting, we have
\[
\langle \un{\Psi}(X_1), \ov{\Psi}(X_2) \rangle = \Psi(X). 
\]
\end{thm}
Note that this pairing formula holds in $\Z$ coefficient, although we will use it in $\Z_2$ coefficient. 

We now assume $b^+(X)=3$. Let us summarize the computations of the $SO(3)$-Donaldson polynomial invariants of $\smash{X=K3\# \oCP}$ or $\smash{X=3\CP \# 20 \oCP}$. 
We focus on the $ SO(3)$-bundle $P_{m,w}$ satisfying 
\[
m=6 \quad \text{and} \quad w^2 = 1 \bmod 4.
\]
In this case, the ASD moduli space $\ov{M(X, g, m,w)} =M(X, g, m,w) $ has dimension zero and is compact. Thus, we do not need to cut down the moduli space with any divisor $V_{x_i}$. We denote by $\Psi (X)(1) \in \Z$ the $SO(3)$-Donaldson invariant, defined as the signed count of points in $M(X, g, m,w)$ with respect to a homology orientation of $X$.
The invariant $\Psi (X)(1) \in \Z$ is called the {\it simple invariant}; see \cite[Section 9.1.1]{DK90}. 

In the simple invariant case, the corresponding relative Donaldson invariants $\un{\Psi}(X_1)$ and $ \ov{\Psi}(X_2)$ are concentrated in a single degree within $\un{I}(Y)$ and $\un{I}(Y)$, respectively. By combining known computations of simple invariants of $K3$ with the blow-up formula \cite[Proposition 9.3.14]{DK90}, we see the following: 
\begin{itemize}
\item(\cite{K91} and \cite[Proposition 9.3.14]{DK90}) Suppose $X= K3 \# \oCP$. Then there is a cohomology class $w\in H^2(X; \Z_2)$ with $w^2 = 1 \bmod 4$ whose Donaldson invariant is computed as
\[
\Psi (X)(1)=  \pm 1 \in \Z. 
\]
\item (\cite[Theorem 9.3.4]{DK90}) Suppose $X= 3\CP \# 20 \oCP $. Then for every cohomology class $w\in H^2(X; \Z_2)$ with $w^2 = 1 \bmod 4$, the simple Donaldson invariant of $X$ is 
\[
\Psi (X)(1)=  0 \in \Z. 
\]
\end{itemize}

Now consider the decomposition 
\[
X = K3 \# \oCP = X_1 \cup_Y X_2,
\]
where $X_1$ is the usual Akbulut cork. Then cutting out and re-gluing $X_1$ via the involution $\tau \colon Y \rightarrow Y$ of \cite{Ak91_cork} gives 
\[
X' = 3\CP \# 20 \oCP = X_1 \cup_{\tau} X_2. 
\]
By the paring formula, we have 
\begin{align*}
\Psi (X')(1) = \langle \tau_* \un{\Psi}(X_1) , \ov{\Psi}(X_2) \rangle =0 \neq  \pm 1 = \langle  \un{\Psi}(X_1) , \ov{\Psi}(X_2) \rangle = \Psi(X)(1).
\end{align*}
Note that $\un{\Psi}(X_1)$ induces a local map from $\un{C}(S^3 )$ to $\un{C}(Y)$. 

\begin{lem}\label{lem:6.11}
The action 
$\tau_* : \un{I}(Y) \to \un{I}(Y) $ induced by tau satisfies the following condition: 
There is a $\theta$-supported element $\psi(X_1)$ and a function $T:  \un{I}(Y) \to \Z_2$  such that $T\tau_* \psi(X_1)= 0$ and $T\psi(X_1)=1$.

\end{lem}

\begin{proof}
Immediate from the pairing formula.
\end{proof}

This trivially gives:

\begin{lem}\label{lem:6.12}
There exists a $\theta$-supported cycle in $\un{I}_{-3}(Y)$ which is not fixed by $\tau_*$.
\end{lem}

\begin{proof}
If not, we would have $\tau_* \psi(X_1) = \psi(X_1)$ in Lemma~\ref{lem:6.11}, which would clearly be a contradiction.
\end{proof}

\section{Examples and applications}\label{sec:7}

In this section, we provide several computations and partial computations of our invariants. We use these to establish the remaining claims of Section~\ref{sec:1}.

\subsection{The Akbulut-Mazur cork} \label{sec:7.1}
Our first (and most fundamental) example is the Akbulut-Mazur cork $Y = S_{+1}(\ov{9}_{46})$. In \cite[Theorem 1]{Sa03}, it is shown that the (irreducible) instanton Floer homology 
\[
I(Y) = H_* (C(Y), d ) 
\]
is isomorphic to $\Z$ in odd gradings and is zero in even gradings. Moreover, since all critical points are known to be nondegenerate, we do not need to take holonomy perturbations to make the Chern-Simons functional Morse. For trajectories, we take small perturbations which are zero near the critical points and make all moduli spaces regular. We claim that the action of $\tau$ on the instanton homology
\[
\un{I}(Y) = H_* (\un{C}(Y), \un{d} ) 
\]
is locally nontrivial. 

To see this, observe that $I(Y)$ is one-dimensional in grading $-3$. Together with the fact that the Fr\o yshov invariant of $Y$ is zero, this shows that $\un{I}(Y)$ is two-dimensional in grading $-3$. Let $x$ be a generator of $\un{I}(Y)$ in grading $-3$ whose image generates the quotient $\un{I}(Y)/I(Y)$ and let $a$ be the class of a cycle in $C(Y) \subset \un{C}(Y)$ with the same grading as $x$. Note that $x$ is the class of a $\theta$-supported cycle. The fact that $\tau_*$ preserves $I(Y)$ and $\tau_*^2 = \id$ shows that $\tau_* a = a$. It follows immediately from Lemma~\ref{lem:6.12} that 
\[
\tau_* x = x + a.
\]
The condition $\tau_*^2 = \id$ then shows that in fact
\[
\tau_* a = a.
\]
This is easily seen to be locally nontrivial just from the action of $\tau$ on homology. Hence:

\begin{lem}\label{cal for Akbulut}
The involutive chain complex $(\un{C}(Y),\tau)$ is not locally trivial. In fact, we have:
\[
r_0(Y, \tau) < \infty \quad \text{and} \quad r_s(-Y, \tau)=\infty.
\] 
\end{lem}

\begin{proof}
By \cref{lem:4.26}, $r_0(Y, \tau)$ characterizes local triviality. Hence the first inequality follows from the discussion at the beginning of this subsection. The claim regarding $r_0(-Y, \tau)$ follows from the fact that $Y$ is $(+1)$-surgery on the strongly invertible knot $\ov{9}_{46}$, combined with Theorem~\ref{thm:definitebounding}.
\end{proof} 

Although not strictly necessary, for completeness we describe the chain complex $\un{C}(Y)$. The irreducible part of this complex was computed by Saveliev in \cite{Sa03} via an analysis of the representation variety of $Y$; see also \cite{RS04}. As shown in \cite[Section 2]{Sa03}, the irreducible representation variety of $Y$ (modulo conjugation) is transversely cut out and consists of six points, denoted by $\beta_1, \beta_2, \alpha_1, \alpha_2, \alpha_3$, and $\alpha_4$. The induced action of $\tau$ on the representation variety interchanges $\beta_1$ and $\beta_2$ and leaves each $\alpha_i$ fixed. Together with the aforementioned computation of $I(Y)$, this shows that one generator of $C(Y)$ (say the generator corresponding to $\alpha_1$) lies in even grading and has
\[
\un{d} \alpha_1 = \beta_1 + \beta_2, 
\]
while all other generators lie in odd grading and have zero differential; see Figure~\ref{fig:7.1}. Hence
\[
\tau \beta_1 = \beta_2, \quad \tau \beta_2 = \beta_1, \quad \text{and} \quad \tau \alpha_i = \alpha_i \text{ for all } i.
\]
Here, we abuse notation by writing $\beta_1$ to mean the generator of $C(Y)$ corresponding to the representation $\beta_1$, and so on. See \cite[Section 9.3]{RS04} for the relevant computation over $\Z$.

We thus have that
\[
\un{C}(Y) = \mathrm{span}\{\theta, \beta_1, \beta_2, \alpha_1, \alpha_2, \alpha_3, \alpha_4\}.
\]
It is somewhat difficult to determine the absolute gradings of these generators, although we know that $\beta_1, \beta_2, \alpha_2, \alpha_3$, and $\alpha_4$ lie in the same mod two grading as $\theta$. It follows from this that either $D_2 \theta = 0$ or $D_2 \theta = \alpha_1$; the latter is impossible as $\alpha_1$ is not a cycle. There are two possibilities for the action of $\tau$ on $\un{C}(Y)$ consistent with our analysis of $\un{I}(Y)$. If some $\alpha_i$ (for $i = 2, 3, 4$) is in the same grading as $\theta$, then
\[
\tau \theta = \theta + \alpha_i.
\]
If $\beta_1$ and $\beta_2$ are in the same grading as $\theta$, then up to filtered homotopy equivalence we have
\[
\tau \theta = \theta + \beta_1.\footnote{To see that the homotopy equivalence is filtered, write $\tau \theta = \theta + m \beta_1 + n \beta_2$. Our previous observation involving the Donaldson invariant implies $m + n \equiv 1 \bmod 2$; in particular, $m \neq n$. It follows that $\tau$ is not literally an involution. The fact that $\tau$ is a homotopy involution in the filtered sense then implies that $\deg_I(\alpha_1) \leq 0$.} 
\]
These two possibilities are displayed in Figure~\ref{fig:7.1}. 

\begin{figure}[h!]
\includegraphics[scale = 1.1]{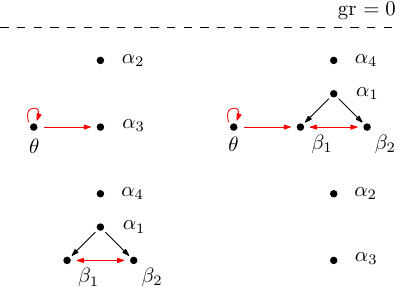}
\caption{Possibilities for the complex of $Y = S_{+1}(\ov{9}_{46})$. Generators are arranged by height according to their homological grading. The differential is given by the black arrows; if no black arrow is drawn, the differential on a given generator is zero. The action of $\tau$ is given by the (sum of) the red arrows; if no red arrow is drawn, the action of $\tau$ on a given generator is the identity.}\label{fig:7.1}
\end{figure}

Finally, we describe the chain complex $\un{C}(-Y)$ to illustrate the fact that $\un{C}(Y)$ and $\un{C}(-Y)$ do not necessarily contain the same information. The generators of $\un{C}(-Y)$ are easily calculated to be
\[
\un{C}(-Y) = \mathrm{span}\{\theta, \beta_1^\vee, \beta_2^\vee, \alpha_1^\vee, \alpha_2^\vee, \alpha_3^\vee, \alpha_4^\vee\}.
\]
However, now the only generator in the same mod two grading as $\theta$ is $\alpha_1^\vee$. Since $-Y$ bounds a homology ball, $D_2 \theta$ is a nullhomologous cycle in $\un{C}(-Y)$. It is easily checked that $D_2 \theta = 0$ from the fact that the Fr\o yshov invariant is zero; see Figure~\ref{fig:7.2}. There are again two possibilities for the action of $\tau$ on $\theta$. If there are no other generators in the same grading as $\theta$, then clearly
\[
\tau \theta = \theta. 
\]
Otherwise, we have $\tau \theta = \theta + n \alpha_1^\vee$ for $n= 0$ or $1$. Note that on the level of homology, we still have $\tau_*[\theta] = [\theta]$; moreover, one can check that this action of $\tau$ is filtered homotopy equivalent to the trivial action on $\theta$.\footnote{If $n \neq 0$, then $\tau$ will not literally be an involution. The fact that $\tau$ is a homotopy involution in the filtered sense then implies $\deg_I(\beta_1^\vee) = \deg_I(\beta_2^\vee) \leq 0$.} In either case, it is straightforward to see that the resulting complex is locally trivial.

\begin{figure}[h!]
\includegraphics[scale = 1.1]{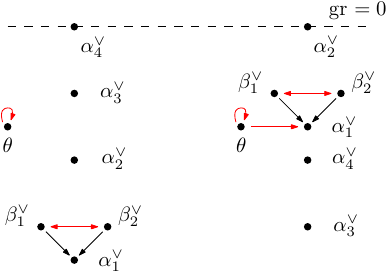}
\caption{Possibilities for the complex of $-Y$. Generators are arranged by height according to their homological grading. The differential is given by the black arrows; if no black arrow is drawn, the differential on a given generator is zero. The action of $\tau$ is given by the (sum of) the red arrows; if no red arrow is drawn, the action of $\tau$ on a given generator is the identity.}\label{fig:7.2}
\end{figure}

\subsection{Brieskorn homology spheres}\label{sec:brieskorn}
We now discuss the involutive $r_{s}$-invariant for Brieskorn homology spheres $\Sigma(p, q, r)$. In all cases other than $S^3$ or $\Sigma(2, 3, 5)$, this has mapping class group $\mathbb{Z}/2\mathbb{Z}$. (For an exposition of this fact, see \cite[Theorem 7.6]{DHM20}.) The generator of $\textit{MCG}(\Sigma(p, q, r))$ is the involution $\tau$ obtained by viewing $\Sigma(p, q, r)$ as the link of a complex singularity and restricting the complex conjugation action to $\Sigma(p, q, r)$. 

We first explain the importance of this analysis. In \cite{DHM20}, it was shown that the Heegaard-Floer-theoretic invariants associated to a pair $(Y, \tau)$ are monotonic under certain types of equivariant negative-definite cobordisms. To constrain the invariants associated to $(Y_1, \tau_1)$, it thus suffices to produce a cobordism of this kind from $(Y_1, \tau_1)$ to another pair $(Y_2, \tau_2)$ whose invariants are already understood. This strategy formed a crucial technique for establishing many of the examples presented in \cite{DHM20}. In particular, the set of Brieskorn spheres (equipped with the complex conjugation involution) formed an essential class of target pairs $(Y_2, \tau_2)$ in the discussion of \cite{DHM20}, since in \cite{ASA20} the invariants of these manifolds were computed and shown to be appropriately nontrivial.

While in this paper we likewise use equivariant negative-definite cobordisms to compare pairs, a crucial difference is that the involutive instanton-theoretic invariants are trivial for Brieskorn spheres. In particular, the methods of this paper cannot immediately be used to re-establish the examples of \cite{DHM20}.

We have:

\begin{thm}\label{thm:Brieskorn}
Let $Y = \Sigma(p, q, r)$ be a Brieskorn sphere with $h(Y) = 0$ and $\tau$ be the complex conjugation action on $Y$. Then
\[
r_s(Y, \tau) = \infty \quad \text{and} \quad r_s(-Y, \tau) = \infty.
\]
Here, $Y$ is oriented as the link of its corresponding singularity and $h(Y)$ is the instanton Fr\o yshov invariant of $Y$. 
\end{thm}
\begin{proof}
Resolving the singularity gives an equivariant negative-definite manifold with boundary $(Y, \tau)$. Puncturing this manifold gives a cobordism from $S^3$ to $(Y, \tau)$; the first claim then follows from Theorem~\ref{thm:1.1}. To prove the second claim, observe that by \cite[Proposition 3.10]{FS90}, the irreducible instanton chain complex for $-Y$ (with our orientation conventions) is supported in even gradings and has no differential. Since $h(Y) = h(-Y) = 0$, it follows that the $D_2$-map is zero on the chain level. Since there are no other generators in the same grading as the reducible $\theta$, the action of $\tau$ must send $\theta$ to itself. It is then easily seen that $r_s(-Y, \tau)$ is identically $\infty$.
\end{proof}

Readers should contrast the above with \cite[Lemma 7.7]{DHM20}. Note that as discussed in Remark~\ref{rem:1.8}, it is also possible to define the involutive $r_s$-invariant using $\Z$-coefficients. However, the same argument as above shows that the involutive $r_s$-invariant with $\Z$-coefficients is also trivial for Brieskorn spheres with $h(Y) = 0$ . (The analogue of Lemma~\ref{lem:4.26} likewise applies to show that the local equivalence class is then trivial.)


\begin{rem} Our local equivalence formalism roughly corresponds to Donaldson's Theorem A; that is, local maps are induced by equivariant negative-definite cobordisms. We expect that a theory such as \cite[Theorem 4.1]{Sa13} more analogous to Donaldson's Theorem B or C might capture the nontriviality of the local equivalence class of $\Sigma(2,3,7)$ with the complex conjugation action. (Such variants of instanton Floer theory were first developed by Fukaya, Furuta and Ohta.)
\end{rem}

\subsection{Definite bounding} \label{sec:7.3}

We now turn to the proof of Theorem~\ref{thm:1.2}. As shown in Figure~\ref{fig:7.4}, the knot $\ov{9}_{46}$ admits two strong inversions, which we denote by $\tau$ and $\sigma$. Let
\[
Y_i = S^3_{1/i}(\ov{9}_{46}).
\]
Write $(Y_i, \tau)$ and $(Y_i, \sigma)$ to mean these surgered $3$-manifolds equipped with the involutions induced by $\tau$ and $\sigma$, respectively; see \cite[Section 5.1]{DHM20}. Note that $(Y_1, \tau)$ is the Abkulut cork. It will also be useful to point out that Lemma~\ref{cal for Akbulut} holds with $(Y_1, \tau)$ replaced by $(Y_1, \sigma)$. Indeed, $\tau$ and $\sigma$ differ by an involution on $Y_1$ which extends over the contractible $4$-manifold $X_1$ discussed in Section~\ref{sec:6.2}. Hence the cork twist along $\sigma$ gives the same $4$-manifold (up to diffeomorphism) as the cork twist along $\tau$. In particular, the Donaldson invariants of the two resulting cork twists agree. Thus the analysis of Section~\ref{sec:7.1} applies to calculate the action of $\sigma$.

\begin{figure}[h!]
\includegraphics[scale=0.5]{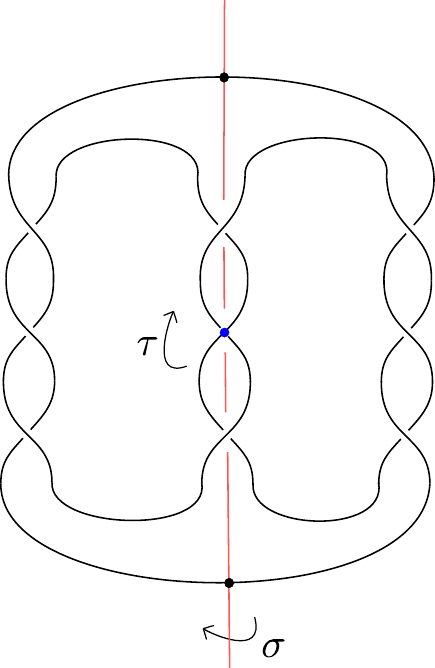}
\caption{Strong inversions on $\ov{9}_{46}$.}\label{fig:7.4}
\end{figure}

The cork $Y$ of Theorem~\ref{thm:1.2} will be constructed as an equivariant connected sum of pairs $(Y_i, \tau)$ and $(Y_i, \sigma)$. For convenience, we use addition to denote the operation of connected sum and negation to denote orientation reversal. For the convenience of the reader, we recall Theorem~\ref{thm:1.2}: \\

\noindent
\textbf{Theorem 1.2}
There exists a cork $Y = \partial W$ such that the boundary involution $\tau$:
\begin{enumerate}
\item Does not extend as a diffeomorphism over any negative-definite $4$-manifold $W^-$ with $H_1(W^-, \Z_2)=0$ bounded by $Y$; and,
\item Does not extend as a homology-fixing or homology-reversing diffeomorphism over any positive-definite $4$-manifold $W^+$ with $H_1(W^+, \Z_2)=0$ bounded by $Y$.
\end{enumerate}

The first part of Theorem~\ref{thm:1.2} is an application of the present paper:

\begin{proof}[Proof of Theorem~\ref{thm:1.2} (1)]
We show that at least one of 
\begin{equation}\label{eq:combA}
(Y, \tau) = - 2(Y_1, \tau) - 2(Y_1, \sigma) + (Y_2, \tau)
\end{equation}
and 
\begin{equation}\label{eq:combB}
(Y, \tau) = -2(Y_1, \tau) - 2(Y_1, \sigma) + (Y_2, \sigma)
\end{equation}
constitutes the desired example. Indeed, note that if $(Y, \tau)$ bounds an equivariant negative-definite manifold $W^-$with $H_1(W^-, \Z_2)=0$, then the inequality of Theorem~\ref{thm:1.1} implies $r_0(Y, \tau) = \infty$. It thus suffices to see that $r_0(Y, \tau) < \infty$ for at least one of \eqref{eq:combA} and \eqref{eq:combB}. 

We know from Lemma~\ref{cal for Akbulut} and the discussion at the beginning of this subsection that $r_0(Y_1, \tau)$ and $r_0(Y_1, \sigma)$ are both finite. Moreover, by Lemma~\ref{lem:knotsA}, we have a simply-connected, negative-definite equivariant cobordism from $(Y_2, \tau)$ to $(Y_1, \tau)$, and similarly from $(Y_2, \sigma)$ to $(Y_1, \sigma)$. Hence
\[
r_0(Y_2, \tau) < r_0(Y_1, \tau) \quad \text{and} \quad r_0(Y_2, \sigma) < r_0(Y_1, \sigma).
\] 
Now suppose that 
\[
r_0(Y_1, \tau) \leq r_0(Y_1, \sigma). 
\]
In this situation we let $(Y, \tau)$ be given by (\ref{eq:combA}). The connected sum inequality for the involutive $r_s$-invariant gives
\[
r_0 (Y_2, \tau)  \geq \min \{ r_0( Y, \tau) , r_0( 2(Y_1, \tau) + 2(Y_1, \sigma))\}. 
\]
Under the supposition that $r_0(Y_1, \tau) \leq r_0(Y_1, \sigma)$, repeatedly applying the connected sum inequality as in the proof of Theorem~\ref{key:sequence} gives
\[
r_0 (Y_2, \tau)  < r_0(Y_1, \tau) = \min \{r_0(Y_1, \tau), r_0(Y_1, \sigma)\} \leq r_0( 2(Y_1, \tau) + 2(Y_1, \sigma)). 
\]
Combining these two inequalities shows that $r_0( Y, \tau) \leq r_0 (Y_2, \tau) < \infty$, as desired. If instead $r_0(Y_1, \sigma) \leq r_0(Y_1, \tau)$, we let $(Y, \tau)$ be given by (\ref{eq:combB}). A similar argument with $r_0 (Y_2, \sigma)$ in place of $r_0 (Y_2, \tau)$ then gives the result.
\end{proof}

We now establish the second part of Theorem~\ref{thm:1.2}. This will require a general understanding of the results of \cite{DHM20}. In \cite[Section 4]{DHM20}, it is shown that an involution $\tau$ on $Y$ induces a homotopy involution (which we also denote by $\tau$) on the Heegaard Floer complex $\CFm(Y)$. The pair $(\CFm(Y), \tau)$ defines an equivalence class in the \textit{local equivalence group} $\mathfrak{I}$ introduced by Hendricks-Manolescu-Zemke \cite[Secion 8.3]{HMZ}. We refer to the local equivalence class of $(\CFm(Y), \tau)$ as the \textit{$\tau$-class} of $(Y, \tau)$, even in the case that the involution on $Y$ is denoted by something other than $\tau$, such as $\sigma$.

The most salient feature of $\mathfrak{I}$ is that it admits a partial order which restricts the existence of equivariant negative-definite cobordisms, as follows. Suppose that $(Y_1, \tau_1)$ and $(Y_2, \tau_2)$ are two homology spheres equipped with involutions and let $(W, \wt{\tau})$ be an equivariant negative-definite cobordism between them with $b_1(W) = 0$. Suppose that there is a $\spinc$-structure $\s$ on $W$ such that the Heegaard Floer grading shift $\Delta_{W, \s}$ (associated to the cobordism map $F_{W, \s}$) is zero and $\wt{\tau}_*(\s) = \s$. Then
\[
\CFm(Y_1, \tau_1) \leq \CFm(Y_2, \tau_2).
\]
We call such a cobordism a \textit{$\spinc$-fixing cobordism}. Thus, if $(\CFm(Y), \tau) > 0$, then there does not exist any $\spinc$-fixing, equivariant negative-definite manifold with boundary $(Y, \tau)$. Here, we write $0$ to mean the local equivalence class of the trivial complex.

A crucial insight of \cite{DHM20} was that the action of $\tau$ may also be combined with the involution $\iota$ defined by Hendricks and Manolescu \cite{HM}. The composition $\iota \tau$ also constitutes a homotopy involution on $\CFm(Y)$ and hence defines another class $(\CFm(Y), \iota \tau)$ in the local equivalence group. We refer to this class as the \textit{$\iota \tau$-class} of $(Y, \tau)$. Crucially, this class has slightly different functoriality properties with respect to the partial order. In particular, let $(W, \wt{\tau})$ be a negative-definite equivariant cobordism as before, but with a $\spinc$-structure $\s$ such that the Heegaard Floer grading shift $\Delta_{W, \s}$ is zero and $\wt{\tau}_*(\s) = \bar{\s}$. Then
\[
\CFm(Y_1, \iota_1\tau_1) \leq \CFm(Y_2, \iota_2\tau_2).
\]
We call such a cobordism a \textit{$\spinc$-reversing cobordism}. Thus, if $(\CFm(Y), \iota \tau) > 0$, then there does not exist any $\spinc$-reversing, equivariant negative-definite manifold with boundary $(Y, \tau)$.

The fact that the invariants of \cite{DHM20} are functorial only in the presence of a $\spinc$-fixing or $\spinc$-reversing cobordism leads to the conditions on $\wt{\tau}$ in the second part of Theorem~\ref{thm:1.2}:



\begin{proof}[Proof of Theorem~\ref{thm:1.2} (2)]
We show that a positive-definite $W^+$ as in the statement of Theorem~\ref{thm:1.2} cannot exist for either of the possibilities (\ref{eq:combA}) or (\ref{eq:combB}) given in the proof of the first part of Theorem~\ref{thm:1.2}.

Suppose for the sake of contradiction that such a $W^+$ did exist. Reversing orientation gives a negative-definite cobordism $\smash{-W^+}$ from $S^3$ to $-Y$ (or, equivalently, a negative-definite cobordism from $Y$ to $S^3$). Since $\pm Y$ bounds a homology ball, it has $d$-invariant zero. It follows from the argument of \cite[Section 9.1]{ozsvath2003absolutely} that any definite manifold with boundary $\pm Y$ must have diagonalizable intersection form. This easily implies the existence of a $\spinc$-structure $\s$ such that the Heegaard Floer grading shift $\Delta_{-W^+, \s}$ is zero. If $\wt{\tau}$ is homology-fixing, as in the statement of the theorem, then it is not hard to see that $\smash{-W^+}$ is a $\spinc$-fixing cobordism as defined previously. If $\wt{\tau}$ is homology-reversing, then $-W^+$ will be a $\spinc$-reversing cobordism.\footnote{The second part of Theorem~\ref{thm:1.2} can thus actually be strengthened to exclude all $\spinc$-fixing and $\spinc$-reversing equivariant positive-definite cobordisms; we have suppressed this for the sake of clarity of the theorem statement.} To produce a contradiction, it thus suffices to prove
\[
(\CFm(Y), \tau) > 0 \quad \text{and} \quad (\CFm(Y), \iota \tau) > 0
\]
for both (\ref{eq:combA}) and (\ref{eq:combB}), as reversing orientation then gives the desired result.

We begin by understanding the Heegaard Floer homologies of $Y_1$ and $Y_2$. The Heegaard Floer homology of $Y_1$ is displayed on the left in Figure~\ref{fig:7.5}. The action of $\tau$ was computed in \cite[Lemma 7.5]{DHM20} up to a change-of-basis, this is likewise shown on the left in Figure~\ref{fig:7.1}. On the right in Figure~\ref{fig:7.5}, we have displayed the Heegaard Floer homology of $Y_2$, which may be calculated via the rational surgery formula \cite{OS11}.
\begin{figure}[h!]
\includegraphics[scale = 1]{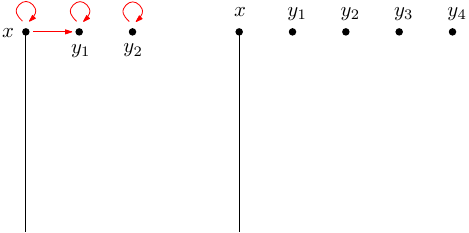}
\caption{The Heegaard Floer homologies $\HFm(Y_1)$ (left) and $\HFm(Y_2)$ (right). The action of $\tau$ is displayed on the left as the sum of red arrows.}\label{fig:7.5}
\end{figure}

We first analyze the actions of $\tau$ and $\sigma$ on $Y_1$. The following was established \cite[Lemma 7.5]{DHM20}:
\begin{equation}\label{eq:y1A}
(\CFm(Y_1), \tau) < 0 \quad \text{and} \quad (\CFm(Y_1), \sigma) = 0
\end{equation}
and
\begin{equation}\label{eq:y1B}
(\CFm(Y_1), \iota \tau) = 0 \quad \text{and} \quad (\CFm(Y_1), \iota \sigma) < 0.
\end{equation}
In fact, we have $(\CFm(Y_1), \tau) = (\CFm(Y_1), \iota\sigma)$. 

The actions of $\tau$ and $\sigma$ on $Y_2$ are not so immediate. However, it is not hard to constrain their possibilities up to local equivalence. Note that if $C_1$ and $C_2$ are two Heegaard Floer complexes whose homologies are supported in even gradings, then chain maps between $C_1$ and $C_2$ (up to homotopy) are determined by their maps on homology. (See for example the argument of \cite[Lemma 4.4]{dai2019involutive}.) Let $\omega$ be any involution on $\HFm(Y_2)$. There are two possibilities:
\begin{itemize}
\item First suppose that $\omega$ fixes the $U$-nontorsion generator $x$. Note that $\omega y_i$ cannot be supported by $x$ for any $i$, since $\omega$ cannot map a $U$-torsion element of homology to a $U$-nontorsion element of homology. Then $(\CFm(Y_2), \omega)$ is easily seen to be locally trivial.
\item Now suppose that $\omega$ does not fix $x$. Perform a change-of-basis on $\HFm(Y_2)$ so that 
\[
y_1 = x + \omega x
\]
and $y_2$, $y_3$, and $y_4$ are still $U$-torsion. Then $(\CFm(Y_1), \tau)$ admits an equivariant local map into $(\CFm(Y_2), \omega)$ by sending $x$ and $y_1$ in $\HFm(Y_1)$ to $x$ and $y_1$ in $\HFm(Y_2)$. (Up to local equivalence, $y_2$ may be discarded.)
\end{itemize}
Thus, in either of the above two cases, we see that
\begin{equation}\label{eq:y1C}
(\CFm(Y_1), \tau) = (\CFm(Y_1), \iota\sigma) \leq (\CFm(Y_2), \omega)
\end{equation}
for any involution $\omega$ on $\HFm(Y_2)$. In particular, we may take $\omega \in \{\tau, \sigma, \iota\tau, \iota\sigma\}$.


Now consider the complex of (\ref{eq:combA}). The $\tau$-class of (\ref{eq:combA}) is given by
\begin{align*}
- 2(\CFm(Y_1), \tau) - 2(\CFm(Y_1), \sigma) + (\CFm(Y_2), \tau) &= -2(\CFm(Y_1), \tau) +  (\CFm(Y_2), \tau) \\
& \geq -(\CFm(Y_1), \tau) \\
&> 0
\end{align*}
where the first equality follows from (\ref{eq:y1A}), the second inequality follows from (\ref{eq:y1C}), and the final inequality follows from (\ref{eq:y1A}). Similarly, the $\iota \tau$-class of (\ref{eq:combA}) is given by
\begin{align*}
- 2(\CFm(Y_1), \iota\tau) - 2(\CFm(Y_1), \iota \sigma) + (\CFm(Y_2), \iota \tau) &= -2(\CFm(Y_1), \iota \sigma) +  (\CFm(Y_2), \iota \tau) \\
& \geq -(\CFm(Y_1), \iota \sigma) \\
&> 0
\end{align*}
where the first equality follows from (\ref{eq:y1B}), the second inequality follows from (\ref{eq:y1C}), and the final inequality follows from (\ref{eq:y1B}). Thus, both the $\tau$-class and the $\iota\tau$-class of (\ref{eq:combA}) are strictly greater than zero. A similar computation holds for (\ref{eq:combB}). This completes the proof.
\end{proof}

Having established Theorem~\ref{thm:1.2}, the proof of Corollary~\ref{cor:1.3} is immediate: 

\begin{proof}[Proof of Corollary~\ref{cor:1.3}]
Consider the cork obtained from Theorem~\ref{thm:1.2}. The first part of Theorem~\ref{thm:1.2} shows that there is no extension of $\tau$ over any negative-definite $X$ as in the corollary statement. On the other hand, any extension of $\tau$ over a positive-definite such $X$ would clearly be either homology-fixing or homology-reversing. Hence such an extension is ruled out by the second part of Theorem~\ref{thm:1.2}. Since $Y$ is a homology sphere, these are the only two possibilities for $X$. 
\end{proof}

\begin{rem}
In principle, it may be possible to replicate Corollary~\ref{cor:1.3} using the methods of \cite{DHM20}. It suffices to find an example of a cork $(Y, \tau)$ such that 
\[
\un{d}_\tau(Y) < 0 < \ov{d}_\tau(Y) \quad \text{and} \quad \un{d}_{\iota \tau}(Y) < 0 < \ov{d}_{\iota \tau}(Y);
\]
see \cite[Remark 4.5]{DHM20}. In the involutive Heegaard Floer setting, such algebraic examples can be constructed through connected sums of Brieskorn spheres of varying sign; see \cite[Corollary 1.7]{DaiStoff19}. Although the corks considered in \cite[Theorem 1.13]{DHM20} are conjecturally similar in behavior to these, at present we do not have the computational tools necessary to verify this. The authors do not expect the specific examples presented in this paper can be re-established using any of the methods of  \cite{LRS18, ASA20, DHM20, KMT23A}.
\end{rem}

In the case that $X$ is of the form $W' \# n \CP$ or $W' \# \oCP$ for a homology ball $W'$, it turns out that we can remove the constraint on the action of $\wt{\tau}$ in Theorem~\ref{thm:1.2}. This leads to the proof of Corollary~\ref{cor:cp2}:

\begin{proof}[Proof of Corollary~\ref{cor:cp2}]
Let $(Y, \tau)$ be any pair for which $r_0(Y, \tau) < \infty$ and $(\CFm(Y), \tau) > 0$. For instance, we may consider the strong cork $(Y, \tau)$ established in the proof of Corollary~\ref{cor:1.3}. (The additional property that $(\CFm(Y), \iota\tau) > 0$ will not be needed; simpler examples are thus possible.) The same argument as before shows $\tau$ does not extend as a diffeomorphism over any negative-definite $X$ with $H_1(X, \Z_2) = 0$ and does not extend as a homology-fixing diffeomorphism over any positive-definite $X$ with $H_1(X, \Z_2) = 0$.

Clearly, the first claim shows $\tau$ does not extend over any $W' \# n \oCP$ as in the corollary statement. Now suppose that there was an extension $\wt{\tau}$ over some $W' \# n \CP$. It is not difficult to see that any automorphism of the intersection form of $n \CP$ can be realized by a self-diffeomorphism of $n \CP$; see for example \cite[Section 2]{Ba23}. After isotopy, we may assume that such a self-diffeomorphism fixes a ball. It follows that by postcomposing $\wt{\tau}$ with an appropriate self-diffeomorphism of (punctured) $n \CP$, we obtain a new extension $\wt{\tau}'$ of $\tau$ which acts as the identity on second homology. This contradicts the second claim in the previous paragraph.
\end{proof}

\begin{rem}\label{rem:7.4}
If the contractible manifold $W$ is fixed in Corollary~\ref{cor:cp2}, then it is not difficult to give several families of corks which survive definite stabilization of either sign. One method (following the formalism of \cite{DHM20}) is as follows: let $W$ be any cork for which the Heegaard Floer relative invariant $F_W(1)$ is not fixed by the action of $\tau$. It is not hard to choose $W$ so that $F_W(1)$ is also not fixed by $\iota \tau$.\footnote{This follows from computations of \cite{DHM20}. As an explicit example, we may take $W$ to be boundary connected sum of two copies of the Akbulut-Mazur cork with the involution $\tau \# \sigma$.} Then the blow-up formula for the Heegaard Floer cobordism maps shows that $\tau$ does not extend over $W \# \oCP$. If we now take any two such corks $W_1$ and $W_2$, then the tensor product formula easily establishes the same for $W_1 \# -W_2$ and $W_2 \# - W_1$; hence $W_1 \# -W_2$ is a cork that survives stabilization by both $\CP$ and $\oCP$ (at least in the case $n = 1$). Many other arguments are possible (see also the techniques of \cite{Y19}); the authors thank Anubhav Mukherjee and Kouichi Yasui for related discussions.
\end{rem}


Finally, we list an additional motivation for Theorem~\ref{thm:1.2}. If $K$ is a strongly invertible or periodic knot with involution $\tau$, then any $1/n$-surgery on $K$ inherits an involution coming from $\tau$. It is natural to ask whether every symmetry on an integer homology sphere arises in this manner: 

\begin{ques}\label{qn:surgery}
Does there exist a pair $(Y, \tau)$ such that $Y$ is surgery on a knot, but $(Y, \tau)$ does not arise via surgery on any equivariant knot?
\end{ques}

The question of whether a manifold is Dehn surgery on a knot is well-studied; see for example \cite{HKL16} for a Floer-theoretic obstruction. Question~\ref{qn:surgery} may be viewed as the natural equivariant version of this problem. The relation between Question~\ref{qn:surgery} and Theorem~\ref{thm:1.2} is given by Lemma~\ref{lem:knotsB}. This shows that Theorem~\ref{thm:1.2} can be used to produce pairs $(Y, \tau)$ which do not arise via surgery on any equivariant knot. 

Unfortunately, it is not difficult to check that the example given in the proof of Theorem~\ref{thm:1.2} does not even arise as (nonequivariant) Dehn surgery on a knot. The obstruction provided by Theorem~\ref{thm:1.2} actually applies to any pair in the equivariant homology cobordism class of $(Y, \tau)$; the authors do not know if this equivariant homology cobordism class contains a representative which arises as (nonequivariant) Dehn surgery on a knot.

\subsection{Proof of Theorems~\ref{thm:1.4} and \ref{thm:1.5}} \label{sec:7.4}
We now prove Theorems~\ref{thm:1.4} and \ref{thm:1.5}. We begin with a general method for establishing the linear independence of a sequence of equivariant homology spheres. For convenience, we use addition to denote the operation of connected sum and negation to denote orientation reversal throughout.

 
\begin{thm}\label{key:sequence}
Let $\{(Y_i, \tau_i)\}_{i = 1}^\infty$ be a sequence of oriented integer homology $3$-spheres equipped with orientation-preserving involutions $\tau_i$. Assume the fixed-point set of each $\tau_i$ is a copy of $S^1$, so that the equivariant connected sum operation is well-defined. Suppose that:
\begin{enumerate}
\item \label{it:seq1} $r_0(Y_1, \tau_1) > r_0( Y_2, \tau_2 ) > \cdots > r_0(Y_i, \tau_i) > \cdots$;
\item \label{it:seq2} $r_0(Y_1, \tau_1) < \infty$; and,
\item \label{it:seq3} $r_0(- Y_i, \tau_i) = \infty$ for each $i$.
\end{enumerate}
Then any nontrivial linear combination of elements from $\{(Y_i, \tau_i)\}_{i = 1}^\infty$ yields a strong cork. 
\end{thm}

\begin{proof}
We follow the proof of \cite[Corollary 5.6]{NST19}, replacing the original $r_s$-invariant with the involutive $r_s$-invariant. Suppose that we had a linear combination
\[
\sum_i n_i (Y_i, \tau_i) = 0.
\]
Let $k$ be the maximal index for which $n_k \neq 0$. Without loss of generality, suppose $n_k > 0$. Then
\[
(Y_k, \tau_k) = - \sum_{i = 1}^{k-1} n_i(Y_i, \tau_i) - (n_k - 1) (Y_k, \tau_k).
\]
Consider the involutive $r_0$-invariant of the right-hand side. By repeatedly applying the connected sum inequality and utilizing (\ref{it:seq1}) and (\ref{it:seq3}), we see that this is greater than or equal to $r_0(Y_{k-1}, \tau_{k-1})$. (Indeed, it is greater than or equal to $r_0(Y_i, \tau_i)$ for the largest value of $i < k$ such that $n_i < 0$.) Considering the involutive $r_0$-invariant of the left-hand side then contradicts (\ref{it:seq1}), given that this is finite by (\ref{it:seq2}).
\end{proof}

Theorem~\ref{thm:1.4} follows quickly from Theorem~\ref{key:sequence}:

\begin{proof}[Proof of \cref{thm:1.4}]
Denote by
\[
Y_n = S^3_{1/n}(K).
\]
It suffices to prove that $\{(Y_n, \tau)\}_{n = 1}^\infty$ satisfies the assumptions of \cref{key:sequence}. The first and third conditions are immediate from Theorem~\ref{thm:1.1} (combined with Lemma~\ref{lem:knotsA}) and Theorem~\ref{thm:definitebounding}. If $K$ is the first knot $K_0$ in the family displayed in Figure~\ref{fig:1.1}, then the second condition is just Lemma~\ref{cal for Akbulut}. If $K$ is a different element in this family, then we have an equivariant negative-definite cobordism from $S^3_{+1}(K)$ to $S^3_{+1}(K_0)$, as displayed in Figure~\ref{fig:7.3}. Hence Theorem~\ref{thm:1.1} again shows that $r_0(Y_n(K), \tau)$ is finite. Applying \cref{key:sequence} completes the proof. 
\end{proof}

\begin{figure}[h!]
\includegraphics[scale=0.95]{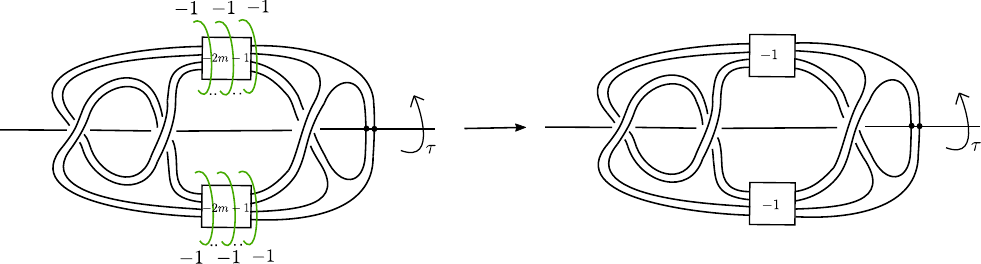} 
\caption{An equivariant negative-definite cobordism from $S^3_{+1}(K)$ to $S^3_{+1}(K_0)$ given by attaching the $-1$-framed green 2-handles.}\label{fig:7.3}
\end{figure}

Finally, we have:

\begin{proof}[Proof of Theorem~\ref{thm:1.5}]
Let $K$ be the any of the strongly invertible knots in Figure~\ref{fig:1.1}. We have already checked in the proof of Theorem~\ref{thm:1.4} that 
\[
r_0(S^3_{1/n}(K), \tau) < \infty
\]
for any $n \in \N$. A simply-connected, equivariant negative-definite cobordism from $\smash{(S^3_{1/n}(K), \tau)}$ to itself would violate the strict inequality in Theorem~\ref{thm:1.1}. (In the positive-definite case, we turn the cobordism around.) This completes the proof. 
\end{proof}

Another application of Theorem~\ref{thm:1.5} is as follows: let $K$ be any of the knots in Theorem~\ref{thm:1.5}. Using the Montesinos trick, write 
\[
S^3_{1/n}(K) = \Sigma_2(K'_n)
\]
for some knot $K'_n$ in $S^3$. Here, $\Sigma_2(K'_n)$ is the branched double cover of $K'_n$ and the induced symmetry $\tau$ on the left is identified with the branching involution on the right. If $C$ is any concordance from $K'_n$ to $K'_n$, then $\Sigma_2(C)$ constitutes an equivariant homology cobordism from $\smash{S^3_{1/n}(K)}$ to itself. Moreover, if we had
\[
\pi_1([0,1]\times S^3 \setminus C ) = \Z,
\]
then this homology cobordism would be simply connected, which is ruled out by Theorem~\ref{thm:1.5}. Hence the methods of this paper can be used to rule out self-concordances with fundamental group $\Z$. Such results can already be obtained using singular knot instanton Floer theory with a general holonomy parameter \cite{DISST22}, at least in the case that $K_n'$ is not algebraically slice. 

\subsection{Nonorientable surfaces}\label{sec:7.5}

We now prove Theorems~\ref{thm:1.6} and \ref{thm:1.7}. These will follow immediately from the discussion of Section~\ref{sec:2.3} and our prior results regarding equivariant bounding. For the benefit of the reader, we first briefly discuss some details regarding the geography question for nonorientable surfaces. Recall that we have the Gordon-Litherland bound \cite{GL78}:
\[
\left | \sigma(K) - e(F)/2 \right | \leq h(F).
\]
This corresponds to the region $\Sigma$ displayed below on the left in Figure~\ref{fig:nonorientable}. In general (ignoring issues of parity), the set of realizable pairs is a subset of $\Sigma$ consisting of the union of sets of the form $\{|a - e/2| \leq h - b\}$, as shown in Figure~\ref{fig:nonorientable}. (To see this, note that we may take the connected sum of any surface with $\RP$, which increases $h$ by one and changes $e$ by $\pm 2$.) Extremal surfaces correspond to the points on the boundary of $\Sigma$.

\begin{figure}[h!]
\includegraphics[scale = 0.83]{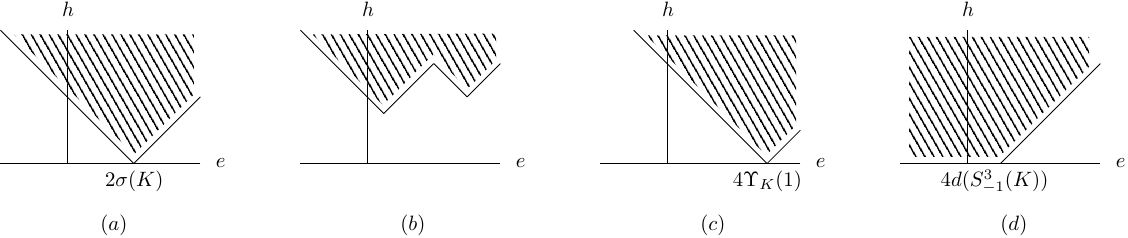}
\caption{From left-to-right: $(a)$ the region due to the Gordon-Litherland bound; $(b)$ a schematic example of the set of realizable pairs; $(c)$ the region due to the Ozv\'ath-Szab\'o-Stipsicz inequality; and $(d)$ the region due to Batson's bound.}\label{fig:nonorientable}
\end{figure}

Several authors have used Floer theory to derive additional constraints on the set of realizable pairs. The most well-known of these are the following: in \cite{Ba14}, Batson showed that 
\begin{equation}\label{eq:constraint1}
e(F)/2 - 2d(S^3_{-1}(K)) \leq h(F),
\end{equation}
while in \cite{OSSunoriented}, Ozsv\'ath-Stipsicz-Szab\'o proved
\begin{equation}\label{eq:constraint2}
\left | 2\Upsilon_K(1) - e(F)/2 \right | \leq h(F).
\end{equation}
These are schematically displayed in Figure~\ref{fig:nonorientable}. Either constraint can be used (in conjunction with the Gordon-Litherland bound) to prove that the nonorientable slice genus may be arbitrarily large and (additionally) rule out a subset of extremal surfaces. Note, however, that neither inequality is individually capable of ruling out the set of \textit{all} extremal surfaces. Moreover, using the fact that
\[
2V_0(-K) \geq \Upsilon_K(1)
\]
as shown in \cite[Proposition 2.54]{Sa23}, it is possible to show that such a result cannot be obtained by even by combining (\ref{eq:constraint1}) and (\ref{eq:constraint2}). The authors thank Kouki Sato for bringing this to their attention.

\begin{proof}[Proof of Theorem~\ref{thm:1.6}]
Let $Y$ be the cork constructed in Theorem~\ref{thm:1.2}. Using the Montesinos trick, we may write $Y$ as the branched double cover of some knot $J$ in such a way so that the branching involution over $J$ corresponds to the relevant involution on $Y$. Explicitly, this means that $J$ is either
\[
-2 A_1 \# - 2B_1 \# A_2 \quad \text{or} \quad -2A_1 \# -2B_1 \# B_2,
\]
with $A_n$ and $B_n$ as in Figure~\ref{fig:1.2}. Suppose that $J$ bounded an extremal surface $F$. Then $\Sigma_2(J) = \partial \Sigma_2(F)$. As discussed in Section~\ref{sec:2.3}, $\Sigma_2(F)$ is a homology-reversing equivariant definite manifold with $H_1(\Sigma(F), \Z_2) = 0$. This contradicts Theorem~\ref{thm:1.2}. 
\end{proof}

\begin{proof}[Proof of Theorem~\ref{thm:1.7}]
Consider any $K$ and $n$ as in Theorem~\ref{thm:1.5}. Consider the cork
\[
Y = S^3_{1/n}(K) \# - S^3_{1/n}(K).
\]
Using the Montesinos trick, we may write $Y$ as the branched double cover of some knot $J$ in such a way so that the branching involution over $J$ corresponds to the relevant involution on $Y$. For instance, when $K = \ov{9}_{46}$, we have
\[
J = A_n \# -A_n,
\]
with $A_n$ as in Figure~\ref{fig:1.1}. Clearly, $J$ is slice. Taking the connected sum of any slice disk with $\RP$ gives a trivial extremal surface. Suppose that $J$ bounded an extremal $\Z_2$-surface $F$. Then $\Sigma_2(J) = \partial \Sigma_2(F)$ and it is easily checked that $\Sigma_2(F)$ is simply-connected. As discussed in Section~\ref{sec:2.3}, $\Sigma_2(F)$ is a homology-reversing equivariant definite manifold. This contradicts Theorem~\ref{thm:1.5}. 
\end{proof}

\bibliographystyle{plain}
\bibliography{tex}

\end{document}